\documentclass[11pt]{article}

\usepackage{amsmath,amssymb,amscd,amsthm,makeidx,txfonts,graphicx,url,bm}
\usepackage{a4wide,pstricks,xy}

\usepackage{youngtab}
\numberwithin{equation}{subsection}

\xyoption{all}
\input{xypic}

\newtheorem{theorem}{Theorem}[subsection]
\newtheorem{proposition}[theorem]{Proposition}
\newtheorem{corollary}[theorem]{Corollary}
\newtheorem{conjecture}[theorem]{Conjecture}
\newtheorem{lemma}[theorem]{Lemma}

\theoremstyle{remark}
\newtheorem{remark}[theorem]{Remark}
\newtheorem{example}[theorem]{Example}

\theoremstyle{definition}

\newtheorem{nothing}[theorem]{}

\def\beq{\begin{eqnarray}}
\def\eeq{\end{eqnarray}}
\def\bes{\begin{eqnarray*}}
\def\ees{\end{eqnarray*}}

\def\omhat{{\bm\omega}}

\def\muhat{{\bm \mu}}

\def\lambdahat{{\bm \lambda}}
\def\alphahat{{\bm \alpha}}
\def\nuhat{{\bm \nu}}
\def\oP{\overline{\calP}}

\def\frakp{{\frak p}}
\def\frakl{{\frak l}}
\def\fraku{{\frak u}}

\DeclareMathOperator{\Tr}{Tr}

\def\A{{\bf A}}
\def\bfB{{\bf B}}

\def\C{\mathbb{C}}
\def\M{{\mathcal{M}}}

\def\calN{{\mathcal{N}}}

\def\bbX{\mathbb{X}}

\def\calA{{\mathcal{A}}}
\def\calI{{\mathcal{I}}}
\def\calX{{\mathcal{X}}}
\def\calU{{\mathcal{U}}}
\def\calE{{\mathcal{E}}}

\def\calF{{\mathcal{F}}}
\def\calS{{\mathfrak{S}}}

\def\calP{\mathcal{P}}
\def\calH{\mathcal{H}}

\def\x{\mathbf{x}}

\def\s{\mathbf{s}}

\def\m{\mathbf{m}}

\def\v{\mathbf{v}}

\def\bbP{\mathbb{P}}

\def\nuhat{{\bf \nu}}

\def\N{\mathbb{Z}_{\geq 0}}

\def\F{\mathbb{F}}
\def\Q{\mathbb{Q}}
\def\T{\mathbb{T}}
\def\Tr{\rm Tr}

\def\t{\mathfrak{t}}
\def\calC{{\mathcal C}}
\def\calM{{\mathcal M}}
\def\calO{{\mathcal O}}

\def\bfC{\mathbf{C}}

\def\Z{\mathbb{Z}}

\def\gl{{\mathfrak g\mathfrak l}}

\def\a{\mathfrak{a}}
\def\b{\mathfrak{b}}
\def\l{\mathfrak{l}}

\def\bor{{\rm bor}}

\newcommand{\nc}{\newcommand}

\def\bfS{{\bf S}}

\nc{\op}[1]{\mathop{\mathchoice{\mbox{\rm #1}}{\mbox{\rm #1}}
{\mbox{\rm \scriptsize #1}}{\mbox{\rm \tiny #1}}}\nolimits}
\nc{\al}{\alpha}

\nc{\ep}{\varepsilon} \nc{\ga}{\gamma} \nc{\Ga}{\Gamma}
\nc{\la}{\lambda} \nc{\La}{\Lambda} \nc{\si}{\sigma}
\nc{\Sig}{{\Gamma}} \nc{\Om}{\Omega} \nc{\om}{\omega}
\def\frakm{\mathfrak{m}}
\def\nuhat{{\bf \nu}}

\nc{\SL}{{\rm SL}} \nc{\GL}{{\rm GL}} \nc{\PGL}{{\rm PGL}}
\nc{\G}{{\rm G}}

\nc{\cpt}{{\op{cpt}}} \nc{\Dol}{{\op{Dol}}} \nc{\DR}{{\op{DR}}}
\nc{\B}{{\op{B}}} \nc{\Triv}{\op{Triv}} \nc{\Hod}{{\op{Hod}}}
\nc{\Log}{{\op{Log}}} \nc{\Exp}{{\op{Exp}}} \nc{\Est}{E_{\op{st}}}
\nc{\Hst}{H_{\op{st}}} \nc{\Left}[1]{\hbox{$\left#1\vbox to
  10.5pt{}\right.\nulldelimiterspace=0pt \mathsurround=0pt$}}
\nc{\Right}[1]{\hbox{$\left.\vbox to
  10.5pt{}\right#1\nulldelimiterspace=0pt \mathsurround=0pt$}}
\nc{\LEFT}[1]{\hbox{$\left#1\vbox to
  15.5pt{}\right.\nulldelimiterspace=0pt \mathsurround=0pt$}}
\nc{\RIGHT}[1]{\hbox{$\left.\vbox to
 15.5pt{}\right#1\nulldelimiterspace=0pt \mathsurround=0pt$}}

\nc{\bee}{{\bf E}} \nc{\bphi}{{\bf \Phi}}

\title{Higgs bundles  and indecomposable parabolic bundles over the projective line}

\author{ Emmanuel Letellier \\ {\it
  Universit\'e  Paris Diderot-Paris 7/ IMJ-PRG} \\ {\tt
 letellier@math.univ-paris-diderot.fr} }

\begin{document}

\maketitle

\begin{abstract} In  this paper we count the number of isomorphism classes of geometrically indecomposable quasi-parabolic structures of a given type on a given vector bundle on the projective line over a finite field. We give a conjectural cohomological interpretation for this counting using the moduli space of Higgs fields on the given vector bundle over the complex projective line  with prescribed residues. We prove a certain number of results which bring evidences to the main conjecture. We  detail  the case of rank $2$ vector bundles.

\end{abstract}

\pagestyle{myheadings}

\newpage
\tableofcontents
\newpage

\section{Introduction}

Fix a reduced divisor $D=\b_1+\dots+\b_r$ of $\bbP^1_{\F_q}$ with $\b_i$ of degree $d_i$ ($d_i$ is the degree of the field extension $\F_q(\b_i)/\F_q$). We call a \emph{quasi-parabolic structure} \footnote{In the literature the terminology \emph{parabolic} means quasi-parabolic together with a weight structure. However, unlike in the introduction, we will never use the terminology ``parabolic " in the common sense  (i.e. with weight structure) in the main text of the paper and so from Section \ref{generalities} we will say for short ``parabolic " instead of ``quasi-parabolic " as there will be no possible confusion.} on a rank $n$  vector bundle $\calE$ on $\bbP^1_{\F_q}$ an $r$-tuple $E=(E_{\a_1},\dots,E_{\a_r})$ of infinite non-increasing sequences

$$E_{\b_i}\,:\,\calE(\b_i)=E_{i0}\supseteq E_{i1}\supseteq\cdots \supseteq E_{ir_i}\supseteq\cdots$$ of $\F_q(\b_i)$-vector spaces with only finitely many non-zero terms. A vector bundle equipped with a quasi-parabolic structure is called  a \emph{quasi-parabolic bundle}.

The \emph{type} of a quasi-parabolic structure $E$ is the $r$-tuple $\s=(s_1,\dots,s_r)$  of non-increasing sequences $s_i: s_{i0}=n\geq s_{i1}\geq\cdots\geq s_{il_i}\geq\cdots$  such that ${\rm dim}\, E_{i,j}=s_{i,j}$ for all $i,j$.

A quasi-parabolic vector bundle is \emph{geometrically indecomposable} if it remains indecomposable after extension of scalars from $\F_q$ to $\overline{\F}_q$ and a quasi-parabolic structure $E$ on $\calE$  is (geometrically) indecomposable  if the quasi-parabolic bundle $(\calE,E)$ is (geometrically) indecomposable.

We denote by $\calA_{\s,D}^\calE(q)$ the number of isomorphism classes of geometrically indecomposable quasi-parabolic structures of type $\s$ on $\calE$. 
\bigskip

As a motivation for studying $\calA_{\s,D}^\calE(q)$, consider the total counting 

$$
\calA_{\s,D}^{n,d}(q):=\sum_\calE\calA_{\s,D}^\calE(q)
$$
where the sum is over the isomorphism classes of vector bundles of degree $d$ and rank $n$. This sum is finite by Theorem \ref{finiteness}. It is also conjectured to be independent on $d$ (see Conjecture \ref{independent}). 
\bigskip

When $(\s, d)$ is sufficiently general and assuming that ${\rm deg}(\a_i)=1$ for all $i=1,\dots,r$, the counting $\calA_{\s,D}^{n,d}(q)$  gives the Poincar\'e polynomial of a certain moduli space of stable parabolic meromorphic Higgs bundles.  As Schiffmann told me, this follows from an adaptation to the parabolic case of his strategy  in the non-parabolic case \cite{Schiffmann}  using the work of Lin \cite{lin}. 

Via non-abelian Hodge theory and the Riemann-Hilbert monodromy map, these moduli spaces of stable parabolic meromorphic Higgs bundles are diffeomorphic to certain character varieties with generic semisimple local monodromy of type $\s$ at punctures. The existence of the diffeomorphism arising from non-abelian Hodge theory we want is a particular case of Biquard, Boalch \cite[Theorem 5]{B} \cite{BB} by setting the \emph{irregular types} to be zero. The tame case (i.e. the particular case of \cite{B}\cite{BB} where the irregular types equal zero) was known before but does not seem to be explicitly written in the literature in the form we want. As Boalch explained to me the tame case follows from the work of many authors including Simpson \cite{Simpson}, Konno \cite{Konno} and Nakajima \cite{Nak}.

The above character varieties (studied in \cite{aha}) are of the following form. Let  $\bfC_\s=(C_1,\dots,C_r)$ be a \emph{generic} tuple of semisimple conjugacy classes of $\GL_n(\C)$ of type $\s$, namely for each $i=1,\dots,r$, the multiplicities of the eigenvalues of $C_i$ are given by the successive differences $s_{ij}-s_{i(j+1)}$, then define the affine GIT quotient

$$
\M_{\bfC_\s}:=\{(x_1,\dots,x_r)\in C_1\times\cdots\times C_r\,|\, x_1\cdots x_r=1\}/\!/\GL_n(\C),
$$
where $\GL_n(\C)$ acts diagonally by conjugation. 

Then it follows from the above discussion that for a general $(\s,d)$ and some $\bfC_\s$ as above we have

$$
\calA_{\s,D}^{n,d}(q)=q^{-\frac{1}{2}{\rm dim}\M_{\bfC_\s}}\sum_i{\rm dim}(H_c^{2i}(\M_{\bfC_\s},\C))\, q^i.
$$

A  computable formula for $\calA_{\s,D}^{n,d}(q)$ was not given in the literature. However, using the work of Lin \cite{lin}, it may be possible (at least in theory) to extend Schiffmann's work in the non-parabolic case \cite{Schiffmann} to get a formula for $\calA_{\s,D}^{n,d}(q)$ (not only for $\bbP^1$ but also for higher genus curves). 
\bigskip

One motivation for computing $\calA_{\s,D}^{n,d}(q)$ is to bring evidences for the conjectural formula for the mixed Hodge polynomial of $\M_{\bfC_\s}$ \cite{aha} which formula involves Macdonald polnomials. Schiffmann's type formulas (even in the non-parabolic case where everything is explicit) are not easily comparable with the conjectural formulas in \cite{aha}\cite{HRV}. In this paper we give a formula for the individual quantities $\calA_{\s,D}^\calE(q)$ (see Theorem \ref{MAINtheo2}) using the same kind of technics as in \cite{aha}\cite{HLRV}\cite{ICQV} (we could generalize our work to higher genus curves assuming that $\calE$ is semisimple). The resulting formula  for $\calA_{\s,D}^{n,d}(q)$ looks similar to those in \cite{aha} even if we can not yet prove any precise relation. For instance we see the appearance of Hall-Littlewood polynomials which are specializations of Macdonald polynomials.

\bigskip

 An other motivation for studying $\calA_{\s,D}^{n,d}(q)$ comes from  a conjecture of Drinfled \cite{Dr} on the existence of a Lefschetz trace formula for the counting of certain $l$-adic local systems on smooth projective curves in positive characteristics.  In this direction, Deligne \cite[Theorem 3.5]{Del} related the quantity $\calA_{\s,D}^{n,d}(q)$  with the counting of certain $l$-adic local systems with some constraints on local monodromy. 
 \bigskip
 
We can also see the study of the individual terms $\calA_{\s,D}^\calE(q)$ as a generalisation of Kac's work on quiver representations of star-shaped quivers. Indeed in the special case where $\calE=\calO(a)^n$ and $d_1=\cdots=d_r=1$, the counting $\calA_{\s,D}^\calE(q)$ is given by the evaluation at $q$ of the so-called Kac polynomial which was introduced 35 years ago and  which still very well studied in the literature  (see \S \ref{Kacpoly} and Remark \ref{kacpolynomial} for more details). Also in the case $\calE=\calO(a)^n$ and $d_1=\cdots=d_r=1$, a cohomological interpretation of $\calA_{\s,D}^\calE(q)$ is given in \cite{HLRV} (see also \cite{crawley-etal} when $\s$ is \emph{indivisible}, i.e. $\text{gcd}(s_{i,j})=1$). In this paper we give a conjectural cohomological interpretation (see \S \ref{conjsec} below) of the terms $\calA_{\s,D}^\calE(q)$ for any $\calE$ and arbitrary $d_1,\dots,d_r$. We prove this conjecture for $\calE=\calO(a)^n$ and arbitrary degrees $d_1,\dots,d_r$ by adapting the ideas of \cite{HLRV}. For a general $\calE$, the difficulty for proving the conjecture comes the fact that ${\rm Aut}(\calE)$ is not reductive and so we can not use GIT.

\bigskip

Before stating precisely  the results of this paper let us remark (see Corollary \ref{s'bis}) that the quantity $\calA_{\s,D}^\calE(q)$ depends only on the multi-partition $\nuhat=(\nu^1,\dots,\nu^r)$ of $n$ arising from $\s=(s_1,\dots,s_r)$ where the parts of $\nu^i$ are given by the successive differences $s_{ij}-s_{i(j+1)}$ of the terms of the sequence $s_i$. Therefore in the following we will use the notation $\calA_{\nu,D}^\calE(q)$ instead of $\calA_{\s,D}^\calE(q)$.
\bigskip

\subsection{The main conjecture} \label{conjsec}

 \bigskip

Consider a reduced divisor $\mathcal{D}=\a_1+\cdots+\a_l$ on $\bbP^1_\C$. Choose a \emph{generic} $l$- tuple $\bfS=(S_1,\dots,S_l)$ of semisimple regular adjoint orbits of $\gl_n(\C)$ (see \S \ref{Higgs} for the definition of generic tuples). Given a rank $n$ vector bundle $\calE$ on $\bbP^1_\C$ we define 

$$
X_\bfS^\calE=X_{\bfS,\mathcal{D}}^\calE:=\left.\left\{\varphi: \calE\rightarrow\calE\otimes\Omega_1(\mathcal{D})\,\right|\, {\rm Res}_{\a_i}(\varphi)\in S_i, \text{ for all }i=1,\dots,l\right\}.
$$

The action of ${\rm Aut}(\calE)$ on $X_\bfS^\calE$ induces a free action of ${\rm PAut}(\calE):={\rm Aut}(\calE)/\C^\times$.

We expect that the categorical quotient $\calX_\bfS^\calE$ of $X_\bfS^\calE$ by ${\rm PAut}(\calE)$ exists (in the category of algebraic variety) and that the quotient map is a principal ${\rm PAut}(\calE)$-bundle in the \'etale topology. We also expect that the quotient $\calX_\bfS^\calE$ is non-singular with pure mixed Hodge structure and has polynomial count (see Conjecture \ref{congind}). In particular, conjecturally  it has vanishing odd cohomology.

 This is the case when $\calE=\calO(a)^n$ (see \cite[Theorem 2.2.4]{aha}) as  ${\rm Aut}(\calE)=\GL_n$ and 

$$
X_\bfS^{\calO(a)^n}\simeq \left.\left\{(X_1,\dots,X_l)\in S_1\times\cdots\times S_l\,\right|\, X_1+\cdots+X_l=0\right\}.
$$

For now we may define $\calX_\bfS^\calE$ as the quotient stack $[X_\bfS^\calE/{\rm PAut}(\calE)]$. 

Fix a multi-partition $\muhat=(\mu^1,\dots,\mu^l)\in(\calP_n)^l$ and let $S_\muhat$ be the stabilizer of $\muhat$ in $\frak{S}_l$ acting on $(\calP_n)^l$ by permutation. The group $S_\muhat$ is of the form $\prod_{i=1}^s\frak{S}_{l_i}$ where $l_1,\dots,l_s$ are the multiplicities of the distinct coordinates $\mu^{h_1},\dots,\mu^{h_s}$ of $\muhat$. Define the $\C[(\frak{S}_n)^l]$-module

$$
R_\muhat={\rm ind}_{\frak{S}_{\mu^1}}^{\frak{S}_n}(1)\boxtimes\cdots\boxtimes {\rm Ind}_{\frak{S}_{\mu^l}}^{\frak{S}_n}(1),
$$
where for a partition $\lambda=(\lambda_1,\dots,\lambda_s)$, we put $\frak{S}_\lambda:=\prod_{i=1}^s\frak{S}_{\lambda_i}$ and ${\rm Ind}_{\frak{S}_\lambda}^{\frak{S}_n}(1)$ is the $\C[\frak{S}_n]$-module induced from the trivial $\C[\frak{S}_\lambda]$-module. The action of $(\frak{S}_n)^l$ on $R_\muhat$ extends in the obvious way to an action of $(\frak{S}_n)^l\rtimes S_\muhat$.

Say now that a reduced divisor on $\bbP^1_{\F_q}$ is compatible with $w\in S_\muhat=\prod_{i=1}^s\frak{S}_{l_i}$ if it is of the form 

$$
D_w=\sum_{i=1}^s\sum_{j=1}^{p_i}\a_{i,j},
$$
where for each $i=1,\dots,s$, the degrees $d_{i,1},\dots,d_{i,p_i}$ of $\a_{i,1},\dots,\a_{i,p_i}$  form a partition of $l_i$ that gives the cycle-type decomposition of the coordinate of $w$ in $\frak{S}_{l_i}$. We let $r=\sum_{i=1}^sp_i$ be the total number of closed points of $D_w$ and denote by $\muhat_w\in(\calP_n)^r$ the multi-partition which has coordinate $\mu^{h_i}$ at $\a_{i,j}$. 

The following conjecture is the main conjecture of the paper (see Conjecture \ref{conjgeneral}).

\begin{conjecture} There exists a structure of $\C[(\frak{S}_n)^l\rtimes\frak{S}_l]$-module on $H_c^i(\calX_\bfS^\calE,\C)$ such that for any $\muhat\in(\calP_n)^l$, any $w\in S_\muhat$, any finite field $\F_q$ and any reduced divisor $D_w$ on $\bbP^1_{\F_q}$ compatible with $w$ we have

$$
\calA_{\muhat_w,D_w}^\calE(q)=q^{-\frac{1}{2}{\rm dim}\, \calX_{\bfS}^\calE}\sum_i{\Tr}\left(w\,|\, W_\muhat^{2i}\right)\, q^i,
$$
where $W_\muhat^i:={\rm Hom}_{(\frak{S}_n)^l}\left(R_\muhat,\varepsilon^l\otimes H_c^{2i}(\calX_\bfS^\calE,\C)\right)$, $\varepsilon^l:=\underbrace{\varepsilon\boxtimes\cdots\boxtimes\varepsilon}_l$ with $\varepsilon$ the sign representation of $\frak{S}_n$ and where $S_\muhat$ acts on $W_\muhat^i$ as $(w\cdot f) (v)=w\cdot f(w^{-1}\cdot v)$.
\label{MAINconj}\end{conjecture}
\bigskip

\begin{remark} The expectation that the quantities $\calA_{\muhat,D}^\calE(q)$ are given by polynomials with coefficients in the character ring of products of symmetric groups (see Conjecture \ref{mainconj1} for a precise statement) was suggested to me by Deligne. The above conjecture gives a geometric realization of Deligne's expectation.
\end{remark}

Following ideas of Nakajima \cite{nakajima2}, Crawley-Boevey and van den Bergh \cite{crawley-etal} we can define an action of the group  $\frak{H}:=(\frak{S}_n)^l\rtimes\frak{S}_l$ on the compactly supported cohomology $H_c^i(\calX_{\bfS}^{\calO(a)^n},\C)$ (see \S \ref{theindivcase} for more details) and we prove that the above conjecture is true when $\calE=\calO(a)^n$ (see Theorem \ref{theotriv}). The difficulty is that the group $\frak{H}$ does not act on the variety itself but only on its cohomology. The same strategy for constructing the action of $\frak{H}$ on $H_c^i(\calX_\bfS^\calE,\C)$ should work also for a general $\calE$ but some theoritical arguments are more delicate due to the fact that ${\rm Aut}(\calE)$ is not a reductive group for instance.

When $\muhat=(\mu^1,\dots,\mu^l)\in(\calP_n)^l$ is \emph{indivisible} (i.e. when the gcd of all parts of the partitions $\mu^1,\dots,\mu^l$ equals $1$)  we have a more direct conjectural geometrical interpretation of $\calA_{\muhat_w,D_w}^\calE(q)$ (see Conjecture \ref{congind}).

\subsection{Main results}\label{section2}

Let $D=\b_1+\cdots+\b_r$ be a reduced divisor on $\bbP^1_{\F_q}$ with $d_i={\rm deg}(\b_i)$ and let $\muhat=(\mu^1,\dots,\mu^r)\in(\calP_n)^r$. We fix a vector bundle $\calE$ on $\bbP^1_{\F_q}$ of rank $n$.

We first start with a representation theoritical interpretation of $\calA_{\muhat,D}^\calE(q)$. 

For each $i$, consider the subgroup $L_i=\GL_{\mu^i_1}\times\cdots\times \GL_{\mu^i_{s_i}}$ of $\GL_n$ where $(\mu^i_1,\dots,\mu^i_{s_i})=\mu^i$. For $i=1,\dots, r$, put $L_i(q^d):=L_i(\F_q(\b_i))$ and denote by $R_{L_i(q^{d_i})}^{\GL_n(q^{d_i})}$ the Harish-Chandra induction which is a $\Z$-linear map from the ring of complex characters of $L_i(q^{d_i})$ to the ring of complex characters of $\GL_n(q^{d_i})$. 

We prove the following result (see Theorem \ref{theo}).

\begin{theorem}For any \emph{generic} $r$-tuple $(\alpha_1,\dots,\alpha_r)$ of linear characters of $L_1(q^{d_1}),\dots,L_r(q^{d_r})$ we have

$$
\calA_{\muhat,D}^\calE(q)=\frac{1}{|{\rm Aut}(\calE)|}\sum_{f\in{\rm Aut}(\calE)}\prod_{i=1}^rR_{L_i(q^{d_i})}^{\GL_n(q^{d_i})}(\alpha_i)(f(\b_i)),
$$
where $f\mapsto f(\b_i)$ is the evaluation map ${\rm Aut}(\calE)\rightarrow \GL_n(q^{d_i})$.
\label{MAINtheo1}\end{theorem}

We use this result to prove our second result. 

Write 

$$
\calE=\bigoplus_{i=1}^f\calO(b_i)^{m_i},
$$
with $b_1>b_2>\cdots>b_f$, and  for $\v=(v_1,\dots,v_r)\in\N^r$ put 

$$
\calE^\v:=\bigoplus_{i=1}^f\calO(b_i)^{v_i}
$$
and $|\v|=\sum_{i=1}^rv_i$.

For  a multipartition $\lambdahat=(\lambda^1,\dots,\lambda^r)$ and $\v=(v_1,\dots,v_r)$ such that $|\lambda^1|=\cdots=|\lambda^r|=|\v|$ define 

$$
\calH_{\lambdahat,D}^\v(q):=\frac{\#\{h\in{\rm Aut}(\calE^\v)\,|\, h(\b_i)\in C_{\lambda^i}(q^{d_i}) \text{ for all }i=1,\dots,r\}}{|{\rm Aut}(\calE^\v)|},
$$
where $C_{\lambda^i}(q^{d_i})$ denotes the unipotent conjugacy class of $\GL_n(q^{d_i})$ of Jordan type $\lambda^i$ (i.e. the size of the Jordan blocks are given by the parts of $\lambda^i$).

Let $Y=\{Y_1,\dots,Y_r\}$ be an other set of commuting variables. Then for any integer $s>0$ and any multipartition $\lambdahat=(\lambda^1,\dots,\lambda^r)\in(\calP_s)^r$, we define 

$$
\calH_{\lambdahat,D}(q,Y):=\sum_{\v, |\v|=s}\calH_{\lambdahat,D}^\v(q)\, Y^\v,
$$
where $Y^\v:=Y_1^{v_1}\cdots Y^{v_r}_r$ if $\v=(v_1,\dots,v_r)$.

For each $i=1,\dots,r$, let $\x_i=\{x_{i,1},x_{i,2},\dots\}$ be an infinite set  of commuting variables and denote by $\Lambda(\x_i)$ the ring of symmetric functions in the variables of $\x_i$.

Put 

$$
\Omega_D(q)=\Omega_D(\x_1,\dots,\x_r; q,Y):=1+\sum_{s>0}\sum_{\lambdahat=(\lambda^1,\dots,\lambda^r)\in(\calP_s)^r}\calH_{\lambdahat,D}(q,Y)\prod_{i=1}^r\tilde{H}_{\lambda^i}(\x_i; q^{d_i}),
$$
where for a partition $\lambda$ of $s$, $\tilde{H}_\lambda(\x_i;q)$ is the modified Hall-Littlewood symmetric function which is of homogeneous degree $s$ (see \S \ref{Hall}).

For a partition $\mu$ denote by $h_\mu(\x_i)$ the corresponding complete symmetric function. 

We prove the following theorem (see Theorem \ref{theo2}).

\begin{theorem}We have

$$
\calA_{\muhat,D}^\calE(q)=\left\langle {\rm Coeff}\,_{Y^\m}\left[(q-1)\Log\,\Omega_D(q)\right],h_{\mu^1}(\x_1)\cdots h_{\mu^r}(\x_r)\right\rangle,
$$
where $\m=(m_1,\dots,m_r)$, $\Log$ is the pletystic logarithm (see \S\ref{relation} and \S\ref{Re-writting}) and $\left\langle\,,\,\right\rangle$ is the usual Hall pairing on symmetric functions.
\label{MAINtheo2}\end{theorem}

\begin{remark}The important feature of this result, which is crucial for us, is the fact that it separates the quasi-parabolic structure (on the right hand side of the pairing) from the data coming from the divisor $D$ and the vector bundle $\calE$ (on the left hand side). The strategy for proving the main conjecture is to decompose (as shown in \S \ref{strategy}) the complete symmetric functions in the basis of power symmetric functions from which we see the appearance of the $\frak{S}_n$-modules ${\rm Ind}_{\frak{S}_{\mu^i}}^{\frak{S}_n}(1)$ and  $\calX_\bfS^\calE$ in Conjecture \ref{MAINconj}.
\end{remark}
\bigskip

The problem of computing explicitly $\calA_{\muhat,D}^\calE(q)$ using Theorem \ref{MAINtheo2} is discussed in \S \ref{explicitcomp}. In the case where $\calE=\calO(b_1)^{m_1}\oplus\calO(b_2)^{m_2}$, there is an efficient way to compute $\calA_{\muhat,D}^\calE(q)$ and we can prove the following result (see Corollary \ref{polynomialitymm}).

\begin{theorem} Assume that $\calE=\calO(b_1)^{m_1}\oplus\calO(b_2)^{m_2}$. Then $\calA_{\muhat,D}^\calE(q)$ is given by the evaluation at $q$ of a polynomial with integer coefficients which depend only on the degrees $d_1,\dots,d_r$ of the points $\a_1,\dots,\a_r$ (and not on their position).
\label{introindep}\end{theorem}
\bigskip

Note that Conjecture \ref{MAINconj} implies that the above statement is true for any $\calE$.

Let us state our last main result which is proved using Theorem \ref{MAINtheo1}.

Assume that $\muhat$ is indivisible (namely the gcd of the parts of all coordinates of $\muhat$ equals $1$). Then under some condition on the characteristic of $\F_q$ (see \S \ref{geominter}) and for $q$ sufficiently large we can always make a choice of a \emph{generic} $r$-tuple $\A=(A_1,\dots,A_r)$ of semisimple adjoint orbits of $\gl_n(q^{d_1}),\dots,\gl_n(q^{d_r})$ of type $\muhat$, i.e. the multiplicities of the eigenvalues of $A_i$ coincides with the parts of the partition $\mu^i$.

We then consider

$$
X_{\A,D}^\calE:=\left\{\varphi:\calE\rightarrow\calE\otimes\Omega^1(D)\,\left|\, {\rm Res}_{\a_i}(\varphi)\in A_i \text{ for all }i=1,\dots,r\right.\right\}
$$
and we prove the following result using Fourier analysis (see Theorem \ref{maintheoHiggs}).

\begin{theorem}
$$
\frac{|X_{{\bf A},D}^\calE|}{|{\rm PAut}(\calE)|}=q^{\frac{1}{2}d_{\bf A}}\calA_{\muhat,D}^\calE(q),
$$
with $d_{\bf A}:=\sum_{i=1}^rd_i\,{\rm dim}\, A_i(\overline{\F}_q)-2n^2+2$.
\label{theointro} \end{theorem}

This result is used to bring evidences to the more direct conjectural geometric interpretation of $\calA_{\muhat,D}^\calE(q)$ in the indivisible case (see Conjecture \ref{congind} and Remark \ref{maintheoappli}) but also to Conjecture \ref{MAINconj} when applied to $r$-tuple of semisimple regular adjoint orbits.

\subsection{Example : Rank 2 case}

We keep the same notation as in \S \ref{section2}. 

We use Theorem \ref{MAINtheo2} to compute explicitly $\calA_{\muhat,D}^\calE(q)$ when $n=2$. Reducing the number of points of $D$ if necessary, we assume without loss of generality that the coordinates of $\muhat$ are all equal to the partition $(1,1)$ of $2$. 

Put $l=\sum_{i=1}^rd_i$. For $0< m\leq l$  denote by $X^l_m$ the set of subsets of $\{1,\dots,l\}$ of size $m$. The symmetric group $\mathfrak{S}_l$ acts naturally on $X^l_m$ and we denote by $\chi^l_m$ the character of the representation of $\mathfrak{S}_l$ in the $\C$-vector space freely generated by $X^l_m$. 

Let $w\in\frak{S}_l$ be of cycle-type decomposition $(d_1,\dots,d_r)$.

 We use the formula in Theorem \ref{MAINtheo2} to prove the following result (see \S \ref{rank2}).

\begin{theorem} (i) If $\calE=\calO(a)^2$, then 

$$
\calA_{\muhat,D}^\calE(q)=\begin{cases}\sum_{m=0}^{l-3}\left(\sum_{i=1}^{\left[\frac{l-m-1}{2}\right]}\chi^l_{m+2i+1}(w)\right)\, q^m&\text{ if }l\geq 3,\\
0&\text{ otherwise.}
\end{cases}
$$
(ii) If $\calE=\calO(a)\oplus\calO(b)$ with $a>b$, then

$$
\calA_{\muhat,D}^\calE(q)=\begin{cases}\sum_{m=0}^{l-(a-b+2)}\left(\sum_{s=m+a-b+2}^l\chi^l_s(w)\right)\, q^m&\text{ if }l\geq a-b+2,\\
0&\text{ otherwise.}
\end{cases}
$$
\end{theorem}

\begin{remark} We see from the above theorem that the polynomials $\calA_{\muhat,D}^\calE(q)$ are monic when $n=2$ and so the corresponding spaces $\calX_\bfS^\calE$ of \S \ref{conjsec} must be irreducible  thanks to Theorem \ref{theointro}.
\end{remark}

We deduce the following corollary which confirm the independence of $\calA_{\muhat,D}^{n,d}(q)$ from $d$ in the $n=2$ case.

\begin{corollary}For any $d\in\Z$,

$$
\calA_{\muhat,D}^{2,d}(q)=\begin{cases}\sum_{m=0}^{l-3}\sum_{a=1}^{\left[\frac{l-m-1}{2}\right]}\left(\sum_{s=m+2a+1}^l\chi_s^l(w)\right)\, q^m&\text{ if }l\geq 3,\\
0&\text{ otherwise.}
\end{cases}
$$
\end{corollary}

Let $\bfC=(C_1,\dots,C_r)$ be a generic $r$-tuple of semisimple regular conjugacy classes of $\GL_2(\C)$. The Poincar\'e polynomial of the character variety $\M_\bfC$ has been computed first by Boden and Yokogawa \cite{boden} and one can verify that the above formula for $q^{\frac{1}{2}{\rm dim}\,\M_\bfC}\calA_{\muhat,D}^{(2,d)}(q)$, with $d_1=\cdots=d_r=1$ (i.e. $w=1$), matches their calculation and also that of \cite[\S 1.5.3]{aha} for the conjectured mixed Hodge polynomial of $\M_\bfC$. 
\bigskip

\textbf{Acknowlegments.} I am grateful to P. Boalch, T. Hausel, F. Rodriguez-Villegas and P. Satg\'e for many useful discussions. I am also very grateful to P. Deligne and O. Schiffmann for sharing their insight with me. This work was motivated in part by Deligne's lectures \emph{Syst\`emes locaux l-adiques sur une vari\'et\'e sur un corps fini­} given at IHES in Spring 2013 and  I wish to thank him for his motivating lectures. This work was supported by ANR-13-BS01-0001-01.

\section{Geometrically indecomposable parabolic bundles: Generalities}\label{generalities}

Let $k$ be a field and  denote by $\bbP^1_k$ the projective line over $k$. For a closed point $\a$ of $\bbP^1_k$, we denote by $k(\a)$ its residue field.

\subsection{Parabolic vector bundles: Definitions}

\begin{nothing}\textbf{Generalities on vector bundles over $\bbP^1$.}
Denote by ${\rm Bun}(\bbP^1_k)$ the category of all vector bundles over $\bbP^1_k$ (which we identify with locally free sheaves on $\bbP^1_k$) and by ${\rm Bun}^{n,d}(\bbP^1_k)$ the full subcategory of vector bundles of rank $n$ and degree $d$. The trivial line bundle is denoted by $\calO=\calO_{\bbP^1_k}$. By a well-known result of Grothendieck,  any vector bundle  on $\bbP^1_k$ is isomorphic to a direct sum $\calO(b_1)^{m_1}\oplus\calO(b_2)^{m_2}\oplus\cdots\oplus\calO(b_f)^{m_f}$ and the two sequences of integers $b_1>b_2>\cdots>b_f$ and $(m_1,\dots,m_f)\in(\N)^s$ are uniquely determined by $\calE$. For a vector bundle   $\calE$  on $\bbP^1_k$ and a closed point $\a$ of $\bbP^1_k$, we denote by $\calE(\a)$ the $k(\a)$-vector space $\calE_\a\otimes k(\a)$ and for a vector bundle morphism $\varphi:\calE\rightarrow\calE'$, we denote by $\varphi(\a)$ the induced $k(\a)$-linear map $\calE(\a)\rightarrow\calE'(\a)$.

For a non-negative integer $d$,  we identify  ${\rm Hom}\left(\calO,\calO(d)\right)$ with the global sections of $\calO(d)$  by mapping $\varphi\in {\rm Hom}\left(\calO,\calO(d)\right)$ to $\varphi(1)$. Writing $\bbP^1_k$ as a gluing of $U_1={\rm Spec}\,k[t]$ with $U_2={\rm Spec}\,k[t^{-1}]$, the global sections of $\calO(d)$ 
are given by pairs of polynomials $(P(t),Q(t^{-1}))\in k[t]\times k[t^{-1}]$ satisfying the relation $P(t)=t^dQ(t^{-1})$, and so the projection $(P,Q)\mapsto P$ gives an isomorphism between the global sections of $\calO(d)$ and the polynomials in $k[t]$ of degree less or equal to $d$. If $d\geq d'$, we  identify ${\rm Hom}\left(\calO(d'),\calO(d)\right)$ with the $k$-vector space of polynomials of $k[t]$ of degree less or equal to $d-d'$. The closed points of $\bbP^1_k$ of degree $d$  correspond to irreducible monic polynomials of degree $d$ in $k[t]$ and so define injections $\calO\hookrightarrow\calO(d)$. 

Consider a rank $n$ vector bundle  $\calE\simeq\bigoplus_{i=1}^f\calO(b_i)^{m_i}$ with $b_1>b_2>\cdots>b_f$. We denote by  $L_\calE$ the standard Levi subgroup $\GL_{m_1}\times\cdots\times \GL_{m_f}$ of $\GL_n$, by $G_\calE$ the unique parabolic subgroup of $\GL_n$ having $L_\calE$ as a Levi factor and containing the upper triangular matrices, and by $U_\calE$ the unipotent radical of $G_\calE$. For any commutative ring $R$, identify ${\rm Mat}_n(R[t])$ with $\prod_{1\leq i,j\leq f}{\rm Mat}_{m_i,m_j}(R[t])$ in the obvious way and denote by $U_{\calE,t}(R)$ the subgroup of $U_\calE(R[t])$ whose $(i,j)$-component are $m_i\times m_j$- matrices with coefficients in the $R$-module of polynomials of degree less or equal to $b_i-b_j$ whenever $i<j$, is the matrix $0$ when $i>j$ and the identity matrix $I_{m_i}$ when $i=j$. We then put $G_{\calE,t}(R):=L_\calE(R)\ltimes U_{\calE,t}(R)$. The automorphism group ${\rm Aut}(\calE)$ is then naturally identified with the  group $G_{\calE,t}(k)$ and the endomorphism algebra ${\rm End}(\calE)$ with ${\rm Lie}(G_{\calE,t})(k)$. For a closed point $\a$ of $\bbP^1_k$ contained in $U_1$, the natural map ${\rm End}(\calE)\rightarrow {\rm End}_{k(\a)} (\calE(\a))$ coincides then with the evaluation map ${\rm Lie}(G_{\calE,t})(k)\rightarrow {\rm Lie}\,G_\calE(k(\a))$ given by the canonical map $k[t]\rightarrow k(\a)$.

Recall that the semisimple elements of a connected affine algebraic group $G$ lives in a maximal torus of $G$  and that all maximal tori are conjugate. Therefore, the semisimple element of the algebraic group $G_{\calE,t}$ are all conjugate to a semisimple element of $L_\calE$.

\label{aut}
\end{nothing}
\begin{nothing}\textbf{Parabolic vector bundles.}\label{parvec}
Fix a reduced divisor $D=\a_1+\dots+\a_r$ of $\bbP^1_k$ with $\a_i$ of degree $d_i$. We assume that $\a_1,\dots,\a_r$ are all in $U_1$. Given a vector bundle $\calE$, we consider the set $\calF(\calE)$ of all $r$-tuples $E=(E_{\a_1},\dots,E_{\a_r})$ of infinite non-increasing sequences

$$E_{\a_i}\,:\,\calE(\a_i)=E_{i0}\supseteq E_{i1}\supseteq\cdots \supseteq E_{ir_i}\supseteq\cdots$$ of $k(\a_i)$-vector spaces with only finitely many non-zero terms.

We call an element of $\calF(\calE)$ a \emph{parabolic structure} on $\calE$. A \emph{Borelic structure} is a parabolic structure with full flags.  

The category ${\rm Bun}_D^{par}(\bbP^1_k)$ of \emph{parabolic}\footnote{In the literature,  the terminology \emph{quasi-parabolic} is rather used, see footnotes in the  introduction.} vector bundles on $\bbP^1_k$ is defined as follows. The objects of ${\rm Bun}_D^{par}(\bbP^1_k)$ are pairs $(\calE,E)$ with $\calE\in {\rm Bun}(\bbP^1_k)$ and $E\in\calF(\calE)$ and morphisms $(\calE,E)\rightarrow(\calE',E')$ are vector bundles morphisms $\varphi: \calE\rightarrow\calE'$ such that for each $i=1,\dots,r$, the induced $k(\a_i)$-linear map $\varphi(\a_i)$ preserves the partial flags, namely $\varphi(\a_i)(E_{ij})\subseteq E'_{ij}$ for all positive integer $j$ which for short will be denoted by $\varphi(E)\subseteq E'$. The obvious definition of direct sums makes ${\rm Bun}_D^{par}(\bbP^1_k)$ a $k$-linear additive category. 

\begin{remark} Each point $\a_i$  defines an inclusion of locally free sheaves $\calO\subseteq\calO(d_i)$ and so an inclusion $\calE\subseteq\calE(d_i):=\calE\otimes\calO(d_i)$ for any vector bundle $\calE$. Note that giving a non-increasing filtration $\calE(\a_i)=E_{i,0}\supseteq E_{i,1}\supseteq\cdots\supseteq E_{i,l}$ of $k(\a_i)$-vector spaces is equivalent to giving a non-increasing filtration 
$$\calE(d_i)=\calE_{i,0}\supseteq\calE_{i,1}\supseteq\cdots\supseteq\calE_{i,l}\supseteq\calE$$of locally free sheaves.  This correspondence is given by  $E_{i,s}={\rm Ker}\left(\calE(\a_i)\rightarrow\calE_{i,s}(\a_i)\right)$.
\end{remark}

Recall the following fact \cite[Corollary 1.8.1]{GL}.

\begin{proposition}The category ${\rm Bun}_D^{par}(\bbP^1_k)$ is a \emph{Krull-Remak-Schmidt category}, i.e.,  any non-zero  object is a finite direct sum of indecomposable objects and up to a permutation, the indecomposable components in such a direct sum are unique up to isomorphism. \end{proposition}

Given a \emph{parabolic type} of rank $n$, that is an $r$-tuple $\s=(s_1,\dots,s_r)$ of infinite non-increasing sequences $s_i\,:\, s_{i0}=n\geq s_{i1}\geq\cdots\geq s_{il_i}\geq\cdots$ of non-negative integers with finitely many non-zero terms, we denote by $\calF_\s(\calE)$ the set of parabolic structures $E=(E_{\a_1},\dots,E_{\a_r})\in\calF(\calE)$ of type $\s$ on a rank $n$ vector bundle $\calE$, namely the $k(\a_i)$-vector subspaces $E_{ij}$ of $\calE(\a_i)$  are of dimension $s_{ij}$. The rank of a parabolic type $\s$ is denoted by $|\s|$.  The group ${\rm Aut}(\calE)$ acts in the obvious way  on $\calF_\s(\calE)$. 
Say that two parabolic structures $E$ and $E'$ on $\calE$ are \emph{isomorphic} if the two parabolic bundles $(\calE,E)$ and $(\calE,E')$ are isomorphic, i.e.  if they live in the same ${\rm Aut}(\calE)$-orbit in $\calF_\s(\calE)$ for some parabolic type $\s$.

For a non-increasing sequence of integers $n\geq n_1\geq n_2\geq\cdots\geq n_h>0$ consider the unique parabolic subgroup of $\GL_n$ containing both the upper triangular matrices and the block diagonal matrices $\GL_{n_h}\times\GL_{n_{h+1}-n_h}\times\cdots\times\GL_{n-n_1}$.   For each $i=1,\dots,r$, let $P_{s_i}$ denote the parabolic subgroup of $\GL_n$ corresponding to the sequence $s_i$. Then the natural action of ${\rm Aut}(\calE)$ on $\calF_\s(\calE)$ is identified with the action of $G_{\calE,t}(k)$ on 
$$
\GL_n(k(\a_1))/P_{s_1}(k(\a_1))\times\cdots\times\GL_n(k(\a_r))/P_{s_r}(k(\a_r))
$$
given by $f\cdot (g_1P_{s_1},\dots,g_rP_{s_r})=(f(\a_1)g_1P_{s_1},\dots,f(\a_r)g_rP_{s_r})$ where $f\mapsto f(\a_i)$ is the evaluation map $G_{\calE,t}(k)\rightarrow\GL_n(k(\a_i))$ given by $k[t]\rightarrow k(\a_i)$ on each matrix coordinate. 
\end{nothing}

\subsection{Geometrically indecomposable parabolic vector bundles}

\begin{nothing}\textbf{Definition and finiteness property}. Let $k$ be a finite field and let $K/ k$ be an algebraic  field extension. We denote by $f=f_{K/k}: \bbP^1_K\rightarrow \bbP^1_k$ the canonical map. We define extension of scalars of parabolic vector bundles from $k$ to $K$ as the pull back functor  $f^*:{\rm Bun}_D^{par}(\bbP^1_k)\rightarrow {\rm Bun}_{D_K}^{par}(\bbP^1_K)$ where $D_K=K\times_kD$. If $(\calE',E')=f^*(\calE,E)$ with $\calE=\bigoplus_i\calO_{\bbP^1_k}(b_i)^{m_i}$, then $\calE'=\bigoplus_i\calO_{\bbP^1_K}(b_i)^{m_i}$ and if $\b_i$ is a closed point of $\bbP^1_K$ above $\a_i$, then $\calE'(\b_i)\simeq\calE(\a_i)\otimes_{k(\a_i)}K(\b_i)$ and $E'$ is obtained from $E$ as $E'_{\b_i}=E_{\a_i}\otimes_{k(\a_i)}K(\b_i)$, i.e., $E'_{ij}=E_{ij}\otimes_{k(\a_i)}K(\b_i)$ for all $j$.

We have isomorphisms

$$
{\rm End}(f^*\calE)\simeq K\otimes_k{\rm End}(\calE)\,,\,\,\,\, {\rm End}(f^*(\calE,E))\simeq K\otimes_k{\rm End}(\calE,E).
$$

A parabolic vector bundle $(\calE,E)$ is said to be \emph{geometrically indecomposable} over $k$, if for $K=\overline{k}$, with $\overline{k}$ an algebraic closure of $k$, the parabolic vector bundle $f^*(\calE,E)$ is indecomposable. 

For $(\calE,E)\in{\rm Bun}_D^{par}(\bbP^1_k)$, denote by ${\rm Jac}(\calE,E)$ the Jacobson radical ideal of the $k$-algebra ${\rm End}(\calE,E)$. It is the largest two-sided nilpotent ideal. 

If $(\calE,E)$ is indecomposable then ${\rm Jac}(\calE,E)={\rm End}_{\rm nil}(\calE,E)$. Indeed $(\calE,E)$ is indecomposable if and only if any element of ${\rm End}(\calE,E)$ is either nilpotent or invertible. 

Since the extension $K/k$ is separable, by  \cite[\S 7, no 2, Corollaire 2]{bourbaki} we have an isomorphism

$$
K\otimes_k {\rm Jac}(\calE,E)\simeq {\rm Jac}(f^*(\calE,E)).
$$

Put 

$$
{\rm topEnd}(\calE,E):={\rm End}(\calE,E)/{\rm Jac}(\calE,E).
$$
We then have 

\begin{equation}
K\otimes_k{\rm topEnd}(\calE,E)\simeq {\rm topEnd}(f^*(\calE,E)).
\label{isoJac}\end{equation}
From the above isomorphisms, we deduce the following criterion for geometrical indecomposability.

\begin{proposition} Let  $(\calE,E)\in{\rm Bun}_D^{par}(\bbP^1_k)$. The following statements are equivalent :

(1) $(\calE,E)$ is geometrically indecomposable.

(2) ${\rm topEnd}(\calE,E)\simeq k$.

(3) ${\rm End}(\calE,E)$ does not contain non-scalar semisimple elements.
\label{indecriterion}\end{proposition}

\begin{proof} From the isomorphism (\ref{isoJac}), the parabolic vector bundle   $(\calE,E)$ is geometrically indecomposable if and only if $\overline{k}\otimes_k{\rm topEnd}(\calE,E)$ is a division ring. The assertion (2) implies thus (1) and for the converse implication we recall that $\overline{k}$ is the only finite dimensional associative division algebra over $\overline{k}$.

\end{proof}

Say that a parabolic structure $E\in\calF(\calE)$ is indecomposable  (resp. geometrically indecomposable) if the parabolic bundle $(\calE,E)$ is indecomposable (resp. geometrically indecomposable). We denote by $\calF_\s(\calE)^{\rm ind}$ (resp. $\calF_\s(\calE)^{\rm geo-ind}$) the subset of $\calF_\s(\calE)$ of indecomposable (resp. geometrically indecomposable) parabolic structures on $\calE$. 

We want to prove the following result.

\begin{theorem}Given integers $n\geq1$ and $d$, there is only a finite number of isomorphism classes of geometrically indecomposable parabolic bundles $(\calE,E)$ with $\calE$ of degree $d$ and rank $n$.
\label{finiteness}\end{theorem}

It is a consequence of the following proposition.

\begin{proposition} Consider the vector bundle on $\bbP^1_k$

$$
\calE=\bigoplus_{i=1}^f\calO(b_i)^{m_i},
$$
with $b_1>b_2>\cdots>b_f$ and $f\geq 2$, and assume that for some $i=1,\dots,f-1$ we have $b_i-b_{i+1}+1\geq\sum_{s=1}^rd_s$. Then there is no geometrically indecomposable  parabolic structure on $\calE$.
\label{DS}\end{proposition}

Before proving Proposition \ref{DS} let us recall the following standard result (see for instance \cite[Lemma 2.6.6]{letellier} for Lie algebras).

\begin{lemma}Let $Q=M\ltimes U$ be a parabolic subgroup of $\GL_n(k)$ with $U$ the unipotent radical of $Q$. Let $\sigma\in\GL_n$ such that the centralizer $C_{\GL_n}(\sigma)\subset M$. Then the map $U\rightarrow\sigma U$, $u\mapsto u\sigma u^{-1}$ is an isomorphism.
\label{lem2.6.6}\end{lemma}

\begin{proof}[Proof of Proposition \ref{DS}]We are reduced to prove the statement for algebraically closed fields. We thus assume that $k$ is algebraically closed. In particular, $\sum_{s=1}^rd_s=r$. We are going to define  for any parabolic structure $E$ on $\calE$ a non-scalar semisimple element of ${\rm End}(\calE,E)$. The proposition will be then a consequence of Proposition \ref{indecriterion}. 

We put $n=\sum_{i=1}^fm_i$ and assume that $b_i-b_{i+1}+1\geq r$ for some $i=1,\dots,f-1$. Let $a,b$ be any distinct non-zero elements of $k$. Denote by $\sigma=(a_{hh})$ the diagonal matrix in $\GL_n(k)$ with $a_{hh}=a$ for $h\leq \sum_{j\leq i}m_j$ and  $a_{hh}=b$  elsewhere. Denote by $M$ the centralizer of $\sigma$  in $\GL_n$ and by $Q$ the parabolic subgroup of $\GL_n$ containing the upper triangular matrices and having $M$ as a Levi subgroup. Then clearly $M$ contains $L_\calE$ and $Q$ contains the parabolic $G_\calE$. Denote by $U$ the unipotent radical of $Q$.
Chose any parabolic subgroups $P_1,\dots,P_r$ of $\GL_n$ and a parabolic structure $E=(g_1P_1,\dots,g_rP_r)$. We need to prove that there exists non-trivial $u_1,\dots,u_r\in U$  such that for each $i=1,\dots,r$, the element $\sigma u_i$ stabilizes $g_iP_i$, namely $g_i^{-1}\sigma u_ig_i\in P_i$. Indeed, because of the assumption $b_i-b_{i+1}+1\geq r$, there exists a non-trivial $u\in {\rm Aut}(\calE)$ whose off-diagonal coordinates are zero outside the block defining $U$ such that for any $i=1,\dots,r$, we have $u(\a_i)=u_i$. The element $\sigma u$ belongs then to ${\rm Aut}(\calE,E)$ and by Lemma \ref{lem2.6.6} it is semisimple.

Since for any parabolic subgroup $P$ of $\GL_n$ we have

$$
\GL_n=\bigcup_{w\in \frak{S}_n}QwP
$$
from Bruhat decomposition, we can write each $g_i$ in the form $q_iw_i$ for some $q_i\in Q$ and $w_i\in \frak{S}_n$. Now write $q_i$ as $t_iv_i$ with $t_i\in M$ and $v_i\in U$ where $U$ is the unipotent radical of $Q$. Then for any $u_i$ in $U$ we have $q_i^{-1}\sigma u_iq_i=v_i^{-1}t_i^{-1}\sigma t_i t_i^{-1} u_i t_i v_i$. Since $\sigma$ is central in $M$, we have $q_i^{-1}\sigma u_iq_i=v_i^{-1}\sigma v_i v_i^{-1} t_i^{-1}u_i t_i v_i$. Using again Lemma \ref{lem2.6.6}, we see that $v_i^{-1}\sigma v_i=\sigma f_i$ for some $f_i\in U$. We now choose $u_i$ such that $f_iv_i^{-1}t_i^{-1}u_it_iv_i=1$. Then $q_i^{-1}\sigma u_i q_i=\sigma$. Since $\sigma$ is diagonal it belongs to the parabolic subgroup $w_iP_iw_i^{-1}$ and so $\sigma u_i$ stabilizes $g_iP_i$.
\end{proof}

\end{nothing}

\begin{nothing}\textbf{Example: Rank-two parabolic bundles.}\label{exrk2}
Vector bundles on $\bbP^1_k$ of rank two are of the form $\calE_{a,b}=\calO(a)\oplus\calO(b)$ with $a\geq b$. We have 

$$
{\rm Aut}(\calE_{a,a})\simeq \GL_2(k),\hspace{1cm}{\rm Aut}(\calE_{a,b})\simeq\left.\left\{\left(\begin{array}{cc}\alpha&P(t)\\0&\beta\end{array}\right)\,\right|\,\alpha,\beta\in k, P(t)\in k[t]\text{ of degree }\leq a-b\right\}
$$
if $a>b$.
Note that the only proper parabolic subgroups of $\GL_2$ are the Borel subgroups. Let thus $B$ be the Borel subgroup of $\GL_2$ of upper triangular matrices and consider the space of borelic structures 

$$
\calF_{\rm bor}(\calE_{a,b})=\GL_2(k(\a_1)/B(k(\a_1))\times\cdots\times\GL_2(k(\a_r))/B(k(\a_r)).
$$
 Let us prove the following result.
 
 \begin{lemma} (i) There exists a geometrically indecomposable borelic structure on $\calE_{a,a}$ if and only if $\sum_{s=1}^rd_s\geq 3$. Moreover there is a unique geometrically borelic structure on $\calE_{a,a}$ up to isomorphism if and only if $\sum_{s=1}^rd_s=3$.
 
 \noindent (ii) There exists a geometrically indecomposable borelic structure on $\calE_{a,b}$ , with $a>b$, if and only if $a-b+2\leq\sum_{s=1}^r d_s$. Moroever it is unique (up to isomorphism) if and only if $a-b+2=\sum_{s=1}^rd_s$.
 \label{existence}\end{lemma}

\begin{proof}We use Proposition \ref{indecriterion}. 

Let us prove (i). 

Consider

$$
u_i=\left(\begin{array}{cc}1&c_i\\0&1\end{array}\right),\hspace{1cm}\sigma=\left(\begin{array}{cc}0&1\\1&0\end{array}\right),\hspace{1cm}g=\left(\begin{array}{cc}\alpha&\beta\\\gamma&\delta\end{array}\right)
$$
with $c_i\in k(\a_i)\backslash k$, $g\in\GL_2(k)$. The element $g$ stabilizes $u_i\sigma B$ if and only if 

\beq
-\gamma c_i^2+(\alpha-\delta)c_i+\beta=0.
\label{equa}\eeq
In particular if $d_i\geq 3$ this equation holds if and only if $\gamma=\beta=0$ and $\alpha=\delta$, i.e., if $g$ is a central matrix. Hence if one of the degrees $d_1,\dots,d_r$, say $d_i$, is larger than $3$, then any borelic structure having $u_i\sigma B$ in its $i$-th coordinate is geometrically indecomposable. If now we have $d_i=2$ and $d_j=1$, then any borelic structure having $u_i\sigma B$ in its $i$-th coordinate and $B$ in its $j$-th coordinate will be geometrically indecomposable. Indeed, $g\in\GL_2(k)$ stabilizes $B$ if and only if $\gamma=0$. Now by equation (\ref{equa}), if $\gamma=0$ we must have $\alpha-\delta=\beta=0$. The last case is when $d_i=d_j=d_s=1$. In this case it is easy to see that any borelic structure having $B$ in its $i$-th coordinate, $\sigma B$ in its $j$-th coordinate and $u_s\sigma B$ in its $s$-th coordinate  must be geometrically indecomposable. Conversely, from the above discussion we prove easily that if $\sum_id_i<3$, then there is no geometrically indecomposable borelic structures.

Let us now prove (ii). Note that the ``only if" part is precisely Proposition \ref{DS}. We assume that $a-b+1<\sum_{s=1}^rd_s$ and we construct a geometrically indecomposable borelic structure. Let $u_i$, $\sigma$ be as above and consider 

$$
g(t)=\left(\begin{array}{cc}\alpha&P(t)\\0&\beta\end{array}\right)\in {\rm Aut}(\calE_{a,b}).
$$
Notice that $g(\a_i)$ always stabilizes $B$. It stabilizes $\sigma B$ if and only if $P(\a_i)=0$ and stabilizes $u_i\sigma B$ if and only if $P(\a_i)=c_i(\beta-\alpha)$. Consider the borelic structure $E=(\sigma B,\dots,\sigma B,u_r\sigma B)$. Suppose that $g(t)$ stabilizes $E$. We must have $P(\a_i)=0$ for $1\leq i\leq r-1$ and $P(\a_r)=c_r(\beta-\alpha)$. For each $i=1,\dots,r$, let $P_i$ be a generator of the prime ideal $\a_i$. If $P(t)\neq 0$, the first constraint implies that $P=P_1\cdots P_{r-1}T$ for some $T\in k[t]$. Put $Q=P_1\cdots P_{r-1}$. The second constraint implies that $T(\a_r)=c_r(\beta-\alpha) Q(\a_r)^{-1}$. Since $a-b<\sum_{s=1}^rd_s-1$ we must have ${\rm deg}(T)\leq d_r-2$. If $\beta-\alpha\neq 0$, we can choose $c_r\in k(\a_r)$ such that $T(\a_r)\neq c_r(\beta-\alpha) Q(\a_r)^{-1}$ for all $T\in k[t]$ of degree less or equal to $d_r-2$.  The constraints $P(\a_i)=0$ for $1\leq i\leq r-1$ and $P(\a_r)=c_r(\beta-\alpha)$ are then possible only if $\alpha=\beta$ and $P=0$, i.e., if $g(t)$ is central. We thus deduce that $E$ is geometrically indecomposable. 
 \end{proof}

\end{nothing}

\begin{nothing}\textbf{Parabolic structures on $\calO(a)^n$.}\label{Kacpoly}
Assume that $\calE=\calO(a)^n$ so that ${\rm Aut}(\calE)=\GL_n(k)$, and that $d_1=\cdots=d_r=1$.  Let $\s=(s_1,\dots,s_r)$ be a parabolic type of rank $n$. 

Let $l_i$ be such that $l_i+1$ is the length of $s_i$, i.e., $s_i : n\geq s_{i1}\geq\cdots\geq s_{il_i}>0$.
Consider the following quiver $\Gamma=(I,\Omega)$ with $I$ the set of vertices $\{0\}\cup\{[i,j]\}_{i,j}$ and $\Omega$ the set of arrows. It is equipped with a dimension vector $\v=(v_i)_{i\in I}\in\N^I$ defined by $v_0=n$ and $v_{[i,j]}=s_{ij}$.

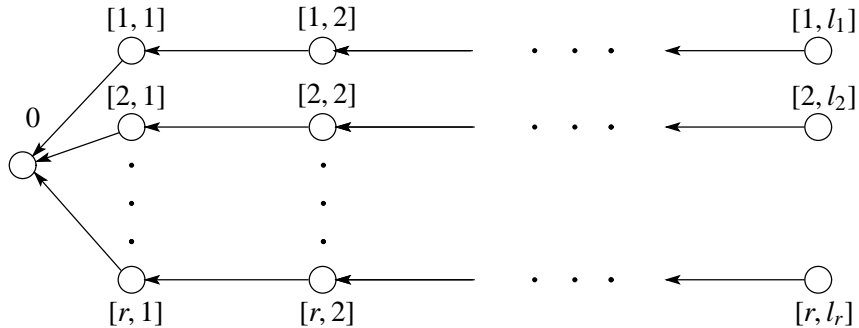
\begin{center}
\unitlength 0.1in
\begin{picture}( 52.1000, 15.4500)(  4.0000,-17.0000)
% CIRCLE 2 0 3 0
% 4 1375 1010 1305 1010 975 1010 975 1010
%
\special{pn 8}%
\special{ar 1376 1010 70 70  0.0000000 6.2831853}%
% CIRCLE 2 0 3 0
% 4 1945 410 1875 410 1545 410 1545 410
%
\special{pn 8}%
\special{ar 1946 410 70 70  0.0000000 6.2831853}%
% CIRCLE 2 0 3 0
% 4 2945 410 2875 410 2545 410 2545 410
%
\special{pn 8}%
\special{ar 2946 410 70 70  0.0000000 6.2831853}%
% CIRCLE 2 0 3 0
% 4 5540 410 5470 410 5140 410 5140 410
%
\special{pn 8}%
\special{ar 5540 410 70 70  0.0000000 6.2831853}%
% CIRCLE 2 0 3 0
% 4 1945 810 1875 810 1545 810 1545 810
%
\special{pn 8}%
\special{ar 1946 810 70 70  0.0000000 6.2831853}%
% CIRCLE 2 0 3 0
% 4 2945 810 2875 810 2545 810 2545 810
%
\special{pn 8}%
\special{ar 2946 810 70 70  0.0000000 6.2831853}%
% CIRCLE 2 0 3 0
% 4 5540 810 5470 810 5140 810 5140 810
%
\special{pn 8}%
\special{ar 5540 810 70 70  0.0000000 6.2831853}%
% CIRCLE 2 0 3 0
% 4 1945 1610 1875 1610 1545 1610 1545 1610
%
\special{pn 8}%
\special{ar 1946 1610 70 70  0.0000000 6.2831853}%
% CIRCLE 2 0 3 0
% 4 2945 1610 2875 1610 2545 1610 2545 1610
%
\special{pn 8}%
\special{ar 2946 1610 70 70  0.0000000 6.2831853}%
% CIRCLE 2 0 3 0
% 4 5540 1610 5470 1610 5140 1610 5140 1610
%
\special{pn 8}%
\special{ar 5540 1610 70 70  0.0000000 6.2831853}%
% VECTOR 2 0 3 0
% 2 1890 1560 1440 1050
%
\special{pn 8}%
\special{pa 1890 1560}%
\special{pa 1440 1050}%
\special{fp}%
\special{sh 1}%
\special{pa 1440 1050}%
\special{pa 1470 1114}%
\special{pa 1476 1090}%
\special{pa 1500 1088}%
\special{pa 1440 1050}%
\special{fp}%
% VECTOR 2 0 3 0
% 2 2870 410 2020 410
%
\special{pn 8}%
\special{pa 2870 410}%
\special{pa 2020 410}%
\special{fp}%
\special{sh 1}%
\special{pa 2020 410}%
\special{pa 2088 430}%
\special{pa 2074 410}%
\special{pa 2088 390}%
\special{pa 2020 410}%
\special{fp}%
% VECTOR 2 0 3 0
% 4 3720 410 3010 410 3730 410 3010 410
%
\special{pn 8}%
\special{pa 3720 410}%
\special{pa 3010 410}%
\special{fp}%
\special{sh 1}%
\special{pa 3010 410}%
\special{pa 3078 430}%
\special{pa 3064 410}%
\special{pa 3078 390}%
\special{pa 3010 410}%
\special{fp}%
\special{pa 3730 410}%
\special{pa 3010 410}%
\special{fp}%
\special{sh 1}%
\special{pa 3010 410}%
\special{pa 3078 430}%
\special{pa 3064 410}%
\special{pa 3078 390}%
\special{pa 3010 410}%
\special{fp}%
% VECTOR 2 0 3 0
% 2 2870 810 2020 810
%
\special{pn 8}%
\special{pa 2870 810}%
\special{pa 2020 810}%
\special{fp}%
\special{sh 1}%
\special{pa 2020 810}%
\special{pa 2088 830}%
\special{pa 2074 810}%
\special{pa 2088 790}%
\special{pa 2020 810}%
\special{fp}%
% VECTOR 2 0 3 0
% 2 2870 1610 2020 1610
%
\special{pn 8}%
\special{pa 2870 1610}%
\special{pa 2020 1610}%
\special{fp}%
\special{sh 1}%
\special{pa 2020 1610}%
\special{pa 2088 1630}%
\special{pa 2074 1610}%
\special{pa 2088 1590}%
\special{pa 2020 1610}%
\special{fp}%
% VECTOR 2 0 3 0
% 4 3730 810 3020 810 3740 810 3020 810
%
\special{pn 8}%
\special{pa 3730 810}%
\special{pa 3020 810}%
\special{fp}%
\special{sh 1}%
\special{pa 3020 810}%
\special{pa 3088 830}%
\special{pa 3074 810}%
\special{pa 3088 790}%
\special{pa 3020 810}%
\special{fp}%
\special{pa 3740 810}%
\special{pa 3020 810}%
\special{fp}%
\special{sh 1}%
\special{pa 3020 810}%
\special{pa 3088 830}%
\special{pa 3074 810}%
\special{pa 3088 790}%
\special{pa 3020 810}%
\special{fp}%
% VECTOR 2 0 3 0
% 4 3730 1610 3020 1610 3740 1610 3020 1610
%
\special{pn 8}%
\special{pa 3730 1610}%
\special{pa 3020 1610}%
\special{fp}%
\special{sh 1}%
\special{pa 3020 1610}%
\special{pa 3088 1630}%
\special{pa 3074 1610}%
\special{pa 3088 1590}%
\special{pa 3020 1610}%
\special{fp}%
\special{pa 3740 1610}%
\special{pa 3020 1610}%
\special{fp}%
\special{sh 1}%
\special{pa 3020 1610}%
\special{pa 3088 1630}%
\special{pa 3074 1610}%
\special{pa 3088 1590}%
\special{pa 3020 1610}%
\special{fp}%
% VECTOR 2 0 3 0
% 2 5465 410 4745 410
%
\special{pn 8}%
\special{pa 5466 410}%
\special{pa 4746 410}%
\special{fp}%
\special{sh 1}%
\special{pa 4746 410}%
\special{pa 4812 430}%
\special{pa 4798 410}%
\special{pa 4812 390}%
\special{pa 4746 410}%
\special{fp}%
% VECTOR 2 0 3 0
% 2 5465 810 4745 810
%
\special{pn 8}%
\special{pa 5466 810}%
\special{pa 4746 810}%
\special{fp}%
\special{sh 1}%
\special{pa 4746 810}%
\special{pa 4812 830}%
\special{pa 4798 810}%
\special{pa 4812 790}%
\special{pa 4746 810}%
\special{fp}%
% VECTOR 2 0 3 0
% 2 5465 1610 4745 1610
%
\special{pn 8}%
\special{pa 5466 1610}%
\special{pa 4746 1610}%
\special{fp}%
\special{sh 1}%
\special{pa 4746 1610}%
\special{pa 4812 1630}%
\special{pa 4798 1610}%
\special{pa 4812 1590}%
\special{pa 4746 1610}%
\special{fp}%
% VECTOR 2 0 3 0
% 2 1880 840 1450 990
%
\special{pn 8}%
\special{pa 1880 840}%
\special{pa 1450 990}%
\special{fp}%
\special{sh 1}%
\special{pa 1450 990}%
\special{pa 1520 988}%
\special{pa 1500 972}%
\special{pa 1506 950}%
\special{pa 1450 990}%
\special{fp}%
% VECTOR 2 0 3 0
% 2 1900 460 1430 960
%
\special{pn 8}%
\special{pa 1900 460}%
\special{pa 1430 960}%
\special{fp}%
\special{sh 1}%
\special{pa 1430 960}%
\special{pa 1490 926}%
\special{pa 1468 922}%
\special{pa 1462 898}%
\special{pa 1430 960}%
\special{fp}%
% DOT 2 0 3 0
% 4 1945 1010 1945 1210 1945 1410 1945 1410
%
\special{pn 8}%
\special{sh 1}%
\special{ar 1946 1010 10 10 0  6.28318530717959E+0000}%
\special{sh 1}%
\special{ar 1946 1210 10 10 0  6.28318530717959E+0000}%
\special{sh 1}%
\special{ar 1946 1410 10 10 0  6.28318530717959E+0000}%
\special{sh 1}%
\special{ar 1946 1410 10 10 0  6.28318530717959E+0000}%
% DOT 2 0 3 0
% 4 4055 410 4265 410 4455 410 4455 410
%
\special{pn 8}%
\special{sh 1}%
\special{ar 4056 410 10 10 0  6.28318530717959E+0000}%
\special{sh 1}%
\special{ar 4266 410 10 10 0  6.28318530717959E+0000}%
\special{sh 1}%
\special{ar 4456 410 10 10 0  6.28318530717959E+0000}%
\special{sh 1}%
\special{ar 4456 410 10 10 0  6.28318530717959E+0000}%
% DOT 2 0 3 0
% 4 4055 810 4265 810 4455 810 4455 810
%
\special{pn 8}%
\special{sh 1}%
\special{ar 4056 810 10 10 0  6.28318530717959E+0000}%
\special{sh 1}%
\special{ar 4266 810 10 10 0  6.28318530717959E+0000}%
\special{sh 1}%
\special{ar 4456 810 10 10 0  6.28318530717959E+0000}%
\special{sh 1}%
\special{ar 4456 810 10 10 0  6.28318530717959E+0000}%
% DOT 2 0 3 0
% 4 4055 1610 4265 1610 4455 1610 4455 1610
%
\special{pn 8}%
\special{sh 1}%
\special{ar 4056 1610 10 10 0  6.28318530717959E+0000}%
\special{sh 1}%
\special{ar 4266 1610 10 10 0  6.28318530717959E+0000}%
\special{sh 1}%
\special{ar 4456 1610 10 10 0  6.28318530717959E+0000}%
\special{sh 1}%
\special{ar 4456 1610 10 10 0  6.28318530717959E+0000}%
\put(19.7000,-2.4500){\makebox(0,0){$[1,1]$}}%
\put(29.7000,-2.4000){\makebox(0,0){$[1,2]$}}%
\put(55.7000,-2.5000){\makebox(0,0){$[1,l_1]$}}%
\put(19.7000,-6.5500){\makebox(0,0){$[2,1]$}}%
\put(29.7000,-6.4500){\makebox(0,0){$[2,2]$}}%
\put(55.7000,-6.5500){\makebox(0,0){$[2,l_2]$}}%
\put(19.7000,-17.8500){\makebox(0,0){$[r,1]$}}%
\put(29.7000,-17.8500){\makebox(0,0){$[r,2]$}}%
\put(55.7000,-17.8500){\makebox(0,0){$[r,l_r]$}}%
\put(14.3000,-7.6000){\makebox(0,0){$0$}}%
\special{pn 8}%
\special{sh 1}%
\special{ar 2950 1010 10 10 0  6.28318530717959E+0000}%
\special{sh 1}%
\special{ar 2950 1210 10 10 0  6.28318530717959E+0000}%
\special{sh 1}%
\special{ar 2950 1410 10 10 0  6.28318530717959E+0000}%
\special{sh 1}%
\special{ar 2950 1410 10 10 0  6.28318530717959E+0000}%
\end{picture}%
\end{center}

\noindent For an arrow $\gamma\in\Omega$, denote by $t(\gamma)$ the tail of $\gamma$ and by $h(\gamma)$ its head. Recall that a representation of $(\Gamma,\v)$ over $k$ is a collection of $k$-linear maps $\varphi=\{\phi_\gamma: V_{t(\gamma)}\rightarrow V_{h(\gamma)}\}_{\gamma\in\Omega}$ with ${\rm dim}_k V_i=v_i$. It is then clear that if $\varphi$ is indecomposable
then its coordinates must be all injective (see \cite[Lemma 3.2.1]{aha2}). Therefore, by taking the images of the $V_{[i,j]}$ in $V_0$ along the arrows of each leg, we define a surjective map from the set of (geometrically) indecomposable representations of $(\Gamma,\v)$ to the set of (geometrically) indecomposable parabolic structures on $\calO(a)^n$ of type $\s$.  It is not hard to see (see \cite[\S 3.2]{aha2}) that this map is a bijection between isomorphism classes of (geometrically) indecomposables. We can thus deduce a necessary and sufficient condition for the existence of geometrically indecomposable structures on $\calO(a)^n$ using the following theorem due to Kac \cite{kac2}.

\begin{theorem} There exists a geometrically \footnote{Kac used the terminology ``absolutely" instead of ``geometrically".} indecomposable representation of $(\Gamma,\v)$ over $k$ if and only if $\v$ is a root of the Kac-Moody algebra associated to $\Gamma$. Moreover, $\v$ is a real root if and only if the geometrically indecomposable representations of $(\Gamma,\v)$ over $k$ are all isomorphic.
\end{theorem}

\end{nothing}

\begin{nothing}\textbf{Constructability.} \label{constructible}In this section we assume that  $k$ is an algebraically closed field and that $\calE$ is a rank $n$ vector bundle on $\bbP^1_k$.  The set $\calF_\s(\calE)^{\rm ind}$ is then a constructible subset of the $k$-variety

$$
\calF_\s(\calE)\simeq\GL_n/P_{s_1}\times\cdots\times\GL_n/P_{s_r}.
$$
The proof goes along the same line as in the quiver case (see \cite[\S 2.5]{KR}). We reproduce it for the convenience of the reader.

For any integer $d$, define 
\begin{align*}
\calF_\s(\calE)_d:&=\{E\in\calF_\s(\calE)\,|\, {\rm dim}\,{\rm End}(\calE,E)=d\}\\
&=\{E\in\calF_\s(\calE)\,|\, {\rm dim}\,{\rm Stab}_{{\rm Aut}(\calE)}(E)=d\}.
\end{align*}
By a standard argument, this set is locally closed in $\calF_\s(\calE)$.

Consider now the closed subvariety

$$
\bbX=\left\{(E,f)\in\calF_\s(\calE)\times{\rm End}(\calE)\,|\,f(E)\subset E, f\text{ nilpotent}\right\}
$$
of $\calF_\s(\calE)\times{\rm End}(\calE)$ and the projection on the first coordinate $p:\bbX\rightarrow \calF_\s(\calE)$. Then using the fact that the function $\bbX\rightarrow\Z$, $x\mapsto{\rm dim}_x p^{-1}(p(x))$ is upper semicontinuous  we see that for all integer $h$ the set $\{(E,f)\in\bbX\,|\, {\rm dim}_f\,{\rm End}_{\rm nil}(\calE,E)\geq h\}$ is closed. The set $\{E\in\calF_\s(\calE)\,|\,{\rm dim}_0\,{\rm End}_{\rm nil}(\calE,E)\geq h\}$ is thus closed. But ${\rm End}_{\rm nil}(\calE,E)$ being a cone, the trivial endomorphism $0$ belongs to all irreducible components of ${\rm End}_{\rm nil}(\calE,E)$ and so $\{E\in\calF_\s(\calE)\,|\,{\rm dim}\,{\rm End}_{\rm nil}(\calE,E)\geq h\}$ is closed in $\calF_\s(\calE)$. Hence for all integers $d$ and $h$, the set $\calF_\s(\calE)_d^h:=\{E\in\calF_\s(\calE)_d\,|\,{\rm dim}\,{\rm End}_{\rm nil}(\calE,E)\geq h\}$ is closed in $\calF_\s(\calE)_d$. Now by Proposition \ref{indecriterion}, an object $E\in\calF_\s(\calE)$ is indecomposable if and only if ${\rm dim}\,{\rm End}_{\rm nil}(\calE,E)\geq{\rm dim}\,{\rm End}(\calE,E)-1$. Putting $h_o:={\rm dim}\,{\rm End}(\calE,E)-1$, we see that $\calF_\s(\calE)^{\rm ind}=\bigcup_d\calF_\s(\calE)_d^{h_o}$ and so that it is a constructible subset of $\calF_\s(\calE)$.
\end{nothing}

\begin{nothing}\textbf{Galois descent theory.}\label{Frobenius}  Here $k=\F_q$ is a finite field and let $K/k$ be a Galois extension. Consider the generator $\sigma:K\rightarrow K, x\mapsto x^q$ of ${\rm Gal}(K/k)$ and denote by $K^\sigma$ the $K$-algebra with underlying ring $K$ and structural morphism $K\rightarrow K^\sigma=K$, $\lambda\mapsto \lambda^q$. Denote by $F^{\rm abs}:\bbP^1_K\rightarrow\bbP^1_K$ the absolute Frobenius which on each chart ${\rm Spec}\, K[z]$ with $z=t,t^{-1}$ is induced by the ring homomorphism $K[z]\rightarrow K[z]$, $P\mapsto P^q$. Then $F^{\rm abs}$ factorizes as $\mathfrak{F}\circ F$ accordingly to the following commutative diagram 

$$\xymatrix{\bbP^1_K\ar[rr]^F\ar[rrd]&&K^\sigma\times_K\bbP^1_K\ar[d]\ar[rr]^{\mathfrak{F}}&&\bbP^1_K\ar[d]\\
&&{\rm Spec}\, K\ar[rr]^{\sigma}&&{\rm Spec}\,K}
$$

The obvious $k$-structure on $\bbP^1_K$ gives a natural identification $\bbP^1_K\simeq K^\sigma\times_K\bbP^1_K$. Under this identification  the relative Frobenius $F$ is given on  each chart ${\rm Spec}\,K[z]$ with $z=t$ or $z=t^{-1}$ by the $K$-algebra homomorphism $P(z)\mapsto P(z^q)$ while  $\mathfrak{F}$  is given by the $k$-algebra homomorphism $K[z]\rightarrow K[z]$, $\sum_ia_iz^i\mapsto \sum_ia_i^qz^i$. Notice that $F^*(\calO(n))\simeq\calO(qn)$ and $\mathfrak{F}^*(\calO(n))\simeq\calO(n)$ for all integer $n$.
\bigskip

The divisor $D_K$ being defined over $k$, we have a natural identification $D_K$ with $K^\sigma\times_K D_K$. The functor $\mathfrak{F}^*$ acts on ${\rm Bun}^{par}_{D_K}(\bbP^1_K)$ as follows. If $(\calE',E')=\mathfrak{F}^*(\calE,E)$ then $\calE'=\calE$ and for any closed point $\b$ of the divisor $D_K\simeq K^\sigma\times_KD_K$, the partial  flag $E'_\b$ is $K^\sigma\times_KE_{\mathfrak{F}(\b)}$. Notice that if $(\calE,E)$ is obtained by extension of scalars from a parabolic vector bundle on $\bbP^1_k$, then it is automatically $\mathfrak{F}^*$-stable as the following diagram commutes

$$
\xymatrix{\bbP^1_K\simeq K^\sigma\times_K\bbP^1_K\ar[dr]_{f_{K/k}}\ar[rr]^{\mathfrak{F}}&&\bbP^1_K\ar[dl]^{f_{K/k}}\\
&\bbP^1_k&}
$$

From Galois descent theory we get the following result.

\begin{proposition} If the parabolic bundle $(\calE,E)\in{\rm Bun}^{par}_{D_K}(\bbP^1_K)$ is $\mathfrak{F}^*$-stable then it is of the form $f_{K/k}^*(\calE_k,E_k)$ with $(\calE_k,E_k)\in{\rm Bun}^{par}_D(\bbP^1_k)$.
\end{proposition}
\end{nothing}

\section{Counting geometrically indecomposable parabolic bundles over finite fields}\label{counting}

\subsection{Preliminaries and notation} Unless specified $k$ is a finite field $\F_q$ and $\overline{k}$ for an algebraic closure of $k$.  Recall that $D=\a_1+\a_2+\cdots+ \a_r$ is a reduced divisor on $\bbP^1_k$ and $d_i$ is the degrees of  $\a_i$. 

We denote by $\calM_\s^\calE(q)=\calM_{\s,D}^\calE(q)$ (resp. $\calI_\s^\calE(q)=\calI_{\s,D}^\calE(q)$, $\calA_\s^\calE(q)=\calA_{\s,D}^\calE(q)$) the number of isomorphism classes of parabolic structures of type $\s$ on $\calE$ (resp. indecomposable parabolic structures, geometrically indecomposable parabolic structures). 

Fix once for all line bundles $\calO(b_1),\dots,\calO(b_f)$ on $\bbP^1_k$ with $b_1>b_2>\cdots>b_f$.

Then if $\m=(m_1,\dots,m_f)$ is an $f$-tuple of non-negative integers, we put 

$$\calE^\m:=\bigoplus_{i=1}^f\calO(b_i)^{m_i},
$$
 and call $\m$ a vector bundle type. The rank of a vector bundle type $\m$ is by definition the rank of $\calE^\m$ that is $\sum_im_i$ and  is denoted by ${\rm rk}(\m)$. We use the notation $\calA_\s^\m(q)$, $\calI_\s^\m(q)$ and $\calM_\s^\m(q)$ instead of  $\calA_\s^{\calE^\m}(q)$, $\calI_\s^{\calE^\m}(q)$ and $\calM_\s^{\calE^\m}(q)$.

Denote by  $\gcd(\m,\s)$ the gcd of the coordinates of the sequences $\m$ and $\s$. If $r\,\mid\, \gcd(\m,\s)$, we denote by $\m/d$ and $\s/d$ the sequences obtained from $\m$ and $\s$ by dividing all coordinates by $d$.

We have the following relation bewteen the number of indecomposables and geometrically indecomposables.

\begin{proposition}
\begin{align*}
&\calA_\s^\m(q)=\sum_{d\,\mid\, \gcd(\m,\s)}\sum_{r\,\mid\, d}\frac{\mu(r)}{d}\calI_{\s/d}^{\m/d}(q^r),\\
&\calI_\s^\m(q)=\sum_{d\,\mid\, \gcd(\m,\s)}\sum_{r\,\mid\, d}\frac{\mu(r)}{d}\calA_{\s/d}^{\m/d}(q^{d/r}).
\end{align*}
In particular if $\gcd(\m,\s)=1$ we have $\calA_\s^\m(q)=\calI_\s^\m(q)$.
\label{abso-ind}\end{proposition}

\begin{proof} From Galois descent theory (see \cite{KR} in the quiver representations case) we find
\begin{equation}
\calX_\s^\m(q):=\sum_{d\,\mid\, \gcd(\m,\s)}\frac{1}{d}\calI_{\s/d}^{\m/d}(q)=\sum_{d\,\mid\, \gcd(\m,\s)}\frac{1}{d}\calA_{\s/d}^{\m/d}(q^d).
\label{id-abi}\end{equation}
We then use the M\"obius function to  obtain 

$$
\calI^\m_\s(q)=\sum_{d\,\mid\,\gcd(\m,\s)}\frac{\mu(d)}{d}\calX^{\m/d}_{\s/d}(q).
$$
and then use the  identity (\ref{id-abi}) to prove the second formula of  the proposition and so also the first one.

\end{proof}

\subsection{Relation between $\calM_\s^\calE(q)$ and $\calA_\s^\calE(q)$}\label{relation}Let $X$ be an infinite set $\{X_{ij}\,|\,i=1,\dots,r,\, j\in\N\}$    of independent commuting variables and $Y$ be a finite set $\{Y_1,\dots,Y_f\}$ of independent commuting variables. 

For a vector bundle type  $\m=(m_1,\dots,m_f)$ and a parabolic type $\s$, we put 
$$
Y^\m:=Y_1^{m_1}\cdots Y_f^{m_f},\hspace{2cm} X^\s:=\prod_{i,j}X_{ij}^{s_{ij}}.
$$
Denote by $R=K_o({\rm Sch}/\F_q)$ the Grothendieck group of the category ${\rm Sch}/\F_q$ of separated $\F_q$-schemes of finite type. The map ${\rm Sch}/\F_q\rightarrow \Z$, that maps $X$ to $\# X(\F_q)$ extends to a map $R\rightarrow \Z$, $\gamma\mapsto \#\gamma(\F_q)$. For any formal power series 

$$
f(q,X,Y,T):=\sum_{n>0}\sum_{{\rm rk}(\m)=|\s|=n}\# \gamma_\s^\m(\F_q) Y^\m X^\s T^n,
$$
where $\gamma_\s^\m\in R$, we define for each integer $d>0$ the series

$$
\psi_d(f):=f(q^d,X^d,Y^d,T^d),$$
where $X^d$ and $Y^d$ denotes the set of variables $\{X_{ij}^d\}$ and $\{Y_i^d\}$.

We then put 

$$
\Psi(f):=\sum_{d>0}\frac{\psi_d(f)}{d},
$$
and we define the plethystic exponential $\Exp$ of $f$ as
$$
\Exp(f):=\exp(\Psi(f)).
$$
The pletystic logarithm of a power series $g$ of the form $1+f$, with $f$ as above, is defined as

$$
\Log(g)=\Psi^{-1}(\log(g)),
$$
with $\Psi^{-1}(g)=\sum_{d>0}\mu(d)\frac{\psi_d(g)}{d}$ where $\mu$ is the ordinary M\"obius function. The operators $\Exp$ and $\Log$ are inverse of each other.

\begin{remark} If $f=\# \gamma(\F_q) T$, then $\Exp(f)$ is the usual zeta function of $\gamma\in{\rm Sch}/\F_q$.

\end{remark}

\begin{proposition} We have 

$$
\Log\left(1+\sum_{n> 0}\left(\sum_{|\s|=n}\sum_{{\rm rk}(\m)=n}\calM_\s^\m(q)Y^\m X^\s \right)T^n\right)=\sum_{n>0}\left(\sum_{|\s|=n}\sum_{{\rm rk}(\m)=n}\calA_\s^\m(q)Y^\m X^\s\right) T^n.
$$
\label{Hua}\end{proposition}

\begin{proof} The proof goes along the same lines as in the quiver representations case \cite{hua}. We recall it for the convenience of the reader. From the fact that the category of parabolic bundles is Krull-Remak-Schmidt, we have the following formal identity

$$
1+\sum_{n>0}\sum_{|\s|=n}\sum_{{\rm rk}(\m)=n}\calM_\s^\m(q) Y^\m X^\s=\prod_{n>0}\prod_{|\s|=n}\prod_{{\rm rk}(\m)=n}\left(1-Y^\m X^\s\right)^{-\calI_\s^\m(q)}.
$$
Taking the formal $\log$ of this identity we find

$$
\log\left(1+\sum_{n>0}\sum_{|\s|=n}\sum_{{\rm rk}(\m)=n}\calM_\s^\m(q) Y^\m X^\s\right)=\sum_{i>0}\sum_{n>0}\sum_{|\s|=n}\sum_{{\rm rk }(\m)=n}\frac{1}{i}\calI_\s^\m(q)Y^{i\m}X^{i\s}.
$$
We now apply the operator $\Psi^{-1}$ and we use Proposition \ref{abso-ind}.

\end{proof}

\bigskip

\subsection{Character formula for $\calM_\s^\calE(q)$ and $\calA_\s^\calE(q)$} 

In this section we give formulas for $\calM_\s^\calE(q)$ and $\calA_\s^\calE(q)$ in terms of character values of general linear groups over finite fields.

\subsubsection{Harish-Chandra induction: Review}\label{Harish-Chandra}For simplicity, we call the Levi factors of the parabolic subgroups of a connected reductive algebraic group $G$, the \emph{Levi subgroups} of $G$. We consider on $\GL_n$ the usual $\F_q$-structure which corresponds to the relative Frobenius $F:\GL_n(\overline{k})\rightarrow\GL_n(\overline{k})$, $(a_{ij})\mapsto (a_{ij}^q)$. For a subgroup $H$ of $\GL_n$ which is defined over $\F_q$, we denote by $H(q)$ the finite group of $\F_q$-rational points of $H$. We will call  the block diagonal subgroups $\GL_{n_1}\times\GL_{n_2}\times\cdots\times\GL_{n_s}$ of $\GL_n$ the \emph{standard Levi subgroups} of $\GL_n$. The \emph{standard parabolic subgroups} are the parabolic subgroups of $\GL_n$ that contains the upper triangular matrices. Let $L$ be a standard Levi subgroup. Choose a parabolic subgroup $P$ of $\GL_n$ defined over $\F_q$ and having $L$ as a Levi factor (for instance take $P$ to be standard). Denote by $U_P$ the unipotent radical of $P$ so that $P=L\ltimes U_P$. Consider the $\C$-vector space $M$ generated by the finite set $\GL_n(q)/U_P(q)=\{gU_P(q)\,|\, g\in \GL_n(q)\}$. It is endowed with an action of $\GL_n(q)$ by left multiplication and with an action of $L(q)$ by right multiplication (which is well-defined as $U_P$ is a normal subgroup of $P$). Obviously these two actions commutes. We define the \emph{Harish-Chandra functor} $R_{L(q)}^{\GL_n(q)}$ from the category  of finite dimensional left $\C[L(q)]$-modules to the catgeory of finite dimensional $\C[\GL_n(q)]$-modules as 

$$
V\rightarrow M\otimes_{\C[L(q)]}V.
$$
It is well-known that the Harish-Chandra functor depends only on $L$ and not on the choice of $P$. The functor induces a $\C$-linear map, denoted again $R_{L(q)}^{\GL_n(q)}$, from the $\C$-vector space $\calC(L(q))$ of class functions on $L(q)$ to $\calC(\GL_n(q))$. It is given by the following explicit formula for any $f\in\calC(L(q))$ and $g\in\GL_n(q)$:

$$
R_{L(q)}^{\GL_n(q)}(f)(g)=\frac{1}{|P(q)|}\sum_{h\in\GL_n(q),\, h^{-1}gh\in P}f(\pi_P(h^{-1}gh)),
$$
where $\pi_P:P\rightarrow L$ is the canonical projection. 

We have the following easy lemmas.

\begin{lemma}Let $1$ denotes the identity character. For any linear character $\alpha:\F_q^\times\rightarrow\C^\times$ we have

$$
R_{L(q)}^{\GL_n(q)}(\alpha\circ{\rm det})=R_{L(q)}^{\GL_n(q)}(1)\cdot(\alpha\circ {\rm det}).
$$
\label{lemtriv}\end{lemma}

\begin{lemma} For any $g\in\GL_n(q)$, we have

$$
R_{L(q)}^{\GL_n(q)}(1)(g)=\#\left\{h\,P(q)\in \GL_n(q)/P(q)\,\left|\, h^{-1}gh\in P(q)\right\}\right.,
$$
from which we see that for any conjugacy class $C$ of $\GL_n(q)$ we have

$$
R_{L(q)}^{\GL_n(q)}(1)(C)=\frac{|C\cap P(q)|\cdot |G(q)/P(q)|}{|C|}.
$$
\label{lemma}\end{lemma}

\subsubsection{Character formulas}

To a parabolic type $\s=(s_1,\dots,s_r)$ of rank $n$, we can associate an $r$-tuple $(L_1,\dots,L_r)$ of standard Levi subgroups of $\GL_n$ such that the size of the blocks of each $L_i$ is given by the differences $s_{ij}-s_{i(j+1)}\geq 0$ of the terms of the sequence $s_i$. It also defines in the obvious way an $r$-tuple $(P_1,\dots,P_r)$ of standard parabolic subgroups of $\GL_n$, each $P_i$ having $L_i$ as a Levi factor.  

 For $f\in{\rm Aut}(\calE)$, recall that $f(\a_i)$ is the induced automorphism $\calE(\a_i)\rightarrow\calE(\a_i)$  of $k(\a_i)$-vector spaces. It defines a  conjugacy class of $\GL_n(q^{d_i}):=\GL_n(k(\a_i))$ and we denote again by $f(\a_i)$ a representative of this conjugacy class.

\begin{proposition} Let $\s=(s_1,\dots,s_r)$ and $(L_1,\dots,L_r)$ be related as above. Then

$$
\calM_\s^\calE(q)=\frac{1}{|{\rm Aut}(\calE)|}\sum_{f\in{\rm Aut}(\calE)}\prod_{i=1}^kR_{L_i(q^{d_i})}^{\GL_n(q^{d_i})}(1)(f(\a_i)).
$$
\label{burn}\end{proposition}

\begin{proof} The group ${\rm Aut}(\calE)$ acts on $\calF_\s(\calE)$ in the obvious way and $\calM_\s^\calE(q)$ is precisely the number of ${\rm Aut}(\calE)$-orbits on $\calF_\s(\calE)$. Hence by Burnside formula (for the number of orbits of a finite group acting on a finite set) we have 

\beq
\calM_\s^\calE(q)=\frac{1}{|{\rm Aut}(\calE)|}\sum_{f\in{\rm Aut}(\calE)}\#\left\{E\in\calF_\s(\calE)\,\left|\,f(E)= E\right\}\right..
\label{Burnside}
\eeq
We conclude from Lemma \ref{lemma} and the fact that there is an ${\rm Aut}(\calE)$-equivariant bijection between $\calF_\s(\calE)$ and $\prod_{i=1}^r\GL_n(k(\a_i))/P_i(k(\a_i))$.
\end{proof}

Denote by $\calP$ the set of all partitions (including the empty partition) and let $\calP_n$ be the set of partitions of size $n$. A partition $\lambda$ of $n$ will be denoted either in the form $(\lambda_1,\lambda_2,\dots,\lambda_t)$ with $\lambda_1\geq\lambda_2\geq\cdots\geq\lambda_t$ and $\sum_i\lambda_i=n$ or in the form $(1^{m_1},2^{m_2},\dots)$ where $m_i$ denotes the number of $j=1,\dots,t$ such that $i=\lambda_j$. If $\lambda=(\lambda_1,\dots,\lambda_t)$, we call the $\lambda_i$'s the parts of $\lambda$.

Let $\muhat=(\mu^1,\dots,\mu^r)\in(\calP_n)^r$ be the multi-partition obtained from the parabolic type $\s=(s,\dots,s_r)$ such that the parts of $\mu^i$ are given by the successive differences $s_{ij}-s_{i(j+1)}$. The multi-partitions $\muhat=(\mu^1,\dots,\mu^r)$ considered in this paper will always have their coordinates $\mu^i$ all of the same size. We denote by $|\muhat|$ the common size of the coordinates. 

Say that two parabolic types $\s$ and $\s'$ are equivalent if they give the same multi-partition.

Since Harish-Chandra induction does not depend on the choice of the parabolic subgroup, the above proposition has the following consequence.

\begin{corollary} If $\s$ and $\s'$ are equivalent parabolic types then
$$
\calM_\s^\calE(q)=\calM_{\s'}^\calE(q).
$$
\label{s'}\end{corollary}

When $\muhat$ is the multi-partition arising from  a parabolic type $\s$, we will use the notation $\calM_\muhat^\calE(q)$ instead of $\calM_\s^\calE(q)$, and $\M^\m_\muhat(q)$ instead of $\M^\m_\s(q)$  (to avoid any confusion we will use the Greek alphabet for multi-partitions and the latin alphabet for parabolic types).
\bigskip

For a group $H$, denotes by $Z_H$ the center of $H$. 

\begin{remark}If $\alpha:\F_q^\times\rightarrow\C^\times$ is a linear character of order $n$ (i.e. the subgroup generated by $\alpha$ is of order $n$) then the restriction of $(\alpha\circ{\rm det}):\GL_n(q)\rightarrow\C^\times$ to $Z_{\GL_n}$ is trivial but its restriction to $Z_L$ is non-trivial  for any proper Levi subgroup $L$ of $\GL_n$ defined over $\F_q$.
\label{remgen}\end{remark}
\bigskip

\begin{proposition}Let $(L_1,\dots,L_r)$ be an $r$-tuple of standard Levi subgroups of $\GL_n$ obtained from a parabolic type $\s$. Then for any  linear character $\alpha:\F_q^\times\rightarrow\C^\times$ of order $n$, we have

$$
\calA_\s^\calE(q)=\frac{1}{|{\rm Aut}(\calE)|}\sum_{f\in{\rm Aut}(\calE)}\left(\prod_{i=1}^rR_{L_i(q^{d_i})}^{\GL_n(q^{d_i})}(1)(f(\a_i))\right)\cdot (\alpha\circ{\rm det})(f).
$$
\label{propindec}\end{proposition}

\begin{proof}Following the proof of Proposition \ref{burn}, we need to see that

\beq
\calA_\s^\calE(q)=\frac{1}{|{\rm Aut}(\calE)|}\sum_{f\in{\rm Aut}(\calE)}\#\left\{E\in\calF_\s(\calE)\,\left|\,f(E)= E\right\}\right.\,(\alpha\circ{\rm det})(f).
\label{burn2}\eeq
Write 

$$
\calF_\s(\calE)=\coprod_{H\leq {\rm Aut}(\calE)}\left\{E\in\calF_\s(\calE)\,\left|\, {\rm Stab}_{{\rm Aut}(\calE)}(E)=H\right.\right\} 
$$
where the union is over the subgroups of ${\rm Aut}(\calE)$. Put $\calF_\s^H(\calE):=\left\{E\in\calF_\s(\calE)\,\left|\, {\rm Stab}_{{\rm Aut}(\calE)}(E)=H\right.\right\}$.

Then the RHS of (\ref{burn2}) equals

$$
\frac{1}{|{\rm Aut}(\calE)|}\sum_{H\leq{\rm Aut}(\calE)}|\calF_\s^H(\calE)|\, \sum_{h\in H}(\alpha\circ{\rm det})(h).
$$
Assume that $H={\rm Stab}_{{\rm Aut}(\calE)}(E)$, with $E\in\calF_\s(\calE)$, contains a non-central semisimple element, then the character $(\alpha\circ{\rm det})$ is a non-trivial character of $H$. Indeed,  conjugating $H$ in ${\rm Aut}(\calE)$ if necessary, we may assume that $H$ contains a non-central semisimple element $s$ of $L_\calE$. Now an element of $\GL_n$ stabilizes $E$ if and only if it lives in  the intersection of the parabolic subgroups $P_i={\rm Stab}_{\GL_n}(E_i)$, $i=1,\dots,r$,  where $E_i$ is the $i$-th coordinate of $E$. Denote by $L$ the centralizer in $\GL_n$ of $s$. Then $Z_L\subset P_i$ for all $i=1,\dots,r$ and so  $Z_{\GL_n}\subsetneq Z_L\subset H$. By Remark \ref{remgen}, the restriction of $(\alpha\circ{\rm det})$ to $Z_L$ is non-trivial from which we have that $(\alpha\circ{\rm det}):H\rightarrow\C^\times$ is non-trivial. We thus have $\sum_{h\in H}(\alpha\circ{\rm det})(h)=0$ when $H$ contains a non-central semisimple element. 

We deduce that
$$
\calA_\s^\calE(q)=\frac{1}{|{\rm Aut}(\calE)|}\sum_H|H|\, |\calF_\s^H(\calE)|,
$$
where the sum is over the subgroups $H$ of ${\rm Aut}(\calE)$ which do not contains non-central semisimple elements. By Proposition \ref{indecriterion}, a parabolic structure $E\in\calF_\s(\calE)$ is geometrically indecomposable if and only if ${\rm Stab}_{{\rm Aut}(\calE)}(E)/Z_{\GL_n}$ does not contain semisimple elements. It is then not difficult using Burnside formula for $\calA_\s^\calE(q)$ to prove Formula (\ref{burn2}).

\end{proof}

Say that an $r$-tuple $(\alpha_1,\dots,\alpha_r)$ of linear characters of $L_1(q^{d_1}),\dots,L_r(q^{d_r})$  is \emph{generic}, if for any Levi subgroup $M$ of $\GL_n$ defined over $\F_q$ and such that $Z_M(q)\subset g_iL_i(q^{d_i})g_i^{-1}$ for some $g_i\in\GL_n(q^{d_i})$, the linear character $(^{g_1}\alpha_1)|_{Z_M(q)}\cdots(^{g_r}\alpha_r)|_{Z_M(q)}$ is non-trivial except for $M=\GL_n$. If $\alpha:\F_q^\times\rightarrow\C^\times$ is of order $nd_r$ and if $N=N_{\F_{q^{d_r}}/\F_q}:\F_{q^{d_r}}\rightarrow\F_q$ is the norm map $x\mapsto xx^q\cdots x^{q^{d_r-1}}$, then the $r$-tuple $(1,1,\dots,1,\alpha\circ N\circ{\rm det})$ is generic by Remark \ref{remgen}.

\bigskip

\begin{theorem} Let $(L_1,\dots,L_r)$ be an $r$-tuple of standard Levi subgroups of $\GL_n$ obtained from a parabolic type $\s$. Then for any generic tuple $(\alpha_1,\dots,\alpha_r)$ of linear characters of $L_1(q^{d_1}),\dots,L_r(q^{d_r})$ we have 

$$
\calA_\s^\calE(q)=\frac{1}{|{\rm Aut}(\calE)|}\sum_{f\in{\rm Aut}(\calE)}\prod_{i=1}^rR_{L_i(q^{d_i})}^{\GL_n(q^{d_i})}(\alpha_i)(f(\a_i)).
$$
\label{theo}\end{theorem}
\bigskip

For each $i=1,\dots,r$, the character $R_{L_i(q^{d_i})}^{\GL_n(q^{d_i})}(\alpha_i)$ defines, by composing with the natural evaluation map ${\rm Aut}(\calE)\rightarrow \GL_n(q^{d_i})=\GL_n(k(\a_i))$, a character of ${\rm Aut}(\calE)$ which we denote by $R_{L_i(q^{d_i}),\alpha_i}^\calE$. Denote by $\langle\,,\,\rangle$ the usual inner product of class functions on ${\rm Aut}(\calE)$. Then we may re-write Proposition \ref{burn} and Theorem \ref{theo} as 

\begin{align*}
\calM_\s^\calE(q)&=\left\langle R_{L_1(q^{d_1}),1}^\calE\otimes\cdots\otimes R_{L_r(q^{d_r}),1}^\calE,1\right\rangle,\\
\calA_\s^\calE(q)&=\left\langle R_{L_1(q^{d_1}),\alpha_1}^\calE\otimes\cdots\otimes R_{L_r(q^{d_r}),\alpha_r}^\calE,1\right\rangle,
\end{align*}
for any generic tuple $(\alpha_1,\dots,\alpha_r)$ of linear characters of $L_1(q^{d_1}),\dots,L_r(q^{d_r})$.

Notice that if $\calE=\calO(a)^n$ then ${\rm Aut}(\calE)=\GL_n$ and so the above inner products are inner products of characters of $\GL_n(\F_q)$.
\bigskip

To prove the theorem we need the following theorem which will be proved in the next section (see below Theorem \ref{intermediate}).

\begin{theorem} The inner product $\left\langle R_{L_1(q^{d_1}),\alpha_1}^\calE\otimes\cdots\otimes R_{L_r(q^{d_r}),\alpha_r}^\calE,1\right\rangle$ does not depend on the choice of the generic tuple $(\alpha_1,\dots,\alpha_r)$.

\label{indep}\end{theorem}

\begin{proof}[Proof of Theorem \ref{theo}]  From Lemma \ref{lemtriv} and Proposition \ref{propindec} we have 

$$
\calA_\s^\calE(q)=\frac{1}{|{\rm Aut}(\calE)|}\sum_{f\in{\rm Aut}(\calE)}\left(\prod_{i=1}^{r-1}R_{L_i(q^{d_i})}^{\GL_n(q^{d_i})}(1)(f(\a_i))\right)R_{L_r(q^{d_r})}^{\GL_n(q^{d_r})}(\alpha_r)(f(\a_r)),
$$
with $\alpha_r=(\alpha\circ N\circ{\rm det})$ where $\alpha$ is a linear character of $\F_q^\times$ of order $nd_r$. The Theorem follows then from Theorem \ref{indep} and the fact that $(1,\dots,1,\alpha_r)$ is a generic tuple.
\end{proof}

From the independence of Harish-Chandra induction from the choice of the parabolic subgroup, Theorem \ref{theo} implies the following result.

\begin{corollary} If $\s$ and $\s'$ are equivalent parabolic types then

$$
\calA_\s^\calE(q)=\calA_{\s'}^\calE(q).
$$
\label{s'bis}\end{corollary}

If $\muhat$ is the multi-partition obtained from the parabolic type $\s$ we will use also the notation $\calA_\muhat^\calE(q)=\calA_{\muhat,D}^\calE(q)$ instead of $\calA_\s^\calE(q)=\calA_{\s,D}^\calE(q)$ and $\calA^\m_\muhat(q)=\calA^\m_{\muhat,D}(q)$  instead of $\calA^\m_\s(q)=\calA^\m_{\s,D}(q)$.

\subsection{Intermediate formula}

Write 

$$
\left\langle R_{L_1(q^{d_1}),\alpha_1}^\calE\otimes\cdots\otimes R_{L_r(q^{d_r}),\alpha_r}^\calE,1\right\rangle=
\sum_{(C_1,\dots,C_r)}\frac{\#\{f\in{\rm Aut}(\calE)\,|\, f(\a_i)\in C_i\}}{|{\rm Aut}(\calE)|}\prod_{i=1}^rR_{L_i(q^{d_i})}^{\GL_n(q^{d_i})}(\alpha_i)(C_i),
$$
where the sum if over the set of $r$-tuples of conjugacy classes of $\GL_n(q^{d_1})\times\cdots\times\GL_n(q^{d_r})$.
\bigskip

We now identify ${\rm Aut}(\calE)$ with the  group  $G_{\calE,t}(q)=L_\calE(q)\ltimes U_{\calE,t}(q)$ as in \S \ref{aut}. The semisimple part $f_s$ of any $f\in{\rm Aut}(\calE)$ is ${\rm Aut}(\calE)$-conjugate to an element in $L_\calE(q)$. Therefore if $f\in{\rm Aut}(\calE)$ satisfies $f(\a_i)\in C_i$ for all $i=1,\dots,r$, then the conjugacy classes $C_1,\dots,C_r$ must have a common semisimple part in $\GL_n(q)$.

The sum of the above formula is then a sum over the set $\calC$ of $r$-tuples $(C_1,\dots,C_r)$ of conjugacy classes of $\GL_n(q^{d_1})\times\cdots\times\GL_n(q^{d_r})$ such that there exist elements $x_1,\dots,x_r\in C_1,\dots,C_r$ with $(x_1)_s=\cdots=(x_r)_s\in\GL_n(q)$.

Let $(C_1,\dots,C_r)$ be an element of $\calC$ and let $l\in\GL_n(q)$ be the common semisimple part of $x_1,\dots,x_r\in C_1,\dots,C_r$. Denote by $M$ the centralizer $C_{\GL_n}(l)$ of $l$ in $\GL_n$ and for each $i=1,\dots,r$, let $\calO_i$ be the $M(q^{d_i})$-conjugacy class of the unipotent part of $x_i$. If we choose other elements $x_1',\dots,x_r'\in C_1,\dots,C_r$ with common semisimple part $l'\in\GL_n(q)$, then the corresponding tuple $(M',\calO_1',\dots,\calO_r')$ will be $\GL_n(q)$-conjugate to $(M,\calO_1,\dots,\calO_r)$.  We call the $\GL_n(q)$-conjugacy class of $(M,\calO_1,\dots,\calO_r)$ the \emph{type} of $(C_1,\dots,C_r)\in\calC$ (for short we will say that $(M,\calO_1,\dots,\calO_r)$ is the type of $(C_1,\dots,C_r)$).
\bigskip

For a Levi subgroup $M$ of $\GL_n$, denote by $(Z_M)_{\rm reg}$ the set of elements of $\GL_n$ whose centralizer in $\GL_n$ equals $M$. We clearly have $(Z_M)_{\rm reg}\subset Z_M$.

\begin{remark}Notice that for any Levi subgroup $L$ of $\GL_n$, if $(Z_M)_{\rm reg}\cap L\neq\emptyset$, then $(Z_M)_{\rm reg}\subset L$.\label{reg}\end{remark}

\begin{remark} (i) The $r$-tuples $(C_1,\dots,C_r)\in\calC$ of type $(M,\calO_1,\dots,\calO_r)$ are all of the form $(C_{l,1},\dots,C_{l,r})$, $l\in (Z_M)_{\rm reg}(q)$, where $C_{l,i}$ is the $\GL_n(q^{d_i})$-conjugacy class of $l\calO_i$.

\noindent (ii) If there exists $f\in{\rm Aut}(\calE)$ satisfying $f(\a_i)\in C_i$ for all $i=1,\dots,r$ and if $(C_1,\dots,C_r)$ is of type $(M,\calO_1,\dots,\calO_r)$, then we can always find a $\GL_n(q)$-conjugate $M'$ of $M$ such that  $Z_{M'}\subset L_\calE$.

\end{remark}

\begin{proposition} The cardinality $
\#\{f\in{\rm Aut}(\calE)\,|\, f(\a_i)\in C_i,\, i=1,\dots,r\}$ depends only on the type of $(C_1,\dots,C_r)$.
\label{indep2}\end{proposition}

\begin{proof}Assume that $(M,\calO_1,\dots,\calO_r)$ is the type of $(C_1,\dots,C_r)$ and  that $Z_M\subset L_\calE$. Since any semisimple element of ${\rm Aut}(\calE)$ is conjugate to a semisimple element of $L_\calE(q)$, we are reduced to prove that cardinality of $\{f\in{\rm Aut}(\calE)\,|\, f_s=l,\hspace{.2cm} f(\a_i)\in C_{l,i}, i=1,\dots, r\}$, where $f_s$ denotes the semisimple part of $f$, is constant when $l$ runs over $(Z_M)_{\rm reg}(q)$. But this is clear as this cardinality equals that of $\{u(t)\in C_{{\rm Aut}(\calE)}(l)\,|\, u(\a_i)\in\calO_i, \, i=1,\dots,r\}$ and $C_{{\rm Aut}(\calE)}(l)=C_{{\rm Aut}(\calE)}(l')$ for all $l,l'\in (Z_M)_{\rm reg}$.
\end{proof}

If $(C_1,\dots,C_r)$ is of type $(M,\calO_1,\dots,\calO_r)$, we put 

$$
\calH_{(M,\calO_1,\dots,\calO_r)}^\calE(q):=\frac{\#\{f\in{\rm Aut}(\calE)\,|\, f(\a_i)\in C_i\}}{|{\rm Aut}(\calE)|}.
$$

For a Levi subgroup $L$ of $\GL_n$ defined over $\F_q$, we define the constant $K_L^o$ as follows. Since $L$ is defined over $\F_q$, there exists a multi-set $\{(e_i,n_i)\}_{i=1,\dots,s}$ of pairs of positive integers such that  $L\simeq\prod_{i=1}^s\GL_{n_i}(\overline{k})^{e_i}$ and $L(q)\simeq \prod_{i=1}^s\GL_{n_i}(q^{e_i})$. Define

$$
K_L^o=\begin{cases}(-1)^{s-1}e^{s-1}\mu(e)(e-1)!&{\rm if }\,e_i=e\, \text{ for all }i,\\
0&{\rm otherwise}.\end{cases}
$$
For a Levi subgroup $M$ of $\GL_n$ defined over $\F_q$, we put

 $$
 W_{\GL_n(q)}(M,\calO_1,\dots,\calO_r):=N_{\GL_n(q)}(M,\calO_1,\dots,\calO_r)/M(q).
 $$
 
 We have the following formula. 
 
 \begin{theorem}For any generic $r$-tuple $(\alpha_1,\dots,\alpha_r)$ of linear characters of $L_1(q^{d_1}),\dots,L_r(q^{d_r})$, we have
 
 $$
 \left\langle R_{L_1(q^{d_1}),\alpha_1}^\calE\otimes\cdots\otimes R_{L_r(q^{d_r}),\alpha_r}^\calE,1\right\rangle=
\sum_{(M,\calO_1,\dots,\calO_r)}\frac{(q-1) K_M^o\,\calH_{(M,\calO_1,\dots,\calO_r)}^\calE(q)}{|W_{\GL_n(q)}(M,\calO_1,\dots,\calO_r)|}\prod_{i=1}^rR_{L_i(q^{d_i})}^{\GL_n(q^{d_i})}(1)(l_M.\calO_i),
$$
where $l_M$ is a fixed element of $(Z_M)_{\rm reg}(q)$. 
\label{intermediate}\end{theorem}
Note that $R_{L_i(q^{d_i})}^{\GL_n(q^{d_i})}(1)(l_M.\calO_i)$ does not depend on the choice of $l_M\in(Z_M)_{\rm reg}(q)$.

Since the right hand side does not depend on the $\alpha_i$, this proves the independance of the left hand side from the choice of a generic tuple $(\alpha_1,\dots,\alpha_r)$.

\begin{proof}[Proof of Theorem \ref{intermediate}] Going through the same computation as the one performed in the proof of \cite[Theorem 4.3.1 (2)]{aha} we find that 

$$
\sum_{(C_1,\dots,C_r)}\prod_{i=1}^rR_{L_i(q^{d_i})}^{\GL_n(q^{d_i})}(\alpha_i)(C_i),
$$
where the sum runs over the set of $(C_1,\dots,C_r)\in\calC$ of fixed type $(M,\calO_1,\dots,\calO_r)$, equals

$$
\frac{(q-1) K_M^o}{|W_{\GL_n(q)}(M,\calO_1,\dots,\calO_r)|}\prod_{i=1}^rR_{L_i(q^{d_i})}^{\GL_n(q^{d_i})}(1)(\calO_i).
$$
\end{proof}
 
\subsection{Formulas in terms of Hall-Littlewood symmetric functions}

We use the notation of \S \ref{relation}. 

\subsubsection{Re-writting Proposition \ref{Hua}}\label{Re-writting}Let $\x=\{x_1,x_2,\dots\}$ be a set of infinitely many commuting variables and denote by $\Lambda(\x)$ the ring of symmetric functions in the variables $\x$ as in \cite{macdonald}. For a partition $\mu$, denote by $m_\mu=m_\mu(\x)\in\Lambda(x)$ the corresponding monomial symmetric function. Recall that the set of monomial symmetric functions $\{m_\mu\}_{\mu\in\calP}$ forms a $\Z$-basis of $\Lambda(\x)$.

Make the change of variables

$$
X_{i0}=x_{i,1},\hspace{1cm}X_{ij}=x_{i,j}^{-1}x_{i,j+1}, \hspace{2cm} i=1,\dots,r,\,\,j\in\Z,
$$
and for each $i=1,\dots,r$, denote by $\x_i$ the set of new variables $\{x_{i,1},x_{i,2},\dots\}$. For a multi-partition $\muhat=(\mu^1,\dots,\mu^r)$, put 

$$
m_\muhat:=m_{\mu^1}(\x_1)m_{\mu^2}(\x_2)\cdots m_{\mu^r}(\x_r).
$$
\bigskip

It belongs to the ring $\Lambda=\Lambda(\x_1)\otimes_\Z\Lambda(\x_2)\otimes_\Z\cdots\otimes_\Z\Lambda(\x_r)$ of functions separately symmetric in each set $\x_1,\x_2,\dots,\x_r$.

With this change of variables we have
$$
\sum_{\s}X^\s=m_\muhat,
$$
where the sum is over the set of parabolic types giving $\muhat$.

From corollaries \ref{s'} and \ref{s'bis} we deduce the following identities.

\begin{lemma} We have 

\begin{align*}
&\sum_{|\s|=n}\calM_\s^\calE(q)X^\s =\sum_{|\muhat|=n}\calM_\muhat^\calE(q) m_\muhat\\
&\sum_{|\s|=n}\calA_\s^\calE(q)X^\s =\sum_{|\muhat|=n}\calA_\muhat^\calE(q) m_\muhat\\
\end{align*}\end{lemma}

Proposition \ref{Hua} can be re-written as follows.

\begin{proposition}We have 

$$
\Log\left(1+\sum_{\m> 0}\sum_{|\muhat|={\rm rk}(\m)}\calM_\muhat^\m(q)\,m_\muhat\, Y^\m\right)=\sum_{\m>0}\sum_{|\muhat|={\rm rk}(\m)}\calA_\muhat^\m(q)\,m_\muhat\, Y^\m,$$
where the first sum (in each side) is over the set of vector bundle types.

\label{propHuafin}\end{proposition}

\subsubsection{Hall-Littlewood symmetric functions}\label{Hall}

For a partition $\lambda=(\lambda_1,\lambda_2,\dots,\lambda_t)$ of $n$, with $\lambda_1\geq\lambda_2\geq\cdots\geq\lambda_t>0$, we denote by $C_\lambda$ the unipotent conjugacy class of $\GL_n$ with $t$ Jordan blocks of size $\lambda_1,\lambda_2,\dots,\lambda_t$. 

It is well-known that for any two partitions $\lambda,\mu\in\calP_n$, there exists a polynomial $\tilde{K}_{\lambda\mu}(t)\in\Z[t]$, called the \emph{modified Kostka-Foulke polynomial}, such that for any finite field $\F_q$, the evaluation $\tilde{K}_{\lambda\mu}(q)$ coincides with the value of $\calU_\lambda$ at $C_\mu$. It is also known that the polynomials $\tilde{K}_{\lambda\mu}(t)$ have non-negative integer coefficients (see for instance \cite{Greenpoly} for a cohomological interpretation of the coefficients).

Denote by $s_\lambda(\x)\in\Lambda(\x)$ the Schur symmetric function associated with a partition $\lambda$. The set $\{s_\mu(\x)\}_{\mu\in\calP}$ forms a $\Z$-basis of $\Lambda(\x)$. We denote by $\langle\,,\,\rangle$ the Hall pairing on $\Lambda(\x)$. Recall that it makes the  basis $\{s_\mu(\x)\}_\mu$ an orthonormal basis. We define the \emph{modified Hall-Littlewood} symmetric functions $\tilde{H}_\lambda(\x;q)\in\Q(q)\otimes_\Z\Lambda(\x)$, $\lambda\in\calP$, as 

$$
\tilde{H}_\lambda(\x;q)=\sum_\mu \tilde{K}_{\lambda\mu}(q)s_\mu(\x).
$$
There are other equivalent definitions  from the theory of symmetric functions (see \cite{macdonald}) but the definition here will be sufficient for us.

\subsubsection{Main formula}

For a vector bundle type $\m$ of rank $n$ and a multi-partition $\muhat=(\mu^1,\dots,\mu^r)$ of size $n$, put

\beq
\calH_\muhat^\m(q)=\calH_{\muhat,D}^\m(q):=\frac{\#\{h\in{\rm Aut}(\calE^\m)\,|\,h(\a_i)\in C_{\mu^i}(q^{d_i})\text{ for all }i=1,2,\dots,r\}}{|{\rm Aut}(\calE^\m)|}.
\label{defH}\eeq
This is just the $\calH$-function $\calH_{(M,\calO_1,\dots,\calO_r)}^{\calE^\m}(q)$ introduced earlier with $M=\GL_n$ and $\calO_i= C_{\mu^i}(q^{d_i})$. Define 

$$
\calH_\muhat(q,Y)=\calH_{\muhat,D}(q,Y):=\sum_{{\rm rk}(\m)=n}\calH_\muhat^\m(q)Y^\m.
$$
For a multi-partition $\muhat=(\mu^1,\dots,\mu^r)\in(\calP_n)^r$,

$$
\tilde{H}_\muhat(\x_1,\dots,\x_r;q):=\prod_{i=1}^r\tilde{H}_{\mu^i}(\x_i;q^{d_i})\in\Q(q)\otimes_\Z\otimes\Lambda.
$$
We denote by $\oP$ the subset of multi-partitions in $\calP^r$ whose coordinates are all of same size. 

Put 

$$
\Omega(q)=\Omega_D(q):=1+\sum_{\muhat\in\oP\backslash\{\emptyset\}}\calH_\muhat(q,Y)\tilde{H}_\muhat(\x_1,\dots,\x_r;q).
$$

Denote by $\{h_\mu(\x)\}_{\mu\in\calP}$ the $\Z$-basis of $\Lambda(\x)$ of complete symmetric function.  It is the dual basis of monomial symmetric functions $\{m_\mu(\x)\}_{\mu\in\calP}$ with respect to the Hall pairing $\langle\,,\,\rangle$.

We extend the definition of the Hall pairing to $\Lambda$ by setting

$$
\langle a_1(\x_1)\cdots a_r(\x_r),b_1(\x_1)\cdots b_r(\x_r)\rangle=\langle a_1(\x_1),b_1(\x_1)\rangle\cdots\langle a_r(\x_r),b_r(\x_r)\rangle.
$$

\begin{theorem} For a vector bundle type $\m$ of rank $n$ and a multi-partition $\muhat$ of size $n$, we have

\begin{equation}
\calA_\muhat^\m(q)=\left\langle {\rm Coeff}\,_{Y^\m}\left[(q-1)\Log\,\Omega(q)\right],h_\muhat\right\rangle,
\label{explicitFor}\end{equation}
with $h_\muhat=h_{\mu^1}(\x_1)\cdots h_{\mu^r}(\x_r)$.

\label{theo2}\end{theorem}

The rest of this section is devoted to the proof of this theorem.

We now attached combinatorial data to the types of the tuples $(C_1,\dots,C_r)\in\calC$.
\bigskip

By \emph{type} we shall mean a function $\omhat:\Z_{>0}\times(\overline{\calP}\backslash\{\emptyset\})\rightarrow\N$ which has a finite support, i.e., the set of pairs $(e,\muhat)$ such that $\omhat(e,\muhat)\neq 0$ is finite. 
 The \emph{size} of the type $\omhat$ is defined as the non-negative integer $|\omhat|:=\sum_{(e,\muhat)}e\, \omhat(e,\muhat)$. The integers $e$ for which there exists $\muhat\in\oP$ such that $\omhat(e,\muhat)\neq 0$ are called the \emph{degrees} of $\omhat$. We denote by $\T$ the set of all types and by $\T_n$ the set of types of size $n$. 
\bigskip

The set $\T_n$ parametrizes the types of the $r$-tuples $(C_1,\dots,C_r)\in\calC$ as follows.
Assume that $(M,\calO_1,\dots,\calO_r)$ is the type of $(C_1,\dots,C_r)$. The Levi subgroup $M$ being defined over $\F_q$, there exists a unique multiset $\{(e_i,n_i)\}_{i=1,\dots,s}$ of pairs of positive integers such that $M\simeq \prod_{i=1}^s\GL_{n_i}(\overline{k})^{e_i}$ and $M(q)\simeq\prod_{i=1}^s\GL_{n_i}(q^{e_i})$. Now the unipotent conjugacy classes $\calO_1,\dots,\calO_r$ of $M(q^{d_1}),\dots, M(q^{d_r})$ are completely determined by multipartitions $\muhat_1,\dots,\muhat_s$ of $n_1,\dots,n_s$: The $i$-th coordinate of $\muhat_j$ is the partition that gives the size of the Jordan blocks of $\calO_i$ in $\GL_{n_j}$. The corresponding type is then the function in $\T_n$ that maps $(e,\muhat)$ to its multiplicity in the multiset $\{(e_i,\muhat_i)\}_{i=1,\dots,s}$.
\bigskip

If $(C_1,\dots,C_r)$ is of type $\omhat\in\T_n$, and if $\m$ is a vector bundle type of rank $n$, we define $\calH_\omhat^\m(q)$ by the right hand side of Formula (\ref{defH}) with $C_{\mu^i}(q^{d_i})$ replaced by $C_i$ (this notation is consistent with Proposition \ref{indep2}). We then put 

$$
\calH_\omhat(q,Y):=\sum_{{\rm rk}(\m)=n}\calH_\omhat^\m(q)Y^\m.
$$

\begin{proposition}We have

$$
\calH_\omhat(q,Y)=\prod_{(e,\muhat)}\calH_\muhat(q^e,Y^e)^{\omhat(e,\muhat)},
$$
where $(e,\muhat)$ runs over the set $\Z_{>0}\times(\overline{\calP}\backslash\{\emptyset\})$.
\label{Hom}\end{proposition}

\begin{proof}Consider a type $\omhat\in \T_n$ which we write (as explained above) in the form of a multiset $\{(e_i,\muhat_i)\}_{i=1,\dots,s}$. We need to check the following identity

\beq
\calH_\omhat^\m(q)=\sum_{(\v_1,\dots,\v_s)}\prod_{i=1}^s\calH_{\muhat_i}^{\v_i}(q^{e_i}),
\label{decomp}
\eeq
where the sum is over the set of $s$-tuples $(\v_1,\dots,\v_s)$ of vector bundle types, with each $\v_i$ of rank $|\muhat_i|$, such that $\sum_{i=1}^se_i\cdot \v_i=\m$ (the vector $e\cdot \v$ being the vector obtained by multiplying all coordinates of $\v$ by $e$). To see this, we fix an  $r$-tuple $(C_1,\dots,C_r)$ of type $\omhat$ and we choose a common semisimple part $l\in L_{\calE^\m}(q)$ of $C_1,\dots,C_r$ (we may assume that such a choice in $L_{\calE^\m}(q)$ exists since otherwise $\calH_\omhat^\m(q)=0$). Let $M$ be the centralizer in ${\rm Aut}(\calE)$ of $l$. Then $M$ is isomorphic to $\prod_{i=1}^s{\rm Aut}(\calE^{\v_i}_{\F_{q^{e_i}}})$ where $\calE_{\F_{q^{e_i}}}^{\v_i}$ is a rank $|\muhat_i|$  vector bundle over $\bbP^1_{\F_{q^{e_i}}}$ of type $\v_i$ with the condition that $\sum_{i=1}^se_i\cdot\v_i=\m$. From this we see that $\#\{f\in{\rm Aut}(\calE^\m)\,|\, f_s=l, f(\a_i)\in C_i\}$ divided by $|M|$ equals 

$$
\prod_{i=1}^s\calH_{\muhat_i}^{\v_i}(q^{e_i}).
$$

We obtain the identity (\ref{decomp}) by summing over all possible choices for $l$ (up to conjugation in $L_{\calE^\m}(q)$).
\end{proof}

We are now in position to prove Theorem \ref{theo2}.

For a family of symmetric functions $A_\muhat(\x_1,\dots,\x_r;q,Y)$ indexed by $\oP$ with $A_\emptyset=1$, we extend its definition to types $\omhat\in\T$ as 

$$
A_\omhat(\x_1,\dots,\x_r;q,Y)=\prod_{(e,\muhat)}A_\muhat(\x_1^e,\dots,\x_r^e;q^e,Y^e)^{\omhat(e,\muhat)},
$$
where $(e,\muhat)$ runs over $\Z_{>0}\times(\oP\backslash\{\emptyset\})$.

From the formal properties of $\Log$ (see \cite[see page 355]{aha} for more details) we see that 

$$
\Log\left(\sum_{\muhat\in\oP}A_\muhat(\x_1,\dots,\x_r;q,Y)T^{|\muhat|}\right)=\sum_{\omhat\in\T}C_\omhat^oA_\omhat(\x_1,\dots,\x_r;q,Y)T^{|\omhat|},
$$
where 

$$
C_\omhat^o:=\frac{\mu(e)}{e}(-1)^{\ell(\omhat)-1}\frac{(\ell(\omhat)-1)!}{\prod_\muhat \omhat(e,\muhat)!},
$$
with $\ell(\omhat):=\sum_{(e,\muhat)}\omhat(e,\muhat)$, if $e$ is the only degree of $\omhat$ and $C_\omhat^o=0$ otherwise.

Theorem \ref{theo2} reduces thus to the following statement.

\begin{equation}
\calA_\muhat^\m(q)=(q-1)\sum_{\omhat\in\T}C_\omhat^o\calH_\omhat^\m(q)\left\langle \tilde{H}_\omhat(\x_1,\dots,\x_r;q),h_\muhat\right\rangle,
\label{theo2bis}\end{equation}
which  by Theorem \ref{theo} reduces to 

\begin{equation}
\frac{1}{|{\rm Aut}(\calE)|}\sum_{h\in{\rm Aut}(\calE)}\prod_{i=1}^rR_{L_i(q^{d_i})}^{\GL_n(q^{d_i})}(\alpha_i)(h(\a_i))=(q-1)\sum_{\omhat\in\T}C_\omhat^o\calH_\omhat^\m(q)\left\langle \tilde{H}_\omhat(\x_1,\dots,\x_r;q),h_\muhat\right\rangle.
\label{formula0}
\end{equation}

This is a consequence of Theorem \ref{intermediate}. Indeed from \cite[Corollary 2.2.3]{aha2} we have 

\begin{equation}
\left\langle \tilde{H}_\omhat(\x_1,\dots,\x_r;q),h_\muhat\right\rangle=\prod_{i=1}^rR_{L_i(q^{d_i})}^{\GL_n(q^{d_i})}(1)(l_M.\calO_i)
\label{eqq}\end{equation}
whenever $(M,\calO_1,\dots,\calO_r)$ corresponds to the type $\omhat$, and $L_1,\dots,L_r$ are standard Levi subgroups of $\GL_n$ whose block sizes are given by the parts of the coordinates of  the multi-partition $\muhat$. Moreover

$$
C_\omhat^o=\frac{K_M^o}{|W_{\GL_n(q)}(M,\calO_1,\dots,\calO_r)|}.
$$
Indeed, the group $W_{\GL_n(q)}(M,\calO_1,\dots,\calO_r)$ is isomorphic to 

$$
\prod_{(e,\muhat)}\left((\Z/e\Z)^{\omhat(e,\muhat)}\rtimes \calS_{\omhat(e,\muhat)}\right).
$$

\subsection{Explicit computation}\label{explicitcomp}
To compute explicitly $\calA_\muhat^\m(q)$ from Formula (\ref{explicitFor}) we need, as shown by Formula (\ref{theo2bis}), to compute the following two quantities

$$
\left\langle \tilde{H}_\omhat(\x_1,\dots,\x_r;q),h_\muhat\right\rangle\hspace{.5cm} \text{ and }\hspace{.5cm}\calH_\omhat^\m(q).
$$
The computation of the first one reduces to the computation  of the modified Kostka-Foulke polynomials  (see \S \ref{Hall}) which can be done by a combinatorial algorithm. The problem of computing Kostka-Foulke polynomials is very well-known and can be found in many places in the literature.

The computation of the quantities $\calH_\omhat^\m(q)$  does not seem to appear in the literature. By Proposition \ref{Hom} it reduces to  the unipotent case, namely to the computation of $\calH_{\muhat}^\m(q)$ defined by Formula (\ref{defH}) with $\muhat$ a multi-partition. Working with unipotent conjugacy classes or nilpotent orbits is equivalent, therefore we are reduced to understand the quantities

$$
\calN_\muhat^\calE(q)=\calN_{\muhat,D}^\calE(q):=\#\{ h\in{\rm End}(\calE)\,|\, h(\a_i)\in N_{\mu^i}\text{ for all }i=1,\dots,r\},
$$
where $\calE$ is rank $n$ vector bundle on $\bbP^1_k$, $\muhat=(\mu^1,\dots,\mu^r)\in(\calP_n)^r$ and where for a partition $\lambda$ of $n$, $N_\lambda$ denotes the nilpotent orbit of $\gl_n$ with Jordan blocks of size given by the parts of $\lambda$. 

A nilpotent orbit $N$ of $\gl_n$  is completely determined by the sequence $n_1,n_2,\dots$ of non-negative integers such that if $A\in N$ then $n_i$ is the rank of $A^i$. Since ${\rm End}(\calE)$ is explicit we have an obvious algorithm for computing the terms $\calN_\muhat^\calE(q)$.

We have the following conjecture :

\begin{conjecture}Let $\m=(m_1,\dots,m_f)\in \N^f$ such that $\sum_{i=1}^fm_i=n$. For ${\bf d}=(d_1,\dots,d_r)\in\N^r$ and $\muhat=(\mu^1,\dots,\mu^r)\in(\calP_n)^r$ there exists a polynomial  $N_{\muhat; {\bf d}}^\m(t)\in\Q[t]$ such that for any finite fields $\F_q$, and any divisor $D=\sum_{i=1}^r\a_i$ on $\bbP^1_{\F_q}$ with ${\rm deg}(\a_i)=d_i$ we have

$$
N_{\muhat; {\bf d}}^\m(q)=\calN_{\muhat,D}^\calE(q),
$$
where $\calE$ is a vector bundle on $\bbP^1_{\F_q}$ of type $\m$.
\label{conjHnil}\end{conjecture}

A priori, the above conjecture implies that the right hand side of Formula (\ref{theo2bis}) is only a rational function in $q$, but since $\calA_\muhat^\m(q)$ counts also the number of points of  a constructible set over finite fields (see \S \ref{constructible}), it must be a polynomial with integer coefficients. Therefore Conjecture \ref{conjHnil} implies the following one.
\bigskip

\begin{conjecture}Let $\m=(m_1,\dots,m_f)\in \N^f$ such that $\sum_{i=1}^fm_i=n$. For ${\bf d}=(d_1,\dots,d_r)\in\N^r$ and $\muhat=(\mu^1,\dots,\mu^r)\in(\calP_n)^r$ there exists a polynomial $A_{\muhat; {\bf d}}^\m(t)\in\Z[t]$ such that for any finite fields $\F_q$, and any divisor $D=\sum_{i=1}^r\a_i$ on $\bbP^1_{\F_q}$ with ${\rm deg}(\a_i)=d_i$ we have

$$
A_{\muhat; {\bf d}}^\m(q)=\calA_{\muhat,D}^\calE(q),
$$
where $\calE$ is a vector bundle on $\bbP^1_{\F_q}$ of type $\m$.
\label{polynomiality}\end{conjecture}

\bigskip

It is easy to compute explicitly $\calN_\muhat^\calE(q)$ and verify Conjecture \ref{conjHnil} in the following cases :

(1) $\calE=\calO(a)^n$, i.e. ${\rm Aut}(\calE)=\GL_n$.

(2) for each $i=1,\dots,r$, the partition $\mu^i$ is either the trivial partition $(1^n)$ (which corresponds to the zero orbits of $\gl_n$) or the partition $(n^1)$ (which corresponds to the regular nilpotent orbit).

We have the following proposition.

\begin{proposition} Conjecture \ref{conjHnil} is true if for all $i=1,\dots, f-1$, we have $b_i-b_{i+1}+1\geq \sum_{s=1}^r d_s$.
\label{notinteresting}\end{proposition}

The above proposition is not very interesting for us as in this case, there is no geometrically indecomposable parabolic structure on $\calE$ by  Proposition \ref{DS}, however it is instructive to see why the conjecture is true in this case.

\begin{proof}[Proof of Proposition \ref{notinteresting}]

 We use the notation of \S \ref{aut}, and we put $\l_\calE={\rm Lie}(L_\calE)$, $\fraku_\calE={\rm Lie}(U_\calE)$. We see under the assumption of Proposition \ref{notinteresting}  the evaluation map 
$$
\varepsilon:{\rm End}(\calE)\rightarrow \l_\calE(q)\times \left(\fraku_\calE(q^{d_1})\times\cdots\times \fraku_\calE(q^{d_r})\right)
$$
$h\mapsto (h(\a_1),\dots,h(\a_r))$
is surjective with fibers all of same cardinality $q^c$ for some explicit non-negative integer $c$. Therefore we are reduced to see that the set 

$$
\left\{(l,u_1,\dots,u_r)\in\l_\calE(q)\times \fraku_\calE(q^{d_1})\times\cdots\times\fraku_\calE(q^{d_r})\,|\, l+u_i\in N_{\mu^i}(q^{d_i}) \text{ for all }i=1,\dots,r\right\}.
$$
has polynomial count. The equation $lu_i\in N_{\mu^i}$ with $l\in\l_\calE$ and $u\in\fraku_\calE$ forces $l$ to be nilpotent.  The set of nilpotent elements of $\l_\calE$ is the  (finite) union of its nilpotent orbits and so we are reduced to prove that for any nilpotent orbit $C$ of $\l_\calE(q)$, the cardinality of the set

$$
\left\{(l,u_1,\dots,u_r)\in C\times \fraku_\calE(q^{d_1})\times\cdots\times\fraku_\calE(q^{d_r})\,|\, l+u_i\in N_{\mu^i}(q^{d_i}) \text{ for all }i=1,\dots,r\right\}
$$
is a polynomial in $q$.

Choosing a representative $l$ of $C$, the cardinality of the above set equals

$$
|C|\cdot \# \left\{(u_1,\dots,u_r)\in\fraku_\calE(q^{d_1})\times\cdots\times \fraku_\calE(q^{d_r})\,|\,  lu_i\in N_{\mu^i}(q^{d_i}) \text{ for all }i=1,\dots,r\right\} =
$$
$$
|C|\cdot\prod_{i=1}^r\#\left\{u_i\in\fraku_\calE(q^{d_i})\,|\, l+u_i\in N_{\mu^i}(q^{d_i})\right\}.
$$
Moreover for each $i=1,\dots,r$

$$
\#\left\{u_i\in\fraku_\calE(q^{d_i})\,|\, l+u_i\in N_{\mu^i}(q^{d_i})\right\}=\frac{\#\left\{(t,u)\in C(q^{d_i})\times \fraku_\calE(q^{d_i})\,|\, t+u_i\in N_{\mu^i}(q^{d_i})\right\}}{|C(q^{d_i})|}.
$$

We are reduced to prove that if $P=L\ltimes U_P$ is a Levi decomposition of a parabolic subgroup of $\GL_n$ defined over $\F_q$  (with Lie algebra decomposition $\frakp=\l\oplus\fraku_P$), then for any nilpotent orbit  $N$ of $\gl_n(q)$ and any nilpotent orbit $C$ of $\l(q)$,  the cardinality of 

$$
\{(t,u)\in C\times \mathfrak{u}_P\,|\, t+u\in N\}=(C+\mathfrak{u}_P)\cap N.
$$
is a polynomial in $q$. Consider the unipotent conjugacy classes $U=N+1$ and $\calO=C+1$. For $x\in U$, we have the formula 

\begin{align*}
R_{L(q)}^{\GL_n(q)}(1_\calO)(x)&=\frac{\#\left\{g\in\GL_n(q)\,|\, g^{-1}xg\in \calO U_P\right\}}{|P(q)|}\\
&=\frac{|\calO U_P\cap U|\cdot |\GL_n(q)|}{|P(q)|\cdot|U|}\\
&=\frac{|(C+\fraku_P)\cap N|\cdot |\GL_n(q)|}{|P(q)|\cdot|N|}\\
\end{align*}
where $1_\calO$ is the characteristic function of $\calO$ (it takes the value $1$ on $\calO$ and zero elsewhere). From the \emph{character formula} for $R_{L(q)}^{\GL_n(q)}(1_\calO)$ (see for instance \cite[Proposition 12.2]{DM}) we see that the computation of $R_{L(q)}^{\GL_n(q)}(1_\calO)(x)$ reduces to the computation of the so-called (two-variable) \emph{Green functions} (see \cite[Definition 12.1]{DM}). Because we are in $\GL_n$, the computation of the values of the two-variable Green functions reduces to the computation of the  \emph{Green polynomials}  which are linear combination  of Kostka-Foulke polynomials \cite[III, \S 7, (7.11)]{macdonald}. We therefore deduce that the cardinality of $(C+\mathfrak{u}_P)\cap N$ is a polynomial in $q$. 
\end{proof}

\begin{remark}In the situation of Proposition \ref{notinteresting}, it is possible, as suggested by the proof, to write  an explicit formula for $\calN_{\muhat,D}^\calE(q)$ in terms of Kostka-Foulke polynomials. 
\end{remark}

We now prove Conjecture \ref{conjHnil} in a more interesting case.

\begin{theorem} Conjecture \ref{conjHnil} is true if $\m=(m_1,m_2)$, i.e. if $\calE$ is of the form $\calO(b_1)^{m_1}\oplus\calO(b_2)^{m_2}$.
\end{theorem}

\begin{proof}For simplicity we write the proof for $\m=(n-1,1)$, i.e. $\calE=\calO(b_1)^{n-1}\oplus\calO(b_2)$, and $d_1=\cdots=d_r=1$. 

Define

$$
Z_\muhat^\calE:=\{ h\in{\rm End}(\calE)\,|\, h(\a_i)\in N_{\mu^i}\text{ for all }i=1,\dots,r\},
$$
and put $m:=b_1-b_2$ and denote by $k[t]^{\leq m}$ the $k$-subspace of $k[t]$ of polynomials of degree $\leq m$. We have $L_\calE=\GL_{n-1}\times\GL_1\subset \GL_n$ and  put $\l_\calE:={\rm Lie}(L_\calE)$.

Then 

$$
{\rm End}(\calE)\simeq \l_\calE\oplus \left(k[t]^{\leq m}\right)^{n-1},
$$
and 

$$
{\rm End}(\calE)_{\rm nil}\simeq (\gl_{n-1}(k))_{\rm nil}\times \left(k[t]^{\leq m}\right)^{n-1}.
$$
Clearly we have 

$$
Z_\muhat^\calE=\coprod_C Z_{\muhat,C}^\calE,
$$
where the union runs over the set of nilpotent orbits of $\gl_{n-1}(k)$ and 

$$
Z_{\muhat,C}^\calE:=\left\{ (A,X(t))\in C\times \left(k[t]^{\leq m}\right)^{n-1}\,|\, M_{A,X(\a_i)}\in N_{\mu^i}\text{ for all }i=1,\dots,r\right\},
$$
with 

$$
M_{A,X(\a_i)}=\left(\begin{array}{ccc}A&\vline& X(\a	_i) \\ \hline 0&\vline &0\end{array}\right).
$$
We need to prove that  $|Z_{\muhat,C}^\calE|$ is a polynomial in $q$ that does not depend on the position of the points $\a_i$'s.

We fix a nilpotent orbit $C$ of $\gl_{n-1}(k)$ and a Jordan form matrix $A\in C$. Then 

$$
|Z_{\muhat,C}^\calE|=|C|\cdot \#\left\{X(t)\in \left(k[t]^{\leq m}\right)^{n-1}\,|\, M_{A,X(\a_i)}\in N_{\mu^i}\text{ for all }i=1,\dots,r\right\}.
$$
Since $|C|$ is a polynomial in $q$ we only care about the right factor. The nilpotent orbit $N_{\mu^i}$ is characterized by the strictly decreasing sequence $n_{i,1}>n_{i,2}>\dots>n_{i,s_i}=0$ such that if $B\in N_{\mu^i}$, then $n_{i,s}$ is the rank of $B^s$ for all $s=1,\dots,s_i$. This sequence can be read off directly from the partition $\mu^i$ as $n-l_{i,1}, n-l_{i,1}-l_{i,2},\dots$ where $l_{i,1},l_{i,2},\dots, l_{i,s_i}$ are the lengths of the columns of the Young diagram of $\mu^i$. 

Saying that $M_{A,X(\a_i)}$ is in $N_{\mu^i}$ is equivalent to saying that the rank of $\left(M_{A,X(\a_i)}\right)^s$ equals $n_{i,s}$ for all $s=1,2,\dots, s_i$. But $\left(M_{A,X(\a_i)}\right)^s=M_{A^s,A^{s-1}X(\a_i)}$ and the rank $h_s$ of $A^s$ is known from $C$. Therefore for each $s$, saying that the rank of 
$\left(M_{A,X(\a_i)}\right)^s$ equals $n_{i,s}$ if and only if we are in one of the following two stuations

\begin{equation}
\begin{cases}A^{s-1}X(\a_i)\in {\rm Im}(A^s)&\text{ and }n_{i,s}=h_s.\\
A^{s-1}X(\a_i)\notin  {\rm Im}(A^s)&\text{ and }n_{i,s}=h_s+1.\end{cases}
\label{sub/comp}
\end{equation}
For $i=1,\dots,r$, consider the $i$-th evaluation map 

$$
\varepsilon_i:\left(k[t]^{\leq m}\right)^{n-1}\rightarrow (k^{n-1})^{s_i-1}
$$
given by $\varepsilon_i(X(t))=(X(\a_i),AX(\a_i),\dots,A^{s_i-1}X(\a_i))$, and define the $k$-linear map

$$
\varepsilon=(\varepsilon_1,\dots,\varepsilon_r):\left(k[t]^{\leq m}\right)^{n-1}\rightarrow (k^{n-1})^{s_1-1}\times\cdots\times (k^{n-1})^{s_r-1}.
$$
The fibers $\varepsilon^{-1}(x)$, with $x\in {\rm Im}(\varepsilon)$, are all of same cardinality  $q^c$ for some explicit non-negative integer $c$ which is independent from $q$. Therefore, for any $k$-vector subspace $W$ of  $(k^{n-1})^{s_1-1}\times\cdots\times (k^{n-1})^{s_r-1}$, the cardinality of $\varepsilon^{-1}(W)$ is a polynomial in $q$. On the other hand, if not empty, by (\ref{sub/comp}) the set

$$
X_{A,\muhat}:=\left\{X(t)\in \left(k[t]^{\leq m}\right)^{n-1}\,|\, M_{A,X(\a_i)}\in N_{\mu^i}\text{ for all }i=1,\dots,r\right\}
$$
is the inverse image by $\varepsilon$ of a set of the form $\prod_{i=1}^r\prod_{j=1}^{s_i-1}A_i$ with $A_i$ a $k$-vector subspace of $k^{n-1}$ or the complementary in $k^{n-1}$ of a $k$-vector subspace. We can therefore express the cardinality of $X_{A,\muhat}$ as a simple combination  of cardinalities of inverse images by $\varepsilon$ of $k$-vector subspaces of $(k^{n-1})^{s_1-1}\times\cdots\times (k^{n-1})^{s_r-1}$. 
\end{proof}

\begin{corollary}Conjecture \ref{polynomiality} is true if $\m=(m_1,m_2)$, i.e. if $\calE=\calO(b_1)^{m_1}\oplus\calO(b_2)^{m_2}$.
\label{polynomialitymm}\end{corollary}

\subsection{Example : Rank-two parabolic bundles}\label{rank2}

In Example \ref{exrk2} we gave a necessary and sufficient condition for  $\calE_{a,b}=\calO(a)\oplus\calO(b)$, with $a\geq b$ to support geometrically indecomposable borelic structures. In this section we compute explicitly the number $\calA^{\calE_{a,b}}_{\rm bor}(q)$ of isomorphism classes of indecomposable borelic structures on $\calE_{a,b}$ using Formula (\ref{theo2bis}). 

For a subset $I$ of $\{1,\dots,r\}$, we put $\mu_I$ the multi-partition $(\mu_1,\dots,\mu_r)$ with $\mu_i=(1^2)$ if $i\in I$ and $\mu_i=(2^1)$ if $i\notin I$. Then the possible types of size $2$ are $(1,(1,\dots,1))^2$, $(2,(1,\dots,1))$ and $(1,\mu_I)$ where $I$ runs over the subsets of $\{1,\dots,r\}$.

We compute Formula (\ref{theo2bis}) using the two following tables.

\begin{equation}
\label{table}
\begin{array}{|c|c|c|}
\hline
(a,b)&\omhat&\calH_\omhat^{\calE_{a,b}}(q)\\
\hline
a=b&(1,(1,\dots,1))^2&\frac{1}{(q-1)^2}  \\
a=b&(2,(1,\dots,1))&\frac{1}{q^2-1} \\
a=b&(1,\mu_I)&\begin{cases} \frac{1}{q(q-1)} &\text{ if }I=\emptyset \\ \frac{1}{q(q-1)^2(q+1)}&\text{ if }I=\{1,\dots,r\}\\0&\text{ otherwise.}\end{cases} \\
a>b&(1,(1,\dots,1))^2&\frac{2}{(q-1)^2} \\
a>b&(2,(1,\dots,1))&0\\
a>b&(1,\mu_I)&\frac{\#\{P(t)\in k[t]\,|\, {\rm deg}\,P\leq a-b,\, P(\a_i)=0 \text{ for all }i\in I,\, P(\a_i)\neq 0 \text{ for all }i\notin I\}}{(q-1)^2q^{a-b+1}}\\
\hline
\end{array}
\end{equation}

Since we are working with borelic structures, we apply Formula (\ref{theo2bis}) with  $\muhat=((1^2),\dots,(1^2))$.

\begin{equation}
\begin{array}{|c|c|c|}
\hline
\omhat&C_\omhat^o&\left\langle \tilde{H}_\omhat(\x_1,\dots,\x_r;q),h_\muhat\right\rangle\\
\hline
(1,(1,\dots,1))^2&-\frac{1}{2}&2^r\\
(2,(1,\dots,1))&-\frac{1}{2}&\begin{cases}0&\text{ if }2\nmid d_i\text{ for some }i\\ 2^r & \text{ if } 2\mid d_i\text{ for all } i\end{cases} \\
(1,\mu_I)&1&\prod_{i=1}^r\begin{cases}(q^{d_i}+1)&\text{ if }\mu_i=(1^2)\\ 1& \text{ if }\mu_i=(2^1)\end{cases}\\
\hline
\end{array}
\end{equation}

\subsubsection{The case $a=b$}

From the tables  we find

\begin{proposition}
$$
\calA^{\calE_{a,a}}_{\rm bor}(q)=\begin{cases}\frac{-2^{r-1}}{q-1}+\frac{1}{q}+\frac{1}{q(q^2-1)}\prod_{i=1}^r(q^{d_i}+1) & \text{ if } 2\nmid d_i \text{ for some }i,\\
\frac{-2^{r-1}}{q-1}+\frac{1}{q}+\frac{-2^{r-1}}{q+1}+\frac{1}{q(q^2-1)}\prod_{i=1}^r(q^{d_i}+1) & \text{ if } 2\mid d_i \text{ for all }i.\end{cases}
$$
\label{proprk1}
\end{proposition}

We see from the above formula that $\calA_{\rm bor}^{\calE_{a,a}}(q)$ depends only on the partition $\lambda$ of $l=\sum_{i=1}^rd_i$ whose parts are $d_1,\dots,d_r$ and not on the ordering of the $d_i$'s. In the following, for the sake of clarity we will write $\calA_{{\rm bor},\lambda}^{\calE_{a,a}}(q)$ instead of $\calA_{\rm bor}^{\calE_{a,a}}(q)$. 

For  an integer $0< m\leq l$ denote by $X^l_m$ the set of subsets of $\{1,\dots,l\}$ of size $m$. The symmetric group $\mathfrak{S}_l$ acts naturally on $X^l_m$ and we denote by $\chi^l_m$ the character of the representation of $\mathfrak{S}_l$ in the $\C$-vector space freely generated by $X^l_m$. For $l\geq 3$, consider the polynomial

$$
A_{\rm bor}^{(a,a)}(t)=\sum_{m=0}^{l-3}\left(\sum_{i=1}^{\left[\frac{l-m-1}{2}\right]}\chi^l_{m+2i+1}\right)t^m,
$$
with coefficients in the character ring of $\frak{S}_l$. For $w\in\frak{S}_l$, put

$$
A_{{\rm bor},w}^{(a,a)}(t):=\sum_{m=0}^{l-3}\left(\sum_{i=1}^{\left[\frac{l-m-1}{2}\right]}\chi^l_{m+2i+1}(w)\right)t^m\in \N[t].
$$

\begin{proposition} Let $w\in \frak{S}_l$ be an element of cycle-type $\lambda$, then 

$$
\calA_{{\rm bor},\lambda}^{\calE_{a,a}}(q)=\begin{cases}A_{\rm bor,w}^{(a,a)}(q)&\text{ if }l\geq 3,\\0&\text{ otherwise.}\end{cases}
$$

\label{proprk2}\end{proposition}
\begin{proof} It follows by calculation from the formulas of Proposition \ref{proprk1}.
\end{proof}

\begin{example}For $l=4$ we have $A_{{\rm bor}}^{(a,a)}(t)=t+\chi^4_1$, with $\chi^4_1$ the character of the standard representation of $\frak{S}_4$.  On the other hand  we have the following table.
\begin{equation}
\label{c}
\begin{array}{|c|c|c|}
\hline
r&\lambda=(d_1,\dots,d_r)&\calA_{{\rm bor},\lambda}^{\calE_{a,a}}(q)\\
\hline
4&(1,1,1,1)&  q+4\\
3&(2,1,1)&q+2 \\
2&(3,1)& q+1 \\
2&(2,2)&q \\
1&(4)&q\\
\hline
\end{array}
\end{equation}
\end{example}

\subsubsection{The case $a>b$}

We have 

\begin{equation}
\calA_{\rm bor}^{\calE_{a,b}}(q)=-\frac{2^r}{q-1}+(q-1)\sum_{I\subseteq\{1,\dots,r\}}\calH_{\mu_I}^{\calE_{a,b}}(q)\prod_{i=1}^r\begin{cases}(q^{d_i}+1)&\text{ if }\mu_i=(1^2)\\1&\text{ if }\mu_i=(2^1)\end{cases}.
\label{eqa-b}\end{equation}

Note that the right hand side depends only on the partition $\lambda$ of $l=\sum_{i=1}^rd_i$ with parts $d_1,\dots,d_r$. In the following we use the notation $\calA_{{\rm bor},\lambda}^{\calE_{a,b}}(q)$ instead of $\calA_{\rm bor}^{\calE_{a,b}}(q)$. 

For $l\geq a-b+2$, define 

$$
A_{\rm bor}^{(a,b)}(t)=\sum_{m=0}^{l-(a-b+2)}\left(\sum_{s=m+a-b+2}^l\chi^l_s\right)t^m,
$$
and for $w\in\frak{S}_l$, put $A_{{\rm bor},w}^{(a,b)}(t)=\sum_{m=0}^{l-(a-b+2)}\left(\sum_{s=m+a-b+2}^l\chi^l_s(w)\right)t^m$.

\begin{theorem} Let  $w\in\frak{S}_l$ an element of cycle-type $\lambda$, we have 

$$
\calA_{{\bor},\lambda}^{\calE_{a,b}}(q)=\begin{cases}A_{{\rm bor},w}^{(a,b)}(q)&\text{ if }l\geq a-b+2,\\0&\text{ otherwise}.\end{cases}
$$
\label{theork2}\end{theorem}

Before proving Theorem \ref{theork2}, let us give some consequences.

\begin{corollary} If $\lambda=(l)$ (i.e. $r=1$), we have $$\calA_{{\rm bor},\lambda}^{\calE_{a,b}}(q)=\begin{cases}\sum_{m=0}^{l-(a-b+2)}q^m&\text{ if }l\geq a-b+2\\0&\text{ otherwise}. \end{cases}$$
\end{corollary}

\begin{theorem}For any integer $d$ and $w\in\frak{S}_l$ of cycle-type $\lambda$, we have 

$$
\sum_{a+b=d}\calA_{{\rm bor},\lambda}^{\calE_{a,b}}(q)=\begin{cases}\sum_{m=0}^{l-3}\sum_{a=1}^{\left[\frac{l-m-1}{2}\right]}\left(\sum_{s=m+2a+1}^l\chi^l_s(w)\right)\, q^m&\text{ if } l\geq 3,\\
0&\text{ otherwise}.\end{cases}
$$In particular this sum does not depend on $d$.
\label{sumind}\end{theorem}
\begin{proof}We need to see that 

$$
\sum_{a+b=0}\calA_{{\rm bor},\lambda}^{\calE_{a,b}}(q)=\sum_{a+b=1}\calA_{{\rm bor},\lambda}^{\calE_{a,b}}(q).
$$
It is easy to see from Theorem \ref{theork2} that  

$$
\sum_{a+b=1}\calA_{{\rm bor},\lambda}^{\calE_{a,b}}(q)=\sum_{m=0}^{l-3}\sum_{a=1}^{\left[\frac{l-m-1}{2}\right]}\left(\sum_{s=m+2a+1}^l\chi^l_s(w)\right)\, q^m.
$$
Now 

$$
\sum_{a+b=0}\calA_{{\rm bor},\lambda}^{\calE_{a,b}}(q)=\calA_{{\rm bor},\lambda}^{\calE_{0,0}}(q)+\sum_{a\geq 1}\calA_{{\rm bor},\lambda}^{\calE_{a,b}}(q),
$$
The term $\calA_{{\rm bor},\lambda}^{\calE_{0,0}}(q)$ is computed in Proposition \ref{proprk2}
and from Theorem \ref{theork2} we find that 

$$
\sum_{a\geq 1}\calA_{{\rm bor},\lambda}^{\calE_{a,b}}(q)=\sum_{m=0}^{l-4}\sum_{a=1}^{\left[\frac{l-m-2}{2}\right]}\left(\sum_{s=m+2a+1}^l\chi^l_s(w)\right)\, q^m.
$$
\end{proof}

\begin{proof}[Proof of Theorem \ref{theork2}] We prove it for $\lambda=(1^l)$, i.e., $l=r$ and $d_1=\cdots=d_r=1$. We assume that $a-b+2\leq r$ as otherwise we have $\calA_{{\bor},(1^r)}^{\calE_{a,b}}(q)=0$ by Lemma \ref{existence}. For $s\leq r$, we denote by $h_s(q)$ the cardinality of the set $H_s$ of polynomials of degree less or equal to $a-b$ that vanishes exactly at $\a_1,\dots,\a_s$ in $\{\a_1,\dots,\a_r\}$ (the polynomials may vanish outside the later set). We also denote by $\overline{h}_s(q)$ the cardinality of the set $\overline{H}_s$ of polynomials of degree less or equal to $a-b$ that vanish at least at $\a_1,\dots,\a_s$. Note that the quantities $h_s(q)$ and $\overline{h}_s(q)$ do not depend on the position of $\a_1,\dots,\a_r$.  For $k=s+1,\dots,r$, denote by $\overline{H}_{s,k}$ the set of polynomials of degree less or equal to $a-b$ that vanish at least at $\a_1,\dots,\a_s,\a_k$. Then

$$
H_s=\overline{H}_s\backslash\bigcup_{k=s+1}^r\overline{H}_{s,k}.
$$
By the inclusion-exclusion principle we have 

$$
h_s(q)=\sum_{k=0}^{r-s}(-1)^k{r-s \choose k}\, \overline{h}_{s+k}(q).
$$

Then

\begin{align*}
\calA_{{\rm bor},(1^r)}^{\calE_{a,b}}(q)&=-\frac{2^r}{q-1}+(q-1)\sum_{s=0}^r\left(\sum_{|I|=s}\mu_I(q)\right)(q+1)^s\\
&=-\frac{2^r}{q-1}+\frac{1}{(q-1)q^{a-b+1}}\sum_{s=0}^r{r \choose s}\, h_s(q) (q+1)^s\\
&=-\frac{2^r}{q-1}+\frac{1}{(q-1)q^{a-b+1}}\sum_{s=0}^r{r \choose s}\, \sum_{k=0}^{r-s}(-1)^k{r-s \choose k}\, \overline{h}_{s+k}(q) (q+1)^s\\
&=-\frac{2^r}{q-1}+\frac{1}{(q-1)q^{a-b+1}}\sum_{s=0}^r\sum_{k=0}^{r-s}\sum_{j=0}^s(-1)^k{r\choose s}{ r-s \choose k}{s \choose j}q^j\, \overline{h}_{s+k}(q)\\
&=-\frac{2^r}{q-1}+\frac{1}{(q-1)q^{a-b+1}}\sum_{s=0}^r\sum_{n=s}^r\sum_{j=0}^s(-1)^{n-s}{r\choose s}{ r-s \choose n-s}{s \choose j}q^j\, \overline{h}_n(q)\\
&=-\frac{2^r}{q-1}+\frac{1}{(q-1)q^{a-b+1}}\sum_{n=0}^r\sum_{s=0}^n\sum_{j=0}^s(-1)^{n-s}{r\choose s}{ r-s \choose n-s}{s \choose j}q^j\, \overline{h}_n(q)\\
&=-\frac{2^r}{q-1}+\frac{1}{(q-1)q^{a-b+1}}\sum_{n=0}^r\sum_{j=0}^n\sum_{s=j}^n(-1)^{n-s}{r\choose s}{ r-s \choose n-s}{s \choose j}q^j\, \overline{h}_n(q)
\end{align*}
From the identities

$$
{r-s \choose n-s}{r \choose s}={ n\choose s}{r \choose n},\hspace{1cm} {n\choose s}{s\choose j}={n\choose j}{n-j\choose s-j},
$$
we find

\begin{align*}
\calA_{{\rm bor},(1^r)}^{\calE_{a,b}}(q)&=-\frac{2^r}{q-1}+\frac{1}{(q-1)q^{a-b+1}}\sum_{n=0}^r\sum_{j=0}^n{r\choose n}{n \choose j}\sum_{s=j}^n(-1)^{n-s}{ n-j \choose s-j}q^j\, \overline{h}_n(q)\\
&=-\frac{2^r}{q-1}+\frac{1}{(q-1)q^{a-b+1}}\sum_{n=0}^r\sum_{j=0}^n{r\choose n}{n \choose j}\sum_{s=0}^{n-j}(-1)^{n-j-s}{ n-j \choose s}q^j\, \overline{h}_n(q)
\end{align*}
But 

$$
\sum_{s=0}^{n-j}(-1)^{n-j-s}{ n-j \choose s}=0
$$
unless $n=j$ in which case it equals $1$.

Hence
$$
\calA_{{\rm bor},(1^r)}^{\calE_{a,b}}(q)=-\frac{2^r}{q-1}+\frac{1}{(q-1)q^{a-b+1}}\sum_{n=0}^r{r\choose n}q^n\, \overline{h}_n(q)
$$
Noticing that $\overline{h}_n(q)=1$ if $a-b-n+1\leq 0$ and $\overline{h}_n(q)=q^{a-b-n+1}$ otherwise, we deduce that

\begin{align*}
\calA_{{\rm bor},(1^r)}^{\calE_{a,b}}(q)&=\frac{1}{q-1}\left(-2^r+\sum_{n=0}^{a-b}{r\choose n}+\sum_{n=a-b+1}^r{ r \choose n}q^{n-(a-b+1)}\right)\\
&=\sum_{m=0}^{r-(a-b+2)}\sum_{s=m+a-b+2}^r{r\choose s}\, q^m.
\end{align*}

\end{proof}

\section{Connection with the moduli space of Higgs bundles}

\subsection{Higgs bundles with prescribed residues}\label{Higgs}
Unless specified $k$ is a finite field $\F_q$ or an algebraically closed field, $D=\a_1+\dots+\a_r$ a reduced divisor on $\bbP^1_k$  as before and $d_i$ is the degree of $\a_i$. A \emph{Higgs bundle} on $\bbP^1_k$ is a pair $(\calE,\varphi)$ with $\calE$ a vector bundle on $\bbP^1_k$ and $\varphi:\calE\rightarrow \calE\otimes \Omega^1(D)$. The map $\varphi$ is called a \emph{Higgs field}. We fix an $r$-tuple ${\bf A}=(A_1,\dots,A_r)$ of adjoint orbits of $\gl_n(k(\a_1)),\dots,\gl_n(k(\a_r))$ respectively.

We are interested in the space 

$$
X_{\bf A}^\calE=X_{{\bf A},D}^\calE:=\left\{\varphi:\calE\rightarrow\calE\otimes\Omega^1(D)\,\left|\, {\rm Res}_{\a_i}(\varphi)\in A_i \text{ for all }i=1,\dots,r\right.\right\}.
$$
It can be explicitly described by equations. To see this, write $\calE=\bigoplus_{i=1}^f\calO(b_i)^{m_i}$ with $b_1> b_2>\cdots>b_f$ and identify ${\rm End}(\calE)$ with $\prod_{1\leq u\leq v\leq f}{\rm Mat}_{m_u,m_v}\left(k[t]^{\leq b_u-b_v}\right)$ where $k[t]^{\leq b_u-b_v}$ is the subspace of $k[t]$ of polynomials of degree smaller or equal to $b_u-b_v$. For $f\in{\rm End}(\calE)$ denote by $f_{u,v}(t)$ its coordinate in ${\rm Mat}_{m_u,m_v}\left(k[t]^{\leq b_u-b_v}\right)$ and for $y\in\gl_n(k)\simeq \bigoplus_{1\leq u,v\leq f}{\rm Mat}_{m_u,m_v}(k)$, denote by $y_{u,v}$ its coordinate in ${\rm Mat}_{m_u,m_v}(k)$. Let $t_i$ be the image of $t$ in the residue field $k(\a_i)$. 

\begin{lemma} We have $X_{\bf A}^\calE\simeq \mathbb{A}_k^{\delta_\calE}\times Y_{\bf A}^\calE$ where

\begin{align*}
Y_{\bf A}^\calE=Y_{{\bf A},D}^\calE:&=\left\{(y^1,\dots,y^r)\in A_1\times\cdots\times A_r\,\left|\, \sum_{i=1}^r\sum_{s=0}^{d_i-1}F^s\left(y^i_{u,v}t_i^j\right)=0, 1\leq v\leq u\leq f, j=0,\dots,b_v-b_u\right\}\right.,\\
&=\left\{(y^1,\dots,y^r)\in A_1\times\cdots\times A_r\,\left|\forall h\in{\rm End}(\calE),\, \sum_{i=1}^r\sum_{s=0}^{d_i-1}{\Tr}\left(F^s(y^ih(t_i))\right)=0\right\}\right.
\end{align*}
where $\delta_\calE=\sum_{1\leq u< v\leq f}m_um_v(b_u-b_v-1)$ and $F$ denotes the Frobenius $t_i\mapsto t_i^q$ on $k(\a_i)/k$ if $k=\F_q$ (if $k=\overline{k}$ then $d_i=1$ for all $i$ and $F={\rm id}$).
\end{lemma}

\begin{proof}Assume that $k=\overline{k}$. To see the lemma we use that (locally) the coordinate $\varphi_{j,i}:\calO(b_i)\rightarrow\calO(b_j)\otimes\Omega^1(D)$ of $\varphi:\calE\rightarrow\calE\otimes\Omega^1(D)$ is of the form 

\begin{align*}
 &\sum_{i=1}^r\frac{\lambda_i}{t-\a_i}\,dt+T(t) dt,\hspace{.5cm} \text{ with } T(t)\in k[t]^{\leq b_j-b_i-2}, &\text{ if }b_j>b_i,\\
 &\sum_{i=1}^r\frac{\lambda_i}{t-\a_i}\,dt, \hspace{.5cm}\text{ with } \sum_{i=1}^r\lambda_i t_i^s=0\text{ for all }s=0,\dots,b_i-b_j, &\text{ if }b_i\geq b_j,
\end{align*}
with the convention that $k[t]^{\leq-1}=\{0\}$.
\end{proof}

\begin{remark} Note that the condition

$$
\sum_{i=1}^r\sum_{j=0}^{d_i-1}{\Tr}(F^j(A_i))=0,
$$
on ${\bf A}$ is a necessary condition for $X_{\bf A}^\calE$ to be non-empty.

\end{remark}

\begin{example} $\calE=\calO(a)^n$, $d_1=\dots=d_r=1$, then $X_{\bf A}^\calE\simeq Y_{\bf A}^\calE$ is identified with

$$
\left\{(X_1,\dots,X_r)\in A_1\times\dots\times A_r\,\left|\, \sum_{i=1}^rX_i=0\right\}\right..
$$
\label{exCB}\end{example}

\begin{example} If $r=1$, $k=\F_q$, $\calE=\calO(a)^n$, then $X_{\bf A}^\calE\simeq Y_{\bf A}^\calE$ is identified with

$$
\left\{X\in A_1\,\left|\, X+F(X)+\cdots+F^{d_1-1}(X)=0\right\}\right..
$$
\end{example}

\begin{example} If $n=2$, $d_1,\dots,d_r=1$ and $\calE=\calO(1)\oplus\calO(-1)$, then $X_{\bf A}^\calE\simeq \mathbb{A}^1\times Y_{\bf A}^\calE$ with $Y_{\bf A}^\calE=$

$$
\left\{\left(\left(\begin{array}{cc}\alpha_1&\beta_1\\\gamma_1&\delta_1\end{array}\right),\dots,\left(\begin{array}{cc}\alpha_r&\beta_r\\ \gamma_r&\delta_r\end{array}\right)\right)\in A_1\times\cdots\times A_r\,\left|\, \sum_{i=1}^r\alpha_i=\sum_{i=1}^r\delta_i=\sum_{i=1}^r\gamma_i=\sum_{i=1}^r\gamma_i t_i=\sum_{i=1}^r\gamma_it_i^2=0\right\}\right..
$$
\end{example}

\begin{remark} Assume that $k=\F_q$ and choose an algebraic closure $\overline{k}$ with Frobenius $F:\overline{k}\rightarrow\overline{k}$, $x\mapsto x^q$ and identify $k(\a_i)$ with $\F_{q^{d_i}}\subset\overline{k}$. 

Consider the divisor $\tilde{D}=D\times_k\overline{k}=t_1+t_1^q+\cdots+t_1^{q^{d_1-1}}+\cdots+t_r+t_r^q+\cdots+t_r^{q^{d_r-1}}$ on $\bbP^1_{\overline{k}}$.  Put $d=d_1+\dots+d_r$ and consider the $d$-tuple $\tilde{\bf A}=\left(\tilde{A}_1,F(\tilde{A}_1),\dots,F^{d_1-1}(\tilde{A}_1),\dots,\tilde{A}_r, F(\tilde{A}_r),\dots,F^{d_r-1}(\tilde{A}_r)\right)$ of adjoint orbits of $\gl_n(\overline{k})$ (where $\tilde{A}_i$ is the $\GL_n(\overline{k})$ saturated of $A_i$). Consider the permutation $w\in \frak{S}_r$ which acts on $\gl_n(\overline{k})^d\simeq \gl_n(\overline{k})^{d_1}\times\cdots\times\gl_n(\overline{k})^{d_r}$ by cyclic permutation on each $\gl_n(\overline{k})^{d_i}$, with $i=1,\dots,r$. The Frobenius $wF$ on $\gl_n(\overline{k})$ preserves the subvariety $Y_{\tilde{\bf A},\tilde{D}}^{\tilde{\calE}}$, where $\tilde{\calE}$ is the vector bundle on $\bbP^1_{\overline{k}}$ obtained from $\calE$ by scalars extension, and 

$$
\left(Y_{\tilde{\bf A},\tilde{D}}^{\tilde{\calE}}\right)^{wF}\simeq Y_{{\bf A},D}^\calE.
$$
\label{F_qbar}\end{remark}

Say that  an $r$-tuple $(O_1,\dots,O_r)$ of adjoint orbit of $\gl_n(\overline{k})$  is \emph{generic} if the following two conditions are satisfied :

(1) For any integer $1\leq m<n$, we choose $m$ eigenvalues of $O_i$ (counted with multiplicities) for each $i=1,\dots,r$, the sum of all selected eigenvalues ($mr$ in total) is never equal to zero.

(2) $\sum_{i=1}^r{\rm Tr}(O_i)=0$.
\bigskip

In \cite[\S 5.1]{ICQV} we discuss the existence of generic tuples with prescribed Jordan form.
\bigskip

Say that ${\bf A}=(A_1,\dots,A_r)$ is \emph{generic} if $\tilde{\bf A}$ is generic.

\subsection{Higgs bundles over finite fields and Fourier transforms}

Assume $k=\F_q$ and, as before, to alleviate the notation write $\gl_n(q^{d_i})$ instead of $\gl_n(k(\a_i))$.  Fix a rank $n$ vector bundle $\calE$ on $\bbP^1_k$. The aim of this section is to give a formula for $| X_{\bf A}^\calE|$. We will use this formula to relate $|X_{\bf A}^\calE|$  with the count of geometrically indecomposable parabolic bundles.

Choose a non-trivial linear additive character $\psi: k\rightarrow \C^\times$ and for $i=1,\dots,r$, put $\psi_{q^{d_i}/q}:\psi\circ {\rm Tr}_{k(\a_i)/k}:k(\a_i)\rightarrow \C^\times$. 
Denote by ${\rm Fun}(\gl_n(q^{d_i}))$ the $\C$-vector space of all functions $\gl_n(q^{d_i})\rightarrow\C$. Define the Fourier transform $\calF^{\gl_n(q^{d_i})}:{\rm Fun}(\gl_n(q^{d_i}))\rightarrow {\rm Fun}(\gl_n(q^{d_i}))$ by 

$$
\calF^{\gl_n(q^{d_i})}(f)(x)=\sum_{y\in\gl_n(q^{d_i})}\psi_{q^{d_i}/q}\left({\rm Tr}(xy)\right) f(y),
$$
for any $f\in {\rm Fun}(\gl_n(q^{d_i}))$ and 
$x\in\gl_n(q^{d_i})$.

\begin{proposition}We have

\begin{equation}
|X_{\bf A}^\calE|=q^{-n^2}\sum_{h\in{\rm End}(\calE)}\prod_{i=1}^r\calF^{\gl_n(q^{d_i})}(1_{A_i})(h(\a_i)),
\end{equation}
where $1_{A_i}:\gl_n(q^{d_i})\rightarrow\C$ is the characteristic function of $A_i$, i.e. it takes the value $1$ on $A_i$ and zero elsewhere.

\label{eqHiggs}\end{proposition}

\begin{proof} We have 
\bigskip

$
\sum_{h\in{\rm End}(\calE)}\prod_{i=1}^r\calF^{\gl_n(q^{d_i})}(1_{A_i})(f(\a_i))
$

\begin{align*}
&=\sum_{h\in{\rm End}(\calE)}\prod_{i=1}^r\sum_{y\in\gl_n(q^{d_i})}\psi_{q^{d_i}/q}\left({\rm Tr}(yh(\a_i))\right)1_{A_i}(y)\\
&=\sum_{(y^1,\dots,y^r)\in A_1\times\cdots\times A_r}\sum_{h\in{\rm End}(\calE)}\prod_{i=1}^r\psi\left(\sum_{s=0}^{d_i-1}{\rm Tr}\left(y^ih(\a_i)\right)^{q^s}\right)\\
&=\sum_{(y^1,\dots,y^r)\in A_1\times\cdots\times A_r}\sum_{h\in{\rm End}(\calE)}\psi\left(\sum_{i=1}^r\sum_{s=0}^{d_i-1}{\rm Tr}\left(y^ih(\a_i)\right)^{q^s}\right)\\
&=\sum_{(y^1,\dots,y^r)\in A_1\times\cdots\times A_r}\sum_{h\in{\rm End}(\calE)}\psi\left(\sum_{i=1}^r\sum_{s=0}^{d_i-1}{\rm Tr}\left(\sum_{u=1}^f\sum_{v=1}^fy^i_{u,v}h_{v,u}(\a_i)\right)^{q^s}\right)\\
&=\sum_{(y^1,\dots,y^r)\in A_1\times\cdots\times A_r}\sum_{h\in{\rm End}(\calE)}\psi\left(\sum_{u=1}^f\sum_{v=1}^f{\rm Tr}\left(\sum_{i=1}^r\sum_{s=0}^{d_i-1}F^s\left(y^i_{u,v}h_{v,u}(\a_i)\right)\right)\right)\\
&=\sum_{(y^1,\dots,y^r)\in A_1\times\cdots\times A_r}\prod_{1\leq v\leq u\leq f}\sum_{h_{v,u}(t)\in{\rm Mat}_{m_v,m_u}(k[t]^{\leq b_v-b_u})}\psi\left({\rm Tr}\left(\sum_{i=1}^r\sum_{s=0}^{d_i-1}F^s\left((y^i_{u,v}h_{v,u}(\a_i)\right)\right)\right)\\
\end{align*}
Recall that $t_i$ is the image of $t$ in $k(\a_i)$ and write $h_{v,u}(\a_i)=\sum_{j=0}^{b_v-b_u}X^j_{v,u}t_i^j$ with $X_{v,u}^j\in{\rm Mat}_{m_v,m_u}(k)$. Then

\begin{align*}
\sum_{h\in{\rm End}(\calE)}\prod_{i=1}^r\calF^{\gl_n(q^{d_i})}(1_{A_i})(h(\a_i))&=
\sum_{(y^1,\dots,y^r)\in A_1\times\cdots\times A_r}\prod_{1\leq v\leq u\leq f}\prod_{j=0}^{b_v-b_u}\sum_{X\in{\rm Mat}_{m_v,m_u}(k)}\psi\left({\rm Tr}\left(X\sum_{i=1}^r\sum_{s=0}^{d_i-1}F^s\left(y^i_{u,v}t_i^j\right)\right)\right)\\
&=q^{\sum_{1\leq u\leq v\leq f}m_um_v(b_u-b_v+1)}\,|Y_{{\bf A}}^\calE|=|{\rm End}(\calE)|\,|Y_{{\bf A}}^\calE|,
\end{align*}
that is

$$
|Y_{{\bf A}}^\calE|=\frac{1}{|{\rm End}(\calE)|}\sum_{h\in{\rm End}(\calE)}\prod_{i=1}^r\calF^{\gl_n(q^{d_i})}(1_{A_i})(h(\a_i)).
$$

The result follows from the relation $q^{\delta_\calE}q^{n^2}=|{\rm End}(\calE)|$.
\end{proof}

\subsection{Fourier transforms and characters of $\GL_n(q)$}

\subsubsection{Harish-Chandra induction : The Lie algebra case}\label{HCILie}

We use the same notation as in \S \ref{Harish-Chandra}. Let $P=L\rtimes U_P$ be an $F$-stable Levi decomposition of a parabolic subgroup $P$ of $\GL_n$. Denote by $\frakp=\frakl\oplus\fraku_P$ the corresponding Lie algebra decomposition in $\gl_n$ (i.e. $\frakp={\rm Lie}(P)$, $\frakl={\rm Lie}(L)$ and $\fraku_P={\rm Lie}(U_P)$). We denote by $C(\frakl(q))$ the $\C$-vector space of functions $\frakl(q)\rightarrow\C$ that are constant on $L(q)$-(adjoint) orbits of $\frakl(q)$. The Harish-Chandra induction $R_{\frak{l}(q)}^{\gl_n(q)}:C(\frakl(q))\rightarrow C(\gl_n(q))$ is the $\C$-linear operator defined by the formula

$$
R_{\frakl(q)}^{\gl_n(q)}(f)(x)=\frac{1}{|P(q)|}\sum_{g\in G(q), g^{-1}xg\in P}f(\pi_\frakp(g^{-1}xg)),
$$
where $\pi_\frakp$ is the projection $\frakp\rightarrow \frakl$.

Recall that since we working in $\GL_n$, centralizers of semisimple elements of $\gl_n$ are always Levi subgroups of $\GL_n$ and all Levi subgroups of $\GL_n$ can be realized as the centralizer of some semisimple element in $\gl_n$.

Let $A$ be an adjoint orbit of $\gl_n(q)$ with $y\in A$ such that $L=C_{\GL_n}(y_s)$. We  denote by $A^L$ the $L(q)$-orbit of $y$ in $\frakl(q)$. We have the following well-known result (it is a particular case of \cite[Proposition 7.1.8]{letellier}).

\begin{proposition} $$R_{\frakl(q)}^{\gl_n(q)}(1_{A^L})=1_A.$$

\end{proposition}
 
Denote by $\calF^{\frakl(q)}:{\rm Fun}(\frakl(q))\rightarrow{\rm Fun}(\frakl(q))$ the Fourier transform with respect to the additive character $\psi:k\rightarrow\C^\times$ and the trace map. We have the following result (see for instance Lehrer \cite{Lehrer}).
\bigskip

\begin{theorem}

$$
\calF^{\gl_n(q)}\circ R_{\frakl(q)}^{\gl_n(q)}=q^{{\rm dim} U_P}R_{\frakl(q)}^{\gl_n(q)}\circ\calF^{\frakl(q)}.
$$
\end{theorem}

Assume now that $A$ is semisimple orbit. Since $y\in A$ is central in $\frakl$, the function $\calF^{\frakl(q)}(1_{A^L})$ is the linear additive character $\eta_{A^L}:\frakl(q)\rightarrow\C^\times$, $x\mapsto \psi ({\rm Tr}(xy))$. From the above results we have

\begin{equation}
\calF^{\gl_n(q)}(1_A)=q^{\frac{1}{2}{\rm dim}\, A(\overline{k})}\,R_{\frakl(q)}^{\gl_n(q)}\left(\eta_{A^L}\right).
\label{Fouror}\end{equation}

\subsubsection{Sum of character values}

If $\frakm$ is a Lie subalgebra of $\gl_n$, we denote by $z_\frakm$ its center. 

Assume given standard Levi subgroups $L_1,\dots,L_r$ of $\GL_n$. Above Theorem \ref{theo} we defined the notion of generic $r$-tuples of linear characters of $L_1(q^{d_1}),\dots,L_r(q^{d_r})$. Let us now define the Lie algebra version of it.

 Say that an $r$-tuple $(\eta_1,\dots,\eta_r)$ of linear additive  characters of $\frakl_1(q^{d_1}),\dots,\frakl_r(q^{d_r})$ is \emph{generic} if for any Levi subgroup $M$ of $\GL_n$ defined over $k$ (with Lie algebra $\frakm$) such that $Z_M(q)\subset g_iL_i(q^{d_i})g_i^{-1}$ for some $g_i\in\GL_n(q^{d_i})$, the linear additive character $(^{g_1}\eta_1)|_{z_\frakm(q)}\cdots(^{g_r}\eta_r)|_{z_\frakm(q)}$ is trivial if and only if  $M=\GL_n$.
\bigskip

Assume that $(\alpha_1,\dots,\alpha_r)$ is a generic $r$-tuple of linear characters of $L_1(q^{d_1}),\dots,L_r(q^{d_r})$ and that $(\eta_1,\dots,\eta_r)$ is a generic $r$-tuple of linear additive characters of $\frakl_1(q^{d_1}),\dots,\frakl_r(q^{d_r})$. We have the following result.

\begin{theorem}For any rank $n$ vector bundle $\calE$ on $\bbP^1_k$ we have

$$
(q-1)\sum_{h\in{\rm End}(\calE)}\prod_{i=1}^rR_{\frakl_i(q^{d_i})}^{\gl_n(q^{d_i})}(\eta_i)(h(\a_i))=q\,\sum_{h\in{\rm Aut}(\calE)}\prod_{i=1}^rR_{L_i(q^{d_i})}^{\GL_n(q^{d_i})}(\alpha_i)(h(\a_i)).
$$
\label{charsum}\end{theorem}

\begin{proof}If $\calE$ is the rank $n$ trivial vector bundle, in which case ${\rm Aut}(\calE)=\GL_n$, the theorem is a particular case of \cite[Theorem 6.9.1, Formula (6.9.1)]{ICQV}. The proof for an arbitrary vector bundle $\calE$ is similar. More precisely, by Theorem \ref{intermediate}, we have 

$$
\sum_{h\in{\rm Aut}(\calE)}\prod_{i=1}^rR_{L_i(q^{d_i})}^{\GL_n(q^{d_i})}(\alpha_i)(h(\a_i))=\sum_{(M,\calO_1,\dots,\calO_r)}\frac{(q-1)K_M^o\hat{\calH}_{(M,\calO_1,\dots,\calO_r)}^\calE(q)}{|W_{\GL_n(q)}(M,\calO_1,\dots,\calO_r)|}\prod_{i=1}^rR_{L_i(q^{d_i})}^{\GL_n(q^{d_i})}(1)(l_M.\calO_i),
$$
where $\hat{\calH}_{(M,\calO_1,\dots,\calO_r)}^\calE(q)$ is the numerator of $\calH_{(M,\calO_1,\dots,\calO_r)}^\calE(q)$.
\bigskip

As we did for conjugacy classes we introduce the set $\mathcal{B}$ of $r$-tuples $(B_1,\dots,B_r)$ of adjoint orbits of $\gl_n(q^{d_1}),\dots,\gl_n(q^{d_r})$ which have a common semisimple part. Then as in the group case we have the notion of type for the elements of $\mathcal{B}$  which can be written in the form $(M,\calN_1,\dots,\calN_r)$ (up to $\GL_n(q)$-conjugation) where $M$ is the centralizer of a common semisimple part and $\calN_1,\dots,\calN_r$ are the nilpotent conjugacy classes in $\frakm$ defining  the orbits $B_1,\dots,B_r$.

Similarly to the group case we prove 

$$
\sum_{h\in{\rm End}(\calE)}\prod_{i=1}^rR_{\frakl_i(q^{d_i})}^{\gl_n(q^{d_i})}(\alpha_i)(h(\a_i))=\sum_{(M,\calN_1,\dots,\calN_r)}\frac{qK_M^o\hat{\calH}_{(M,\calN_1,\dots,\calN_r)}^\calE(q)}{|W_{\GL_n(q)}(M,\calN_1,\dots,\calN_r)|}\prod_{i=1}^rR_{\frakl_i(q^{d_i})}^{\gl_n(q^{d_i})}(1)(l'_M.\calN_i),
$$
where 

$$
\hat{\calH}_{(M,\calN_1,\dots,\calN_r)}^\calE(q):=\#\{f\in{\rm End}(\calE)\,|\, f(\a_i)\in B_i\}
$$
for some $r$-tuple $(B_1,\dots,B_r)$ of type $(M,\calN_1,\dots,\calN_r)$. Note that as in the group case (cf. Proposition \ref{indep2}), the right hand side depends only on the type of $(B_1,\dots,B_r)$. 
\bigskip

The natural bijection between unipotent conjugacy classes and nilpotent orbits induces a bijection between types of elements of $\mathcal{C}$ and types of elements of $\mathcal{B}$. If $(M,\calO_1,\dots,\calO_r)$ corresponds to $(M,\calN_1,\dots,\calN_r)$, then for all $i=1,\dots,r$

$$
R_{\frakl_i(q^{d_i})}^{\gl_n(q^{d_i})}(1)(l_M.\calN_i)=R_{L_i(q^{d_i})}^{\GL_n(q^{d_i})}(1)(l_M.\calO_i).
$$

\end{proof}

\subsection{Higgs bundles over finite fields and indecomposable bundles}

We assume that $k=\F_q$. 

Recall that ${\bf A}=(A_1,\dots,A_r)$ is a fixed $r$-tuple of adjoint orbits of $\gl_n(q^{d_1}),\dots,\gl_n(q^{d_r})$ and let $\tilde{\bf A}$ be defined from ${\bf A}$ as in  Remark \ref{F_qbar}. For each $i=1,\dots,r$, we choose an element $\sigma_i$ which is the semisimple part of an element of $A_i$. We put $L_i=C_{\GL_n}(\sigma_i)$ and $\frakl_i={\rm Lie}(L_i)$.  We consider the linear additive characters $\eta_i:=\calF^{\frakl_i(q^{d_i})}(1_{\sigma_i}):\frakl_i(q^{d_i})\rightarrow\C^\times$ for $i=1,\dots,r$. 

\begin{lemma}If ${\bf A}$ is generic then the $r$-tuple $(\eta_1,\dots,\eta_r)$ is generic.
\end{lemma}

\begin{proof}Assume that $\tilde{\bf A}$ is generic and let $M$ be a Levi subgroup of $\GL_n$ defined over $k$ such that for all $i=1,\dots,r$, we have $Z_M(q)\subset g_iL_i(q^{d_i})g_i^{-1}$ for some $g_i\in\GL_n(q^{d_i})$. Let $z\in z_\frakm(q)$. Then

\begin{align*}
\prod_{i=1}^r({^{g_i}}\eta_i)(z)&=\prod_{i=1}^r\psi_{q^{d_i}/q}\left({\rm Tr}(g_i\sigma_i g_i^{-1}z)\right)\\
&=\prod_{i=1}^r\psi\left({\rm Tr}\left(z\sum_{j=1}^{d_i-1}F^j(g_i\sigma_ig_i^{-1})\right)\right)\\
&=\psi\left({\rm Tr}\left(z\sum_{i=1}^r\sum_{j=1}^{d_i-1}F^j(g_i\sigma_ig_i^{-1})\right)\right)
\end{align*}

Since $\sigma_i$ is central in $\frakl_i$, the element $g_i\sigma_ig_i^{-1}$ is central in $g_i\frakl_ig_i^{-1}$ and so it commutes with all element of $z_\frakm$, i.e. $g_i\sigma_ig_i^{-1}\in C_{\gl_n}(z_\frakm)=\frakm$. The element $\sum_{i=1}^r\sum_{j=1}^{d_i-1}F^j(g_i\sigma_ig_i^{-1})$ is thus in $\frakm(q)$. Now we may write $\frakm\simeq\bigoplus_{i=1}^s\gl(V_i)$ over $k$. The result follows thus from the fact that for any $X=(X_i)_{i=1,\dots,s}\in\frakm(q)$, the character $\theta:z_\frakm(q)\rightarrow \C^\times$, $z\mapsto \psi({\rm Tr}(zX))$ is trivial is and only if ${\rm Tr}(X_i)=0$ for all $i=1,\dots, s$.
\end{proof}

From ${\bf A}=(A_1,\dots,A_r)$ we define a multi-partition $\muhat=(\mu^1,\dots,\mu^r)$ such that for each $i=1,\dots,r$, the parts of $\mu^i$ are the multiplicities of the eigenvalues of $A_i$.

We put 

\begin{equation}
d_{\bf A}:=\sum_{i=1}^rd_i\,{\rm dim}\, A_i(\overline{k})-2n^2+2,
\label{dA}\end{equation}
 and 
 
 $$
 {\rm PAut}(\calE):={\rm Aut}(\calE)/\mathbb{G}_m(\F_q),
 $$
 where $\mathbb{G}_m(\F_q)$ is identified with the homotheties in ${\rm Aut}(\calE)$.

\begin{theorem}Assume that ${\bf A}$ is generic, then 

$$
\frac{|X_{\bf A}^\calE|}{|{\rm PAut}(\calE)|}=q^{\frac{1}{2}d_{\bf A}}\calA_{\muhat}^\calE(q).
$$
\label{maintheoHiggs}\end{theorem}

\begin{proof}Applying Proposition \ref{eqHiggs} and Formula (\ref{Fouror}) we have

\begin{align*}
\frac{|X_{\bf A}^\calE|}{|{\rm PAut(\calE)}|}&=\frac{(q-1)q^{-n^2}}{|{\rm Aut}(\calE)|}\sum_{h\in{\rm End}(\calE)}\prod_{i=1}^r\calF^{\gl_n(q^{d_i})}(1_{A_i})(h(\a_i))\\
&=\frac{(q-1)q^{\frac{1}{2}\sum_{i=1}^rd_i\,{\rm dim}\,A_i(\overline{k})-n^2}}{|{\rm Aut}(\calE)|}\sum_{h\in{\rm End}(\calE)}\prod_{i=1}^r\calF^{\gl_n(q^{d_i})}(\eta_{A_i^{L_i}})(h(\a_i))
\end{align*}
Applying Theorem \ref{charsum} we deduce

\begin{align*}
\frac{|X_{\bf A}^\calE|}{|{\rm PAut(\calE)}|}&=\frac{q^{\frac{1}{2}\sum_{i=1}^rd_i\,{\rm dim}\,A_i(\overline{k})-n^2+1}}{|{\rm Aut}(\calE)|}\sum_{h\in{\rm Aut}(\calE)}\prod_{i=1}^rR_{L_i(q^{d_i})}^{\GL_n(q^{d_i})}(\alpha_i)(h(\a_i))
\end{align*}
for some generic tuple $(\alpha_1,\dots,\alpha_r)$. 

The theorem is thus a consequence of Theorem \ref{theo}.
\end{proof}

\section{Conjectures and geometric interpretations}

For ${\bf b}=(b_1,\dots,b_f)\in\Z^f$ with $b_1>\cdots>b_f$ and $\m=(m_1,\dots,m_f)\in\N^f$ we put

$$
\calE^{{\bf b};\m}:=\bigoplus_{i=1}^f\calO(b_i)^{m_i}
$$
on $\bbP^1_k$. Note that in \S \ref{counting}, the $f$-tuple ${\bf b}$ was fixed and we wrote $\calE^\m$ instead of $\calE^{{\bf b},\m}$. The reason is that below we will need to vary ${\bf b}$. Define the degree of $({\bf b};\m)$ to be the degree of the vector bundle $\calE^{{\bf b};\m}$.
\bigskip

Let  $\muhat=(\mu^1,\dots,\mu^l)$ be an $l$-tuple of partitions of  $n=\sum_{i=1}^fm_i$.  We denote by $S_\muhat$ the stabilizer of $\muhat$ in $\mathfrak{S}_l$. It is of the form $\prod_{i=1}^s\mathfrak{S}_{l_i}$ where $l_1,\dots,l_s$ are the multiplicities of the distinct partitions $\mu^{h_1},\dots,\mu^{h_s}$ in $\muhat$. We put $\calP_\muhat:=\calP_{l_1}\times\cdots\times\calP_{l_s}$. Say that a reduced divisor $D$ on $\bbP^1_k$ is compatible with $\lambdahat=(\lambda^1,\dots,\lambda^s)\in \calP_\muhat$ where $\lambda^i=(\lambda^i_1,\dots,\lambda^i_{p_i})$ if $D=\sum_{i=1}^s\sum_{j=1}^{p_i}\a_{i,j}$ and ${\rm deg}(\a_{i,j})=\lambda^i_j$. If $w\in S_\muhat$ is of cycle-type decomposition $\lambdahat\in \calP_\muhat$ we will also say that $D$ is compatible with $w$ if it is compatible with $\lambdahat$.

\subsection{Some conjectures}\label{mainconj}

For a divisor $D$ as above compatible with $\lambdahat\in\calP_\muhat$, we denote by $\calA^{\calE^{{\bf b};\m}}_{\muhat,\lambdahat, D}(q)$ the number of isomorphism classes of geometrically indecomposable parabolic structures (of type $\mu^{h_i}$ at $\a_{i,j}$) on the vector bundle  $\calE^{\bf b;\m}$  on $\bbP^1_k$. The appearance of both $\lambdahat$ and $D$ in the notation $\calA^{\calE^{{\bf b};\m}}_{\muhat,\lambdahat, D}(q)$ is redundant as $D$ determines $\lambdahat$ (which explains by the way why we omitted both $D$ and $\lambdahat$ in the previous sections as $D$ was fixed). However this notation should clarify slightly  the statement of the conjectures below.
\bigskip

We denote by $\calC h(S_\muhat)$ the character ring of $S_\muhat$ over $\Z$, and for a polynomial $P(t)\in\calC h(S_\muhat)[t]$ we denote by $P_w(t)$ the polynomial obtained by evaluating the coefficients of $P(t)$ at $w\in S_\muhat$. We denote by $\calC h(S_\muhat)^+$ the subset of $\calC h(S_\muhat)$ of positive linear combinations of irreducible characters. The following conjecture was suggested to me by Deligne.

\begin{conjecture}There exists a polynomial $A^{{\bf b};\m}_\muhat(t)\in\calC h(S_\muhat)[t]$  with coefficients in $\calC h(S_\muhat)^+$ such that for any finite field $k=\F_q$, for any $\lambdahat\in\calP_\muhat$ and any reduced divisor $D$ on $\bbP^1_k$ compatible with $\lambdahat$,  we have 

$$
\calA^{\calE^{{\bf b};\m}}_{\muhat,\lambdahat,D}(q)=A^{{\bf b};\m}_{\muhat,w}(q).
$$
for any element $w\in S_\muhat$ of cycle-type $\lambdahat$. 
\label{mainconj1}\end{conjecture}

We proved the conjecture when $n=2$ in \S \ref{rank2}.

\begin{remark} This conjecture is a stronger version of Conjecture \ref{polynomiality} which is proved for vector bundles $\calE$ of the form $\calO(b_1)^{m_1}\oplus\calO(b_2)^{m_2}$ (Corollary \ref{polynomialitymm}).
\end{remark}

\begin{conjecture} Let $d$ be an integer. Then for  any finite field $k=\F_q$, any positive integer $n$,  any $\muhat\in(\calP_n)^l$, any $\lambdahat \in\calP_\muhat$ and any reduced divisor $D$ on $\bbP^1_k$ compatible with $\lambdahat$, the sum

$$
\sum_{{\rm deg}({\bf b};\m)=d,\, \sum_{i=1}^fm_i=n}\calA^{\calE^{{\bf b};\m}}_{\muhat,\lambdahat,D}(q)
$$
is independent of $d$.
\label{independent}\end{conjecture}

In the rank $n=2$ case we already proved this conjecture (see Theorem \ref{sumind}).
\bigskip

\begin{remark}(Generalisation of Kac's conjecture to non-trivial bundles) In \S \ref{Kacpoly} we noticed that when $\calE$ is a trivial rank $n$ bundle and the coordinates of $\lambdahat$ are the trivial partitions, the quantity $\calA^\calE_{\muhat,\lambdahat,D}(q)$ coincides with the number of isomorphism classes of geometrically indecomposable representations of the quiver in \S \ref{Kacpoly} equiped with a certain dimension vector $\v_\muhat$ (the relation between $\s$ in \S \ref{Kacpoly} and $\muhat$ is explained above Corollary \ref{s'}). On the other hand Kac \cite{kac2} proved that for any finite quiver $\Gamma$ equiped with any dimension vector $\v$, there exists a polynomial $A_{\Gamma,\v}(t)\in\Z[t]$, so-called Kac polynomial,  such that for any finite field $\F_q$, the evaluation $A_{\Gamma,\v}(q)$ is the number of isomorphism classes of geometrically indecomposable representations of $(\Gamma,\v)$ over $\F_q$. He conjectured that this polynomial has non-negative coefficients. Kac's conjecture was proved in \cite{HLRV} in full generality. Notice that Conjecture \ref{mainconj1} implies a generalisation of  Kac's results to non-trivial bundles.
\label{kacpolynomial}\end{remark}

\subsection{The geometric interpretation}\label{geominter}

In the following we put  $\calE:=\calE^{{\bf b};\m}$ to alleviate the notation. We say that $\muhat$ is \emph{indivisible} if the gcd of all parts of the partitions $\mu^1,\dots,\mu^l$ equals $1$.

We assume that  

$$
{\rm char}(k) \nmid  d !,
$$
where $d={\rm min}_i\,{\rm max}_j \,\mu^i_j$. Recall \cite[\S 2.2]{aha} that if $\muhat$ is indivisible then there always exists a generic $l$-tuple $\A=(A_1,\dots,A_l)$ of adjoint orbits of $\gl_n(\overline{k})$ of type $\muhat$, i.e. such that the multiplicities of the eigenvalues of $A_i$ are given by the partition $\mu^i$. If $\muhat$ is divisible (i.e. not indivisible), generic tuples of adjoint orbits of type $\muhat$ do not exist.

\subsubsection{Polynomial count varieties}

Let $R$ be a subring of $\C$ which is finitely generated as a $\Z$-algebra and let $\calX$ be a separated $R$-scheme of finite type. According to Katz \cite[Appendix]{HRV}, we say that $\calX$ is \emph{strongly polynomial count} if there exists a polynomial $P(T)\in\C[T]$ such that for any finite field $\F_q$ and any ring homomorphism $\varphi: R\rightarrow\F_q$, the $\F_q$-scheme $\calX^{\varphi}$ obtained from $\calX$ by base change is polynomial count with counting polynomial $P_X$, i.e., for every finite extension $\F_{q^n}/\F_q$, we have $$\#\,\calX^{\varphi}(\F_{q^n})=P_X(q^n).$$

We call a separated $R$-scheme $\calX$ which gives back $X/_\C$ after extension of scalars from $R$ to $\C$ a \emph{spreading out} of $X/_\C$. 

Say that $X/_\C$ has \emph{polynomial count} if it admits a spreading out $\calX$  which is strongly polynomial count.

As an easy consequence of \cite[Appendix]{HRV} we have.

\begin{proposition} Assume that $X/_\C$ has polynomial count and that its compactly supported cohomology $H_c^i(X/_\C,\C)$ has pure mixed Hodge structure. Then its odd cohomology vanishes and 

$$
P_X(t)=\sum_i{\rm dim}\,H_c^{2i}(X/_\C,\C)\, t^i.
$$
\label{Katz}\end{proposition}

\subsubsection{The indivisible case}\label{theindivcase}

Assume that $\muhat$ is indivisible and let $\A=(A_1,\dots,A_l)$ be a generic $l$-tuple of adjoint orbits of $\gl_n(\C)$ of type $\muhat$  and consider the complex affine algebraic variety $X_\A^\calE$ defined in \S \ref{Higgs} with respect to some reduced divisor $\mathcal{D}=\a_1+\cdots+\a_l$ on $\bbP^1_\C$. We have the following conjecture.

\begin{conjecture} (i) The categorical quotient $\calX_\A^\calE$ of $X_\A^\calE=X_{\A,\mathcal{D}}^\calE$ by ${\rm PAut}(\calE)$ exists (in the category of algebraic varieties) and the quotient map $X_\A^\calE\rightarrow\calX_\A^\calE$ is a principal ${\rm PAut}(\calE)$-bundle in the \'etale topology.

\noindent (ii) The variety $\calX_\A^\calE$ is non-singular, has polynomial count and the compactly supported cohomology $H_c^i(\calX_\A^\calE,\C)$ has pure mixed Hodge structure.

\noindent (iii) There is a natural structure of $\C[S_\muhat]$-module on $H_c^i(\calX_\A^\calE,\C)$ such that $A_\muhat^{{\bf b},\m}(t)$ in Conjecture \ref{mainconj1} is given by
\begin{equation}\label{iden}
A_\muhat^{{\bf b};\m}(t)=t^{-\frac{1}{2}d_\A}\sum_i[H_c^{2i}(\calX_\A^\calE,\C)] t^i,
\end{equation}
where $d_\A$ is defined by (\ref{dA}), and $[H_c^{2i}(\calX_\A^\calE,\C)]$ is the corresponding element in $\calC h(S_\muhat)$.
\label{congind}\end{conjecture}

The integer $d_\A$  is  the dimension of the moduli space of Higgs bundles 

$$
\calX_\A^{n,d}=\left\{(\calE,\varphi)\,|\, \calE\in{\rm Bun}^{n,d}(\bbP^1_\C), \varphi\in X_\A^\calE\right\}/\sim,
$$
which, unlike its strata $\calX_\A^\calE$, is known to be a complex algebraic variety.
\bigskip

\begin{remark}Theorem \ref{maintheoHiggs} is an evidence for the assertion (iii) as detailed below in the special case where $\calE=\calO^n$.

\label{maintheoappli}\end{remark}

Assume that $\calE=\calO(a)^n$ (i.e. ${\bf b}=(a)$ and $\m=(n)$), and write $\calX_\A$ instead of $\calX_\A^\calE$. The assertions (i) and (ii) are  \cite[Theorem 2.2.4, Proposition 2.2.6]{aha}. The identity (\ref{iden}) evaluated at $w=1\in S_\muhat$ is true by Crawley-Boevey \cite{crawley-mat} together with Crawley-Boevey and van den Bergh \cite{crawley-etal}. Roughly speaking Crawley-Boevey identified in \cite{crawley-mat} the affine GIT quotient

$$
\calX_\A=\left.\left.\left\{(X_1,\dots,X_l)\in A_1\times\cdots\times A_l\,\left|\, \sum_iX_i=0\right\}\right.\right/\!\right/\GL_n
$$
with a certain generic quiver variety attached to the star shaped quiver in \S \ref{Kacpoly}. By the results of \cite{crawley-etal} and \cite{aha} we can choose an appropriate subalgebra $R$ of $\C$ which is a finitely generated $\Z$-algebra, a separated $R$-scheme $\bbX_\A$ which gives back $\calX_\A$ after scalar extension from $R$ to $\C$ and such that for any ring homomorphism $\varphi:R\rightarrow\F_q$

$$
A_{\muhat,1}^{(a);(n)}(q)=q^{-\frac{1}{2}d_\A}\,|\bbX_\A^\varphi(\F_q)|,
$$
where $A_{\muhat,1}^{(a),(n)}(t)$ is the Kac polynomial mentioned in \S \ref{kacpolynomial}. In particular we get that the variety $\calX_\A$ has polynomial count. The variety $\calX_\A$ has pure mixed Hodge structure \cite[Proposition 2.2.6]{aha} and so the formula (\ref{iden}) evaluated at $w=1$ is true by Proposition \ref{Katz}. We can also argue as in \cite{crawley-etal} where it is proved that  the eigenvalues of Frobenius $F^*$ on $H_c^i(\bbX_\A^\varphi(\overline{\F}_q),\overline{\Q}_\ell)$ are exactly $q^{i/2}$.\bigskip

We now prove the assertion (iii) of  Conjecture \ref{congind} in the case $\calE=\calO^n$. 

To a partition $\lambda=(\lambda_1,\lambda_2,\dots,\lambda_r)$, we denote by $H_\lambda$ the stabilizer of $\lambda$ in $\frak{S}_r$ (for the permutation action of $\frak{S}_r$ on $\N^r$). The group $S_\muhat$ acts on $\prod_{i=1}^lH_{\mu^i}$ by permutation of the coordinates and we consider the semi-direct product

$$
\mathfrak{H}_\muhat:=\left(\prod_{i=1}^lH_{\mu^i}\right)\rtimes S_\muhat.
$$
We now define an action of $\frak{H}_\muhat$ on $H_c^i(\calX_\A,\C)$ as follows (for this section only the action of $S_\muhat$ matters).

We denote by $\t_\muhat^{\rm gen}$ the set of generic $l$-tuples $(\sigma_1,\dots,\sigma_l)$ of diagonal matrices of $\GL_n$ such that for each $i$, the diagonal matrix $\sigma_i$ is of type $\mu^i=(\mu^i_1,\mu^i_2,\dots)$, i.e. 

$$(\overbrace{\alpha_1,\dots,\alpha_1}^{\mu^i_1},\overbrace{\alpha_2,\dots,\alpha_2}^{\mu^i_2},\dots).$$

The space $\t_\muhat^{\rm gen}$ is endowed with a natural action of $\frak{H}_\muhat$ where $S_\muhat$ acts  by permutation of the coordinates and $H_{\mu^i}$ acts on the set of diagonal matrices of type $\mu^i$ par permutation of the blocks (of size given by the parts of the partitions).

Denote by $L_i$ the Levi subgroup $C_{\GL_n}(\sigma_i)$. Consider the space

$$
\calX_\muhat=\left.\left.\left\{\big((X_i)_{i=1,\dots,l},( \sigma_i)_{i=1,\dots,l},(g_iL_i)_{i=1,\dots,l}\big)\in \gl_n^l\times \t_\muhat^{\rm gen}\times \prod_{i=1}^l\GL_n/L_i\,\left|\, \sum_{i=1}^lX_i=0, g_i^{-1}X_ig_i=\sigma_i\right\}\right.\right/\!\right/\GL_n.
$$
where $\GL_n$ acts as 

$$
g\cdot\big((X_i)_i,( \sigma_i)_i,(g_iL_i)_i\big)=\big((gX_ig^{-1})_i,(\sigma_i)_i, (gg_iL_i)_i\big).
$$
Consider the projection $\pi:\calX_\muhat\rightarrow \t_\muhat^{\rm gen}$. Then the projection $\pi^{-1}(\sigma)\rightarrow \gl_n^l$ defines an isomorphism $\pi^{-1}(\sigma)\simeq\calX_\A$ if for each $i$ the adjoint orbit of the $i$-th coordinate of $\sigma\in\t_\muhat^{\rm gen}$ is $A_i$.

For each $i$, the group $H_{\mu^i}$ is naturally isomorphic to $W_{\GL_n}(L_i):=N_{\GL_n}(L_i)/L_i$ where $N_{\GL_n}(L_i)$ is the normalizer of $L_i$ in $\GL_n$ (see \cite[\S 6.2, \S 6.3]{ICQV} for details). For $w\in H_{\mu^i}=W_{\GL_n}(L_i)$ we denote by $\dot{w}\in N_{\GL_n}(L_i)$ a representative of $w$. The group $\prod_{i=1}^lH_{\mu^i}$ acts on $\calX_\muhat$ as 
\bigskip

$
(w_1,\dots,w_l)\cdot\big((X_1,\dots,X_l),( \sigma_1,\dots,\sigma_l),g_1L_1,\dots,g_lL_l\big)=$

$$\big((X_1,\cdots,X_l),(\dot{w}^{-1}_1\sigma_1\dot{w}_1,\dots,\dot{w}_l^{-1}\sigma_l\dot{w}_l), g_1\dot{w}_1L_1,\dots,g_l\dot{w}_lL_l\big),$$
and the group $\frak{S}_\muhat$ by permutation of the coordinates.

The map $\pi:\calX_\muhat\rightarrow\t_\muhat^{\rm gen}$ commutes with the action of $\frak{H}_\muhat$ and so for $w\in\frak{H}_\muhat$ and $\sigma\in\t_\muhat^{\rm gen}$, the corresponding isomorphism $\pi^{-1}(\sigma)\rightarrow \pi^{-1}(\dot{w}\sigma\dot{w}^{-1})$ induces an isomorphism $w^*$ on cohomology.

Following the arguments in \cite[\S 2.2]{HLRV} which appeared first in the work of Nakajima \cite{nakajima2}, Maffei \cite{maffei} and that of Crawley-Boevey and van den Bergh \cite{crawley-etal}, we prove that the sheaf $R^i\pi_!\C$ is constant. This implies the following result.

\begin{theorem} For each $\sigma,\tau\in\t_\muhat^{\rm gen}$ there exists a canonical  isomorphism $i_{\sigma,\tau}:H_c^i(\pi^{-1}(\sigma),\C)\rightarrow H_c^i(\pi^{-1}(\tau),\C)$ such that for all $\sigma,\tau,\zeta\in\t_\muhat^{\rm gen}$ we have $i_{\sigma,\tau}\circ i_{\zeta,\sigma}=i_{\zeta,\tau}$.
\label{CBidea}\end{theorem}

\begin{theorem} The map $\rho^i:\frak{H}_\muhat\rightarrow \GL\big(H_c^i(\pi^{-1}(\sigma),\C)\big)$, $w\mapsto i_{\dot{w}\sigma\dot{w}^{-1},\sigma}\circ (w^*)^{-1}$ is a representation of $\frak{H}_\muhat$ that does not depend on the choice of $\sigma\in\t_\muhat^{\rm gen}$.
\end{theorem}

We now prove the assertion (iii) of Conjecture \ref{congind} for $\calE=\calO(a)^n$, namely.

\begin{theorem}For any $w\in S_\muhat$ and any reduced divisor $D_w=\sum_{i=1}^s\sum_{j=1}^{p_i}\a_{i,j}$ of $\bbP^1_{\F_q}$ compatible with $w$, the polynomial

$$
q^{-\frac{1}{2}d_\A}\sum_i{\rm Tr}\left(w\,|\, H_c^{2i}(\calX_\A^{\calO(a)^n},\C)\right)\, q^i
$$
counts the number of isomorphism classes of geometrically indecomposable parabolic structures  of type  $\mu^{h_i}$ at $\a_{i,j}$  on the vector bundle $\calO(a)^n$ of $\bbP^1_{\F_q}$ (where we recall that $\mu^{h_1},\dots,\mu^{h_s}$ are the distinct coordinates of $\muhat$).
\label{theo2.5}\end{theorem}

\begin{proof} The strategy is similar to that in \cite[\S 2.2.2]{HLRV}. By Remark \ref{F_qbar} together with Theorem \ref{maintheoHiggs} the counting polynomial $A_{\muhat,w}^{(a);(n)}(t)$ of the isomorphism classes of geometrically indecomposable parabolic structures on $\calO(a)^n$ of type  $\mu^{h_i}$ at $\a_{i,j}$ verify an identity of the form

$$
A_{\muhat,w}^{(a);(n)}(q)=q^{-\frac{1}{2}d_\A}\left|\left.\left(\calX_{\A/_{\overline{\F}_q}}\right)^{wF}\right.\right|.
$$
for some generic $l$-tuple $\A/_{\overline{\F}_q}=(A_1/_{\overline{\F}_q},\dots,A_l/_{\overline{\F}_q})$ of semisimple adjoint orbits of $\gl_n(\overline{\F}_q)$ of same type as $\A$ (namely the multiplicities of the eigenvalues are given by $\muhat$). Then we apply Grothendieck trace formula to get

$$
A_{\muhat,w}^{(a);(n)}(q)=q^{-\frac{1}{2}d_\A}\sum_i{\rm Tr}\left((wF)^*\,|\, H_c^{2i}(\calX_{\A/_{\overline{\F}_q}},\overline{\Q}_\ell)\right).
$$
We can proceed as in \cite{HLRV} to see that the above construction of the action $\rho^i$ of $\frak{H}_\muhat$  can be done over the $\ell$-adic cohomology  $H_c^i(\calX_{\A/_{\overline{\F}_q}},\overline{\Q}_\ell)$ and that the trace of $w\in\frak{H}_\muhat$ is the same on both $H_c^i(\calX_{\A/_{\overline{\F}_q}},\overline{\Q}_\ell)$ and $H_c^i(\calX_\A,\C)$. 

The Frobenius map $wF$ acts on  $\calX_{\A/_{\overline{\F}_q}}$ but $w$ and $F$ don't. Therefore instead we choose a generic  $l$-tuple $\bfB/_{\overline{\F}}$ of $F$-stable semisimple adjoint orbits of $\gl_n(\overline{\F}_q)$ of same type as $\A$. Now the Frobenius $F$ acts on $\calX_{\bfB/_{\overline{\F}_q}}$ and we can prove that the two maps $F^*\circ \rho^i(w)$ on $H_c^i(\calX_{\bfB/_{\overline{\F}_q}},\overline{\Q}_\ell)$ and $(wF)^*$ on $H_c^i(\calX_{\A/_{\overline{\F}_q}},\overline{\Q}_\ell)$ are conjugate by $\iota_{\sigma,\tau}$ where $\sigma\in(\t_\muhat^{\rm gen})^{wF}$ is such that $\pi^{-1}(\sigma)=\calX_{\A/_{\overline{\F}_q}}$ and $\tau\in(\t_\muhat^{\rm gen})^F$ is such that $\pi^{-1}(\tau)=\calX_{\bfB/_{\overline{\F}_q}}$.

The theorem follows from the fact that $\rho^i(w)$ commutes with $F^*$  and that the unique eigenvalue of $F^*$ on $H_c^{2i}(\calX_{\bfB/\overline{\F}_q},\overline{\Q}_\ell)$ is $q^i$ (see below Remark \ref{maintheoappli}). \end{proof}

\subsubsection{The geometric conjecture in the general case}

 The idea is to consider a larger generic variety $\calX_{\bf S}^\calE$.

We choose a generic $l$-tuple $\bfS=(S_1,\dots,S_l)$ of regular semisimple adjoint orbits of $\gl_n(\C)$. It corresponds to the multi-partition $\alphahat=((1^n),\dots,(1^n))$ and $\frak{H}=\frak{H}_\alphahat=\left(\prod_{i=1}^l\frak{S}_n\right)\rtimes\frak{S}_l$.

Starting from our $\muhat$ we consider the $(\frak{S}_n)^l$-module 

$$
R_\muhat={\rm Ind}_{\frak{S}_{\mu^1}}^{\frak{S}_n}(1)\boxtimes\cdots\boxtimes {\rm Ind}_{\frak{S}_{\mu^l}}^{\frak{S}_n}(1),
$$
where for a partition $\lambda=(\lambda_1,\dots,\lambda_s)$, $\frak{S}_\lambda$ is the subgroup $\prod_{i=1}^s\frak{S}_{\lambda_i}$ of $\frak{S}_n$, and ${\rm Ind}_{\frak{S}_{\mu^i}}^{\frak{S}_n}(1)$ is the $\frak{S}_n$-module induced from the trivial $\frak{S}_{\mu^i}$-module.  

Note that the group $(\frak{S}_n)^l\rtimes S_\muhat$ acts naturally on $R_\muhat$.

If $V_1$ and $V_2$ are two $\C[(\frak{S}_n)^l\rtimes S_\muhat]$-modules, the group $S_\muhat$ acts on 

$$
W={\rm Hom}_{(\frak{S}_n)^l}(V_1,V_2)
$$ as $(r\cdot f)(v)=r\cdot f(r^{-1}\cdot v)$ for $f\in W$, $r\in S_\muhat$ and $v\in V_1$, and we denote by $[W]$ its class in $\calC h(S_\muhat)$.

\begin{conjecture}There exists an action of the group $\frak{H}$ on $H_c^i(\calX_\bfS^\calE,\C)$ such that the polynomial $A_\muhat^{{\bf b},\m}(t)$ in Conjecture \ref{mainconj1} is given by
\begin{equation}\label{iden2}
A_\muhat^{{\bf b};\m}(t)=t^{-\frac{1}{2}d_\bfS}\sum_i[W_\muhat^{2i}] t^i,
\end{equation}
where $W_\muhat^i$ is the $\C[S_\muhat]$-module ${\rm Hom}_{(\frak{S}_n)^l}\left(R_\muhat,\varepsilon^l\otimes H_c^{2i}(\calX_\bfS^\calE,\C)\right)$, $\varepsilon^l:=\overbrace{\varepsilon\boxtimes\cdots\boxtimes\varepsilon}^{l}$ with $\varepsilon$ the sign character of $\frak{S}_n$.
\label{conjgeneral}\end{conjecture}

\begin{theorem} Conjecture \ref{conjgeneral} is true if $\calE=\calO(a)^n$.
\label{theotriv}\end{theorem}

The proof of this theorem will be given after Theorem \ref{lasttheo} below.
\bigskip

\subsubsection{Strategy/evidences for Conjecture \ref{conjgeneral}} \label{strategy}
\bigskip

Let $w\in S_\muhat\subset\frak{S}_l$, $D_w=\sum_{i=1}^s\sum_{j=1}^{p_i}\a_{i,j}$ be a reduced divisor on $\bbP^1_{\F_q}$ compatible with $w$. Let $r=\sum_{i=1}^s p_i$ be the total number of points of $D_w$. 

If Conjecture \ref{mainconj1} is true, by Theorem \ref{theo2} we must have

$$
A_{\muhat,w}^{{\bf b};\m}(t)=\left\langle {\rm Coeff}_{Y^\m}\left[(t-1)\Log\,\Omega_{D_w}(t)\right],h_{\muhat_w}\right\rangle,
$$
where  $\muhat_w$ is the multi-partition in $(\calP_n)^r$ with coordinate $\mu^{h_i}$ at the point $\a_{i,j}$ of $D_w$. 

For a partition $\mu$ of $n$, the complete symmetric function $h_\mu$ decomposes into the basis of power symmetric function $\{p_\lambda\}_{\lambda \in\calP}$ as follows

$$
h_\mu=\sum_\nu z_\nu^{-1} R_{\mu,\nu}\, p_\nu,
$$
where $z_\nu$ is the cardinality of the centralizer in $\frak{S}_n$ of an element $u$ of cycle-type $\nu$, and $R_{\mu,\nu}$ is the trace of $u$ on the $\frak{S}_n$-module $R_\mu={\rm Ind}_{\frak{S}_\mu}^{\frak{S}_n}(1)$. Therefore
\begin{equation}
A_{\muhat,w}^{{\bf b};\m}(t)=\sum_{\lambdahat\in (\calP_n)^r} z_\lambdahat^{-1} R_{\muhat_w,\lambdahat}\, \left\langle {\rm Coeff}_{Y^\m}\left[(t-1)\Log\,\Omega_{D_w}(t)\right],p_\lambdahat\right\rangle.
\label{equaA}\end{equation}

Write $(\frak{S}_n)^l=\prod_{i=1}^s\prod_{j=1}^{p_i}(\frak{S}_n)^{d_{i,j}}$. For an element $\v_{i,j}=(v_1,\dots,v_{d_{i,j}})\in(\frak{S}_n)^{d_{i,j}}$ we denote by $\overline{\v}_{i,j}$ the product $v_1v_2\cdots v_{d_{i,j}}\in\frak{S}_n$ of its coordinates.

We have the following conjecture.
 
 \begin{conjecture}If $\lambdahat\in(\calP_n)^r=\prod_{i=1}^s\prod_{j=1}^{p_i}\calP_n$ we have
 
 \begin{equation}
 \left\langle {\rm Coeff}_{Y^\m}\left[(t-1)\Log\,\Omega_{D_w}(t)\right],p_\lambdahat\right\rangle=t^{-\frac{1}{2}d_\bfS}\sum_i {\Tr}\left((\v_\lambdahat,w)\,|\, \varepsilon^l\otimes H_c^{2i}(\calX^\calE_\bfS,\C)\right)\, t^i,
 \label{equaB}\end{equation}
 for any element $\v_\lambdahat=(\v_{i,j})\in \prod_{i=1}^s\prod_{j=1}^{p_i}(\frak{S}_n)^{d_{i,j}}$ such that the element $(\overline{\v}_{i,j})\in\prod_{i=1}^s\prod_{j=1}^{p_i}\frak{S}_n$ is of cycle-type $\lambdahat$.
 \label{conjB}\end{conjecture}

\begin{remark} If $\lambdahat=((1^n),\dots,(1^n))$  then $p_\lambdahat=h_\lambdahat$ and so, if $\v_\lambdahat=1$,  the conjecture is a particular case of Conjecture \ref{congind}(iii) by Theorem \ref{theo2}.
\end{remark}

We will prove that this conjecture is true for $\calE=\calO(a)^n$, but before let us prove that Equation (\ref{equaA}) together with (\ref{equaB}) implies Conjecture \ref{conjgeneral}.

Combining (\ref{equaA}) and (\ref{equaB}) we have

$$
A_{\muhat,w}^{{\bf b};\m}(t)=t^{-\frac{1}{2}d_\bfS}\sum_i\sum_{\lambdahat\in (\calP_n)^r} z_\lambdahat^{-1} R_{\muhat_w,\lambdahat}\, {\Tr}\left((\v_\lambdahat,w)\,|\, \varepsilon^l\otimes H_c^{2i}(\calX^\calE_\bfS,\C)\right)\, t^i.
$$

Now we have (see for instance \cite[Lemma 6.2.3]{ICQV}) 

$$
R_{\muhat_w,\lambdahat}={\Tr}\left((\v_\lambdahat,w)\,|\, R_\muhat\right).
$$
Hence 

\begin{align*}
A_{\muhat,w}^{{\bf b};\m}(t)&=t^{-\frac{1}{2}d_\bfS}\sum_i\frac{1}{|(\frak{S}_n)^l|}\sum_{v\in(\frak{S}_n)^r}  \sum_{\v\in(\frak{S}_n)^l, \v\mapsto v}{\Tr}\left((\v,w)\,|\, R_\muhat\right)\, {\Tr}\left((\v,w)\,|\, \varepsilon^l\otimes H_c^{2i}(\calX^\calE_\bfS,\C)\right)\, t^i\\
&=t^{-\frac{1}{2}d_\bfS}\sum_i\frac{1}{|(\frak{S}_n)^l|}\sum_{\v\in(\frak{S}_n)^l}{\Tr}\left((\v,w)\,|\, R_\muhat\right)\, {\Tr}\left((\v,w)\,|\, \varepsilon^l\otimes H_c^{2i}(\calX^\calE_\bfS,\C)\right)\, t^i,
\end{align*}
where $\v\mapsto v$ means that if $\v=(v_{i,j})\in\prod_{i=1}^s\prod_{j=1}^{p_i}\frak{S}_n^{d_{i,j}}$, then $v=(\overline{v}_{i,j})\in\prod_{i=1}^s\prod_{j=1}^{p_i}\frak{S}_n$. 

By \cite[Proposition 6.1.1]{ICQV} we conclude that

$$
A_{\muhat,w}^{{\bf b};\m}(t)=t^{-\frac{1}{2}d_\bfS}\sum_i{\Tr}\left(w\,|\, {\rm Hom}_{(\frak{S}_n)^l}(R_\muhat,\varepsilon^l\otimes H_c^{2i}(\calX_\bfS^\calE,\C))\right)\, t^i=t^{-\frac{1}{2}d_\bfS}\sum_i{\Tr}\left(w\,|\, W_\muhat^{2i}\right)\, t^i.
$$

\subsubsection{Evidences for Conjecture \ref{conjB}}
\bigskip

We first give a representation theoritical meaning of  $\left\langle {\rm Coeff}_{Y^\m}\left[(t-1)\Log\,\Omega_{D_w}(t)\right],p_\lambdahat\right\rangle$.

Recall that the $\GL_n(\overline{\F}_q)$-conjugacy classes of the maximal tori of $\GL_n(\overline{\F}_q)$ defined over $\F_{q^d}$  (i.e. $F^{q^d}$-stable) are naturally parametrized by the conjugacy classes of $\mathfrak{S}_n$. If $v\in\mathfrak{S}_n$ we choose an $F^{q^d}$-stable maximal torus $T_v$ in the corresponding conjugacy class and we denote by $T_v(\F_{q^d})$ the group of its rational points. Let $\ell$ be a prime invertible in $\F_q$. Then Deligne and Lusztig \cite{DLu} defined for any $\overline{\Q}_\ell$-character $\alpha$ of $T_v(\F_{q^d})$ a virtual  $\overline{\Q}_\ell$-character $R_{T_v(\F_{q^d})}^{\GL_n(\F_{q^d})}(\alpha)$ of $\GL_n(\F_{q^d})$. If $v=1$, then the Deligne-Lusztig character $R_{T_v(\F_{q^d})}^{\GL_n(\F_{q^d})}(\alpha)$ is the Harish-Chandra induced of  $\alpha$ as $T_1$ is contained in a rational Borel subgroup.

\begin{theorem}If $(\alpha_{i,j})_{i,j}$ is a generic $r$-tuple of linear characters of the finite tori $T_{v_{i,j}}(q^{d_{i,j}})$, $i=1,\dots,s, j=1,\dots,p_i$, then
$$
\left\langle {\rm Coeff}_{Y^\m}\left[(q-1)\Log\,\Omega_{D_w}(q)\right],p_\lambdahat\right\rangle=\frac{1}{|{\rm Aut}(\calE)|}\sum_{h\in{\rm Aut}(\calE)}\prod_{i=1}^s\prod_{j=1}^{p_i}R_{T_{v_{i,j}}(q^{d_{i,j}})}^{\GL_n(q^{d_{i,j}})}(\alpha_{i,j})(h(\a_{i,j})),
$$
where $v=(v_{i,j})_{i,j}\in\prod_{i=1}^s\prod_{j=1}^{p_i}\frak{S}_n$ is of cycle-type $\lambdahat$.
\label{avantlast}\end{theorem}

\begin{proof}We need to prove the analogue of Formula (\ref{formula0}) but the proof is completely similar (see also \cite[\S 6.10.5 and Theorem 6.10.1]{ICQV} in the case $\calE=\calO(a)^n$ and $w=1$).\end{proof}

As for Harish-Chandra induction (see \S \ref{HCILie}) there exists a Lie algebra version of Deligne-Lusztig characters (see \cite{letellier}) which we denote by $R_{\t_{v_{i,j}}(q^{d_{i,j}})}^{\gl_n(q^{d_{i,j}})}(\eta_{i,j})$ and the analogue of Theorem \ref{charsum} hold.

On the other hand we have the commutation formula (see \cite[Remark 6.2.15]{letellier})

$$
\calF^{\gl_n(q^d)}\circ R_{\t_v(q^d)}^{\gl_n(q^d)}=\varepsilon(v) q^{{\rm dim}\, U}R_{\t_v(q^d)}^{\gl_n(q^d)},
$$
where $U$ is the group upper triangular unipotent matrices. We deduce

\begin{equation}
\calF^{\gl_n(q^d)}(1_{S_v})=\varepsilon(v)q^{\frac{1}{2}{\rm dim}\, S_v(\overline{\F}_q)}R_{\t_v(q^d)}^{\gl_n(q^d)}(\eta_{S_v}),
\label{forKS}\end{equation}
where $S_v$ is the $\GL_n(q^d)$-orbits of a semisimple regular element $\sigma_v\in \t_v(q^d)$ and $\eta_{S_v}:\t_v(q^d)\rightarrow\overline{\Q}_\ell$, $x\mapsto\psi\big(\Tr(x\sigma_v)\big)$.

\begin{remark}Formula  (\ref{forKS}) is more complicated than its analogue for Harish-Chandra induction (see Formula (\ref{Fouror})) and is originally due to the combination of the work of Kazhdan and Springer \cite{Kazhdan}\cite{springer}.
\end{remark}

If $\lambdahat=(\lambda_{i,j})_{i,j}\in \prod_{i=1}^s\prod_{j=1}^{p_i}\frak{S}_n$ and $v_{i,j}\in\frak{S}_n$ is of cycle-type $\lambda_{i,j}$, put

$$
\varepsilon_\lambdahat:=\prod_{i,j}\varepsilon(v_{i,j}).
$$
We have $\varepsilon_\lambdahat=(-1)^{r(\lambdahat)}$ where

$$
r(\lambdahat)=rn+\sum_{i=1}^s\sum_{j=1}^{p_i}\ell(\lambda_{i,j}),
$$
with $\ell(\lambda_{i,j})$ the length of the partition $\lambda_{i,j}$.

Following the proof of Theorem \ref{maintheoHiggs} we finally deduce that

\begin{theorem}

$$
\left\langle {\rm Coeff}_{Y^\m}\left[(q-1)\Log\,\Omega_{D_w}(q)\right],p_\lambdahat\right\rangle=\varepsilon_\lambdahat q^{-\frac{1}{2}d_\bfS}\frac{|X_{\bfS_\lambdahat, D_w}^\calE|}{|{\rm PAut}(\calE)|},
$$
where
$$
X_{\bfS_\lambdahat,D_w}^\calE:=  \left\{\varphi:\calE\rightarrow\calE\otimes\Omega^1(D_w)\,\left|\, {\rm Res}_{\a_{i,j}}(\varphi)\in S_{v_{i,j}}\text{ for all }i=1,\dots,s, j=1,\dots,p_i\right.\right\}.
$$
\label{lasttheo}\end{theorem}

\begin{proof} By Theorem \ref{avantlast} we need to prove the following formula

$$
\varepsilon_\lambdahat q^{-\frac{1}{2}d_\bfS}\frac{|X_{\bfS_\lambdahat, D_w}^\calE|}{|{\rm PAut}(\calE)|}=\frac{1}{|{\rm Aut}(\calE)|}\sum_{h\in{\rm Aut}(\calE)}\prod_{i=1}^s\prod_{j=1}^{p_i}R_{T_{v_{i,j}}(q^{d_{i,j}})}^{\GL_n(q^{d_{i,j}})}(\alpha_{i,j})(h(\a_{i,j})).
$$
Using the analogue of Theorem \ref{charsum} together with Formula (\ref{forKS}) we are reduced to prove
the formula 
$$
|X_{\bfS_\lambdahat,D_w}^\calE|=q^{-n^2}\sum_{h\in{\rm End}(\calE)}\prod_{i=1}^s\prod_{j=1}^{p_i}\calF^{\gl_n(q^{d_{i,j}})}(1_{S_{v_{i,j}}})(h(\a_{i,j})),
$$
whose proof is completely similar to that of Proposition \ref{eqHiggs}.

\end{proof}

In the case where $\calE=\calO(a)^n$, we understand the action of $\frak{H}_\nu=(\mathfrak{S})^l\rtimes\frak{S}_l$ on the cohomology $H_c^i(\calX^\calE_\bfS,\C)$ and we can show that Conjecture \ref{conjB} (and so Conjecture \ref{conjgeneral} by the discussion in \S \ref{strategy}) is true from the above theorem, namely we can show that

$$
\varepsilon_\lambdahat \frac{|X_{\bfS_\lambdahat, D_w}^{\,\calO(a)^n}|}{|{\rm PAut}(\calE)|}=\sum_i {\rm Tr}\left((\v_\lambdahat,w)\,|\, \varepsilon^l\otimes H_c^{2i}(\calX^{\,\calO(a)^n}_\bfS,\C)\right)\, q^i
$$
for any $\v_\lambdahat=(\v_{i,j})_{i,j}\in\prod_{i=1}^s\prod_{j=1}^{p_i}\frak{S}_n^{d_{i,j}}$ such that $\overline{\v}_{i,j}$ is of cycle-type $\lambda_{i,j}$. 

We proved this formula in the proof of Theorem \ref{theo2.5} in the case where the coordinates of $\lambdahat$ are all equal to the trivial partition $(1^n)$. We also proved it in \cite[Theorem 2.6]{HLRV}\cite[\S 7.4]{ICQV} when $\lambdahat$ is arbitrary but $w=1$. The proof of this slightly more general formula is not more difficult and is  a combination of these two extreme cases.

\end{document}